\documentclass[11pt]{article}
\usepackage[margin=1.1 in]{geometry}
\usepackage[english]{babel}
\usepackage{lipsum}
\usepackage{cite}
\usepackage{amsfonts}
\usepackage{mathtools}
\usepackage{graphicx}
\usepackage{epstopdf}
\usepackage{enumitem}
\usepackage{wrapfig}
\usepackage{pifont}
\usepackage{mdframed}
\usepackage{amssymb}
\usepackage{booktabs} 
\usepackage{multirow} 
\usepackage[most]{tcolorbox}
\usepackage{subcaption}
\usepackage{url} 
\newtcolorbox{mybox}[3][]
{
  colframe = #2!25,
  colback  = #2!10,
  coltitle = #2!20!black,  
  title    = {#3},
  #1,
}

\usepackage{algorithm}
\usepackage{algpseudocode}
\usepackage{float}
\usepackage{epstopdf}
\usepackage{footnote}
\makesavenoteenv{tabular}
\makesavenoteenv{table}
\usepackage{bbm}
\usepackage{amsthm}
\usepackage[makeroom]{cancel}
\usepackage{amsmath,bm}
\usepackage{amssymb}
\usepackage{mathtools, cuted}
\usepackage{kantlipsum,setspace}
\ifpdf
  \DeclareGraphicsExtensions{.eps,.pdf,.png,.jpg}
\else
  \DeclareGraphicsExtensions{.eps}
\fi

\theoremstyle{plain}
\newtheorem{theorem}{Theorem}[section]
\newtheorem{lem}[theorem]{Lemma}
\newtheorem*{lem*}{Lemma}
\newtheorem{prop}[theorem]{Proposition}

\newtheorem*{cor*}{Corollary}

\theoremstyle{definition}
\newtheorem{defn}{Definition}[section]
\newtheorem*{defn*}{Definition}
\newtheorem{assump}{Assumption}

\theoremstyle{remark}
\newtheorem{rem}{Remark}
\newtheorem*{rem*}{Remark}

\DeclareMathOperator*{\argmin}{arg\,min}

\usepackage{xcolor,colortbl}

\definecolor{Gray}{gray}{0.85}
\definecolor{LightCyan}{rgb}{0.88,1,1}

\usepackage{hyperref}

\newcommand{\x}{\mathbf{x}}
\newcommand{\X}{\mathcal{X}}
\newcommand{\y}{\mathbf{y}}
\newcommand{\z}{\mathbf{z}}
\newcommand{\m}{\mathbf{m}}
\newcommand{\g}{\mathbf{g}}

\newcommand{\lb}{\boldsymbol{\lambda}}

\newcommand{\ab}{\mathbf{a}}
\newcommand{\bb}{\mathbf{b}}
\newcommand{\cb}{\mathbf{c}}
\newcommand{\db}{\mathbf{d}}
\newcommand{\vb}{\mathbf{v}}
\newcommand{\qb}{\mathbf{q}}
\newcommand{\xo}{\overline{\x}}
\newcommand{\yo}{\overline{\y}}

\newcommand{\gr}{\nabla}

\newcommand{\dl}{\mathbb{R}^{d_{\ell}}}
\newcommand{\du}{\mathbb{R}^{d_u}}
\newcommand{\dk}{\mathbb{R}^{k}}
\newcommand{\Ebb}{\mathbb{E}}

\newcommand{\overbar}[1]{\mkern 1.5mu\overline{\mkern-1.5mu#1\mkern-1.5mu}\mkern 1.5mu}

\newcommand{\Fob}{\overbar{F}}
\newcommand{\XCalOb}{\overbar{\mathcal{X}}}

\newcommand{\R}{\mathbb{R}}

\newcommand{\algo}{{DS-BLO}}

\newcommand{\RT}{\ding{226}~\!}

\newcommand{\1}{\ding{172}}
\newcommand{\2}{\ding{173}}
\newcommand{\3}{\ding{174}}
\newcommand{\4}{\ding{175}}

\title{\Large \bf A Doubly Stochastically Perturbed Algorithm for Linearly Constrained Bilevel Optimization}
 \date{}
\author{\large Prashant Khanduri$^{\dagger,\diamond}$, Ioannis Tsaknakis$^{\ddagger,\diamond}$, Yihua Zhang$^\ast$, \\Sijia Liu$^\ast$, and Mingyi Hong$^\ddagger$ \\[0.5 cm]
\small $^{\dagger}$Department of Computer Science,\\
	\small Wayne State University, MI\\
\small $^{\ddagger}$Department  of Electrical and Computer Engineering, \\
	\small University of Minnesota, MN \\
	\small $^{\ast}$Department of Computer Science,\\
	\small Michigan State University, MI \\
	\small \texttt{Email:} \texttt{khanduri.prashant@wayne.edu}, \texttt{tsakn001@umn.edu}, 
    \\\small \texttt{ \qquad \{zhan1908, liusiji5\}@msu.edu},
	 \texttt{mhong@umn.edu}}

\begin{document}

\maketitle
\def\thefootnote{$\diamond$}\footnotetext{These authors contributed equally to this work.}
\begin{abstract}
 In this work, we develop analysis and algorithms for a class of (stochastic) bilevel optimization problems whose lower-level (LL) problem is strongly convex and linearly constrained. Most existing approaches for solving such problems rely on unrealistic assumptions or penalty function-based approximate reformulations that are not necessarily equivalent to the original problem. In this work, we develop a stochastic algorithm based on an implicit gradient approach, suitable for data-intensive applications. It is well-known that for the class of problems of interest, the implicit function is nonsmooth. To circumvent this difficulty, we apply a smoothing technique that involves adding small random (linear) perturbations to the LL objective and then taking the expectation of the implicit objective over these perturbations. This approach gives rise to a novel stochastic formulation that ensures the differentiability of the implicit function and leads to the design of a novel and efficient doubly stochastic algorithm.  We show that the proposed algorithm converges to an $(\epsilon, \overline{\delta})$-Goldstein stationary point of the {stochastic objective in  $\widetilde{\mathcal{O}}(\epsilon^{-4} \overline{\delta}^{-1})$ iterations. Moreover, under certain additional assumptions, we establish the same convergence guarantee for the algorithm to achieve a $(3\epsilon, \overline{\delta} + \mathcal{O}(\epsilon))$-Goldstein stationary point of the original objective.} Finally, we perform experiments on adversarial training (AT) tasks to showcase the utility of the proposed algorithm. 
\end{abstract}


\section{Introduction}
\label{sec: Intro}
Bilevel optimization \cite{dempe2020bilevel} is an important class of problem with two levels of hierarchy, commonly referred to as  {\em upper-level (UL)} and the {\em lower-level (LL)}. 
This class of problems can be used to formulate many contemporary applications in machine learning such as  meta-learning \cite{aravind2019meta,franceschi_ICML_2018bilevel}, data hyper-cleaning \cite{shaban2019truncated},
hyperparameter optimization \cite{sinha2020gradient,franceschi_ICML_2018bilevel,franceschi2017forward,pedregosa2016hyperparameter},  
adversarial learning \cite{Li_PR_2019Learning, liu2021investigating,zhang2022revisiting}, as well as in other domains including network optimization \cite{migdalas1995bilevel}, economics \cite{cecchini2013solving}, and transport research \cite{Didi-Biha2006, kalashnikov2010comparison}.  
In this work, we focus on a special class of (stochastic) bilevel  problems expressed below, where the LL problem consists of a \textit{strongly convex} objective over a set of \textit{coupled linear inequality constraints}:
\begin{mybox}{gray}{\bf Bilevel optimization with coupled linear constraints}
\vspace{-6 mm}
\begin{subequations}\label{eq: Problem_Bilevel}
    \begin{align}
  &  \min_{\x \in \du} F (\x)  \coloneqq  f(\x , \y^\ast(\x)) \label{eq:bp_ul} \\
    &  \text{s.t.}  ~   \y^\ast(\x) =   \argmin_{\y \in \mathcal{Y}  \subset \dl} g(\x, \y) ~\text{with}~\mathcal{Y} \coloneqq \{\y :A \y + B \x \leq  \bb\}, \label{eq:bp_ll}
   \end{align}
\end{subequations}
\vspace{-5 mm}
\end{mybox}
\noindent
where $f: \du \times \dl \to \R$ is given by a stochastic function expressed as $f(\x, \y^\ast(\x)) \coloneqq  \Ebb_{\xi \sim \mathcal{D}_f} [\widetilde{f}(\x, \y^\ast(\x) ; \xi)]$ with $\xi \sim \mathcal{D}_f$ representing a sample of $f(\cdot, \cdot)$ drawn from $\mathcal{D}_f$; $g: \du \times \dl \to \R$ is a function
strongly convex in $\y$; $A \in \R^{k \times d_{\ell}}, B \in \R^{k \times \du}$ and $\bb \in \dk$ define the coupled linear constraints. 
Throughout the paper, we refer to \eqref{eq:bp_ul} as the UL, and \eqref{eq:bp_ll} as the LL problem, respectively, and $F(\x)$ is referred to as the {\it implicit function}. Note that although we do not explicitly define $g(\cdot)$ as a stochastic function,  our analysis and algorithm work for this case as well. The choice of writing $g(\cdot)$ as a deterministic function is made to simplify the presentation of the paper.
 
The gradient of $F(\x)$, referred to as the {\it implicit gradient}, is given by:
\begin{align*} 
{\nabla} F(\x) \coloneqq \nabla f(\x, \y^\ast(\x)) = \nabla_{x} f(\x,  {\y}^{\ast}(\x)) + [\nabla  {\y}^{\ast}(\x)]^{\top} \nabla_{y} f(\x, {\y}^{\ast}(\x)).
\end{align*}
Computing this gradient requires access to the (exact) LL optimal solution ${\y}^\ast(\x)$ and the differentiability of the map ${\y}^\ast(\x): \du \to \dl$. The former requirement can typically be relaxed to having an approximate solution $\widehat{\y}(\x) \approx {\y}^\ast(\x)$, which leads to an approximated implicit gradient
\cite{ghadimi2018approximation}. However, the latter requirement is often not easy to satisfy, as the solution mapping $\y^\ast(\x)$ may not be differentiable for every $x \in \du$. 

It is known that when the LL problem is strongly convex and {\it unconstrained}, then by the Implicit Function Theorem \cite{ghadimi2018approximation} we can derive an explicit expression for  $\nabla {\y}^\ast(\x)$ for every $x \in \du$. For this reason, the majority of recent works have focused on developing algorithms for unconstrained bilevel problems \cite{ghadimi2018approximation,hong2020two,khanduri2021near,chen2021singletimescale, zhang2024introduction}. 
However, when the LL problem is constrained, $  {\y}^\ast(\x)$ might not be differentiable. 

A few works have tackled the bilevel problem with specific types of LL constraints \cite{xu2023efficient,lu2024first,yao2024constrained,jiang2024primal, khanduri2023linearly, xiao2023alternating, kornowski2024first, tsaknakis2022implicit,tsaknakis2023implicit}. For example, the work in \cite{xiao2023alternating} solves bilevel problems with linear equality constraints, the works in \cite{khanduri2023linearly,tsaknakis2022implicit, kornowski2024first} address bilevel problems with linear inequality constraints, and the works \cite{yao2024constrained, xu2023efficient, tsaknakis2023implicit, lu2024first,jiang2024primal} tackle bilevel problems with general constraints in the LL problem. For linear equality-constrained problems, differentiability of the map ${\y}^\ast(\x)$ can be established under standard assumptions \cite[Lemma 2]{xiao2023alternating}; however, for linear inequality constrained problems, it may be non-differentiable. The work \cite{tsaknakis2022implicit} imposes restrictive assumptions (such as strict complementarity) that guarantee differentiability of $\y^\ast(\x)$ while the authors in \cite{khanduri2023linearly} establish almost sure differentiability of $\y^\ast(\x)$ by introducing a small linear perturbation in the LL problem. In contrast to \cite{tsaknakis2022implicit, khanduri2023linearly}, the authors in \cite{kornowski2024first} do not rely on ensuring differentiability of $\y^\ast(\x)$, they instead develop first-order (Hessian free) algorithms building upon the theory of non-smooth optimization \cite{zhang2020complexity, tian2022finite, davis2022gradient}. All these works aim to develop 
implicit gradient methods with finite-time (or asymptotic) guarantees for solving constrained bilevel problems. However, a major drawback of these works is that they either require restrictive assumptions \cite{tsaknakis2022implicit}, do not provide finite-time guarantees without imposing non-verifiable assumptions (like weak convexity) on the implicit function \cite{khanduri2023linearly}, yield dimension-dependent guarantees or require restrictive assumptions (precise access to dual variables and boundedness of dual variables) to get dimension-free guarantees \cite{kornowski2024first}, and only focus on deterministic algorithms \cite{kornowski2024first,khanduri2023linearly,tsaknakis2022implicit}. 

Finally, the algorithms developed in \cite{yao2024constrained, xu2023efficient, tsaknakis2023implicit, lu2024first, jiang2024primal} have also addressed deterministic bilevel problems with {\it non-linear} and {\it convex} constraints in the LL. The majority of these works including \cite{yao2024constrained, tsaknakis2023implicit, lu2024first,jiang2024primal} rely on penalty-based reformulations with the aim of computing the KKT stationary point. However, it is known that these KKT conditions are {\it not} necessary conditions for a local minimum of the original objective \eqref{eq: Problem_Bilevel}, for example, see \cite[Example A.1]{chen2024bilevel}. The work in \cite{xu2023efficient} develops a gradient approximation scheme to compute the descent direction with asymptotic performance guarantees for solving general constrained bilevel problems. Reference \cite{tsaknakis2023implicit} employs a log-barrier penalty for the LL problem and measures stationarity using the standard implicit gradient norm.  Please see Table \ref{tab: Comparison_Bilevel} for a comparison. \vspace{5 mm}

\begin{table}[t]
\scriptsize
\begin{center}
\caption{Comparison of existing algorithms for constrained LL problems. ``Stationarity'' refers to the condition $\|\nabla f\| \leq \epsilon$ for a differentiable objective $f$ and Clarke's stationarity of Definition \ref{def: Clarke} for a non-differentiable objective.  ``Goldstein'' refers to the $(\epsilon, \overline{\delta})$-Goldstein stationarity condition of Definition \ref{Def: Gold}. ``Moreau env." denotes the gradient of the Moreau envelope. The {\color{cyan}highlighted rows} show algorithms that tackle the same problem as our work. 
The abbreviations used are defined as, {\bf LE}: Linear equality constraints, {\bf LI}: Linear inequality constraints, {\bf GC}: General convex constraints
{\bf IG}: Implicit gradient-based approach, {\bf Penalty}: Penalty function-based approach, {\bf CONS}: Constraints, {\bf CONV}: Convergence. { The notation $\widetilde{\mathcal{O}}$ omits logarithmic factors in a typical big-$\widetilde{{O}}$ notation.}
}
\label{tab: Comparison_Bilevel}
\renewcommand{\arraystretch}{1.3}
\resizebox{\textwidth}{!}{\begin{tabular}{|c|c|c|c|c|c|}
  \hline
\bf ALGORITHM & \bf CONS & \bf SETTING &\bf APPROACH & \bf MEASURE  &  \bf CONV   \\ 
        \hline
         \hline
 AiPOD \cite{xiao2023alternating}  & 
LE   & Stochastic &  IG & Stationarity & $\mathcal{O}(\epsilon^{-4})$  \\ \hline 
 {[S]SIGD$^1$ \cite{khanduri2023linearly}} & LI & Stochastic & IG & Moreau env. & $\mathcal{O}(\epsilon^{-4})$  \\ \hline
\rowcolor{LightCyan}
{[D]SIGD \cite{khanduri2023linearly}} & LI  & Deterministic & IG & Stationarity & Asymptotic  \\ \hline
\rowcolor{LightCyan}
 {\cite[Algorithm 3]{kornowski2024first}} & LI & Deterministic & IG & Goldstein & $\mathcal{O}(d_u \epsilon^{-3} \overline{\delta}^{-1})$  \\ \hline
\rowcolor{LightCyan}
 { Perturbed Inexact GD \cite{kornowski2024first}} & LI & Deterministic & IG & Goldstein & $\mathcal{O}(\epsilon^{-4}\overline{\delta}^{-1})$  \\ \hline
{IGBA$^2$ \cite{tsaknakis2023implicit}} & GC + LE & Deterministic & IG & Stationarity & $\mathcal{O}(\epsilon^{-2})$  \\ \hline
{LV-HBA \cite{yao2024constrained}} & GC & Deterministic & Penalty & KKT & $\mathcal{O}(\epsilon^{-2})$  \\ \hline
{BLOCC \cite{jiang2024primal}} & GC & Deterministic & Penalty & KKT & $\mathcal{O}(\epsilon^{-5})$  \\ \hline
{GAM \cite{xu2023efficient}} & GC & Deterministic & IG & Stationarity & Asymptotic  \\ \hline
\rowcolor{LightCyan}
\algo (Ours)$^3$ & LI  & Stochastic & IG & Goldstein &  $\widetilde{\mathcal{O}}(\epsilon^{-4} \overline{\delta}^{-1})$
\\ 
\hline
\rowcolor{LightCyan}
 \algo (Ours)$^4$ & LI  & Stochastic & IG & Goldstein &  $\widetilde{\mathcal{O}}(\epsilon^{-4} \widehat{\delta}^{-1})$
\\ 
  \hline  
  \end{tabular}}
  \end{center}
  \begin{flushleft}
  {\footnotesize $^1$Under the assumption that the implicit function $F(\x)$ is weakly convex.} \\
  {\footnotesize $^2$Utilizes log-barrier penalty for the linear inequality constraint to reformulate the problem to a bilevel problem with linear equality constraints in the LL and provides convergence for the reformulated problem.}\\
  {\footnotesize $^3$This convergence result refers to the reformulated problem \eqref{eq: Stochastic_Problem_Bilevel}.
    }\\
   {\footnotesize $^4$This convergence result refers to the convergence to an $(3 \epsilon,\widehat{\delta})$-Goldstein solution of the original problem \eqref{eq: Problem_Bilevel} under Assumption \ref{ass: additional} with $\widehat{\delta} =  \overline{\delta} + \epsilon/\widetilde{L}_{F_q}$ where $\widetilde{L}_{F_q}$ is defined in Lemma \ref{lem: grad_y_lip}. 
   }
\end{flushleft}
\end{table}

{\bf Contributions of this work.} In this work, we study the stochastic bilevel problem \eqref{eq: Problem_Bilevel}, where the LL task has a strongly convex objective and linear constraints that couple both the UL and the LL variables. This is a challenging problem, and due to its stochasticity and non-smoothness of the implicit function, so far we are not aware of any work that has developed satisfactory algorithms for it. Comparing to the guarantees developed in the above-referred works, we design algorithms with the following characteristics: \1 Dimension-independent and non-asymptotic convergence guarantees unlike \cite{kornowski2024first} and \cite{xu2023efficient}, respectively; \2 Stationarity measures defined directly based on both the perturbed and the original implicit functions, unlike \cite{yao2024constrained, lu2024first, jiang2024primal}; \3 No restrictive assumptions such as strict complementarity \cite{tsaknakis2022implicit}, weak convexity on the implicit objective \cite{khanduri2023linearly} or having access and boundedness of the dual parameters to guarantee convergence \cite{kornowski2024first}, and \4 Stochastic objective function in contrast to majority of works discussed above that only work under the deterministic setting \cite{kornowski2024first,khanduri2023linearly,tsaknakis2022implicit,yao2024constrained, xu2023efficient, tsaknakis2023implicit, lu2024first, jiang2024primal}.
Specifically, the contributions of this work are:  

$\bullet$ We propose a random perturbation-based smoothing technique that makes the implicit objective $F(\x)$ almost surely differentiable. Then, taking the expectation of $F(\x)$ over these perturbations leads to a {novel stochastic formulation that guarantees the differentiability of the (stochastic) implicit function}. Additionally, a closed-form expression is obtained for the stochastic implicit gradient that facilitates the development of gradient-based algorithms. 

$\bullet$ 
The stochastic problem is challenging since the smoothed stochastic implicit objective does not have Lipschitz continuous gradients. As a result, conventional gradient-based algorithms may fail. To address this, we propose \algo, a {doubly stochastic method for solving linearly constrained stochastic bilevel optimization problems} that utilizes two perturbations, one for tackling non-differentiability and the other for non-Lipschitz smoothness of the implicit objective. We establish finite-time convergence guarantees provided by \algo~and show that it converges to an $(\epsilon, \overline{\delta})$-stationary solution of the perturbed problem in $\widetilde{\mathcal{O}}(\epsilon^{-4} \overline{\delta}^{-1})$ iterations. {Importantly, under an additional set of assumptions we establish the same convergence guarantee to achieve an $(3 \epsilon,\widehat{\delta})$-Goldstein stationary solution of the original problem \eqref{eq: Problem_Bilevel} with $\widehat{\delta} =  \overline{\delta} + \epsilon/\widetilde{L}_{F_q}$ where $\widetilde{L}_{F_q}$ is defined in Lemma \ref{lem: grad_y_lip}.} To our knowledge, this is the first algorithm that guarantees (finite-time) convergence for stochastic bilevel problems with LL constraints without directly imposing any assumptions on the implicit function. 

$\bullet$ Finally, we assess the effectiveness of the proposed method through experiments on an adversarial learning problem. Our experiments show that \algo~ demonstrates improved performance against existing bilevel approaches while providing competitive performance against popular adversarial learning baselines even though \algo~is not customized for adversarial learning problems. 

{\bf Comparison with the conference version \cite{khanduri2023linearly}.} 
This work significantly extends and improves upon the formulation and algorithms presented in its conference version \cite{khanduri2023linearly}, as explained below. \1 In \cite{khanduri2023linearly}, a perturbed {\em deterministic} problem is tackled, whereas this work proposes a more general {\em stochastic} formulation where the expectation over the random perturbations (and stochastic samples) is computed, making the objective function differentiable. \2 The formulation in \cite{khanduri2023linearly} utilized a {\it single} perturbation throughout the execution of the algorithm which required restrictive assumptions on the problem such as (weak) convexity to develop algorithms with finite-time convergence guarantees. In contrast, this work treats the perturbation as {\it sampled} from certain distributions, allowing us to develop algorithms with finite-time convergence guarantees {\it without} imposing restrictive assumptions on the implicit function class. \3 In \cite{khanduri2023linearly}, without restrictive assumptions, only {\it asymptotic} convergence is provided by {\it directly} leveraging existing Armijo line-search algorithm, while in this work, we develop a novel algorithm and conduct extensive analysis to establish {\it finite-time} guarantees. \4 This work rigorously establishes the relation between the stationary solution of the perturbed and that of the original problem, further justifying the utility of the proposed reformulation. 
Finally, to distinguish the contributions of this work compared to \cite{khanduri2023linearly}, we will clearly mention the results that we will utilize from \cite{khanduri2023linearly}. Moreover, all the technical proofs in the current paper are new. 
 
\section{Smoothing by stochastic perturbation}
\label{sec:stochastic_formulation}
Let us consider the following  {\em perturbation-based smoothing}, where for any fixed $\x \in \du$, the LL objective $g(\x,\y)$ is augmented by a linear perturbation $\qb^{\top}\y$, where $\qb$ is a random vector sampled from some continuous distribution $\mathcal{Q}$. We then define the following {\em stochastically perturbed bilevel optimization} problem as a surrogate to solve problem \eqref{eq: Problem_Bilevel}:
\begin{mybox}{gray}{\bf Stochastically perturbed bilevel problem with linear constraints}
\vspace{-5 mm}
\begin{subequations}\label{eq: Stochastic_Problem_Bilevel}
\begin{align}
 &  \min_{\x \in \du} \Fob(\x)  \coloneqq \Ebb_{\qb \sim \mathcal{Q}} [F_q(\x)] \coloneqq \Ebb_{\qb \sim \mathcal{Q}} [f(\x, \y_q^\ast(\x))] \label{eq:sbp_ul}\\
&   \text{s.t.}  ~  \y_q^\ast(\x) = \argmin_{\y \in \mathcal{Y}} g_q(\x, \y) ~\text{with}~\mathcal{Y} = \{\y :A \y +B\x \leq  \bb\}, \label{eq:sbp_ll}
\end{align}
\end{subequations}
\end{mybox}
\noindent where $g_q(\x,\y) \coloneqq g(\x, \y) + \qb^\top \y$ with perturbation $\qb \sim \mathcal{Q}$ chosen from a continuous distribution $\mathcal{Q}$. To justify the reformulation in \eqref{eq: Stochastic_Problem_Bilevel}, we first analyze the approximation error incurred in solving \eqref{eq: Stochastic_Problem_Bilevel} compared to \eqref{eq: Problem_Bilevel}. Later, in Section \ref{sec: pre} we discuss how perturbing the LL problem ensures (local) differentiability of the stochastic implicit objective, $F_q(\cdot)$, and thereby the implicit objective $\Fob(\cdot)$. We first introduce some basic assumptions. 
\begin{assump}\label{ass:basics}
The following conditions hold for problems \eqref{eq: Problem_Bilevel} and \eqref{eq: Stochastic_Problem_Bilevel}:
\begin{enumerate}[label=(\alph*)]
    \item \label{ass:diff}  $f(\x,\y)$ is once and $g(\x,\y)$ is twice continuously differentiable.   
    \item \label{ass:sets} $\mathcal{Y}:=\big\{ \y \in \dl ~\big| ~A\y  +B \x \leq \bb \big\}$ is a compact set , $\forall \x \in \du$. 
    \item \label{ass:str_cvx_h}  $g(\x,\y)$ is $\mu_{g}$-strongly convex in $\y$, $\forall \x \in \du$.   
\end{enumerate}
\end{assump}
Assumption \ref{ass:basics} is standard in bilevel optimization \cite{kornowski2024first, khanduri2023linearly,hong2020two}. Specifically, Assumption \ref{ass:basics}\ref{ass:diff} ensures that the (partial) gradients of the UL objective $f(\cdot, \cdot)$ and the (partial) gradients and Hessians (Jacobians) of the LL objective $g(\cdot, \cdot)$ exist. Assumptions  \ref{ass:basics}\ref{ass:sets}, \ref{ass:str_cvx_h} ensure the continuity of the implicit objective, $F(\x)$, and uniqueness of the LL problem's solution, respectively. This implies that the implicit function $F(\x)$ (and $\Fob(\x)$) is well defined. 
\begin{assump}\label{ass:Fn_UL_LL}
The following hold for problem \eqref{eq: Problem_Bilevel},  $\forall~ \x,\xo \in \du, \y,\yo \in \dl$:  
\begin{enumerate}[label=(\alph*)]
    \item\label{ass:grad_f_bound} $f$ has bounded gradients, i.e., 
    $\| \nabla f(\x,\y) \| \leq \overline{L}_{f}$. 
    \item \label{ass:grad_f_lip} 
    $f$ has Lipschitz gradients, i.e., $\|\nabla f(\x,\y) - \nabla f(\xo,\yo)\| \leq L_{f} \|[\x; \y]-[\xo; \yo]\|$.  
    \item  \label{ass:g_lip_grad} $g$ has Lipschitz gradient, i.e., $\|\nabla_{y} g(\x,\y) - \nabla_{y} g(\xo,\yo)\| \leq L_{g} \|[\x;\y]-[\xo;\yo]\|$. 
    \item\label{ass:Hes_gyy} $g$ has Lipschitz Hessian in $\y$, i.e., $\|\nabla_{yy}^2 g(\x,\y) - \nabla_{yy}^2 g(\xo,\yo)\| \leq L_{g_{yy}} \|[\x; \y]-[\xo; \yo]\|$.   
    \item\label{ass:Hes_lip_gxy} $g$ has Lipschitz Jacobian, i.e.,$\|\nabla_{xy}^2 g(\x,\y) - \nabla_{xy}^2 g(\xo,\yo)\| \leq L_{g_{xy}} \|[\x; \y]-[\xo; \yo]\|$.  
    \item\label{ass:Hes_bound_gxy} $g$ has a bounded Jacobian, 
    $\|\nabla_{xy}^2 g(\x,\y)\| \leq \overline{L}_{g_{xy}}$. 
\end{enumerate}
\end{assump}
 Assumption \ref{ass:Fn_UL_LL} is standard in bilevel optimization literature \cite{ghadimi2018approximation,hong2020two,chen2021singletimescale,ji2021bilevel} and is utilized to derive some useful properties of the (approximate) implicit gradient (see, e.g., Lemma \ref{lem:F} in Appendix \ref{app: lem:F}). Note that Assumptions \ref{ass:basics}\ref{ass:diff},\ref{ass:str_cvx_h} and \ref{ass:Fn_UL_LL} also hold for the perturbed LL objective $g_q(\cdot)$ in problem \eqref{eq: Stochastic_Problem_Bilevel}. Next, we compute an upper bound on the approximation error between the objectives of problems \eqref{eq: Problem_Bilevel} and \eqref{eq: Stochastic_Problem_Bilevel}.

 \begin{prop}
\label{prop:perturbed_prob}
Under Assumptions \ref{ass:basics} and \ref{ass:Fn_UL_LL}, we have:
$$| \Fob(\x) - F(\x) | \leq \overline{L}_{f} \cdot \sup_{\qb \sim \mathcal{Q}} \frac{\|\qb\|}{\mu_{g}},~\forall~\x \in \du,$$
where $F(\x)$ and $\Fob(\x)$ are defined in \eqref{eq: Problem_Bilevel} and \eqref{eq: Stochastic_Problem_Bilevel}, respectively. 
\end{prop}
\begin{proof}
    The proof follows directly from \cite[Prop. 3.2]{khanduri2023linearly}.
\end{proof}

Intuitively, the size of perturbation $\sup_{\qb \in \mathcal{Q}} \|\qb \|$ determines the difference between the original and the perturbed implicit functions. We note that the only requirement on $\qb$ is that it is generated from some continuous measure. Therefore we can always choose a distribution with small support such that $\sup_{\qb \in \mathcal{Q}}  \|\qb\|$ is arbitrarily small without essentially changing the landscape of the original problem \eqref{eq: Problem_Bilevel} \cite[pg. 5]{lu2020finding}. Next, we discuss how the objective in \eqref{eq: Stochastic_Problem_Bilevel} is differentiable.  

\subsection{The non-smoothness of problem \eqref{eq: Problem_Bilevel}}
It can be easily shown that under Assumptions \ref{ass:basics} and \ref{ass:Fn_UL_LL}, the LL solution map ${\y}^\ast(\x)$, and thus the implicit function $F(\x)$ of problem \eqref{eq: Problem_Bilevel} are both continuous (see \cite[Proposition 2.2]{khanduri2023linearly}), but not differentiable in general  \cite[Section 2.1]{khanduri2023linearly}. This implies that the implicit function, $F(\x)$, for problem \eqref{eq: Problem_Bilevel} is generally non-smooth. Therefore, the standard notion of stationarity (cf. Definition \ref{def: std_stationary}) is not suitable for problem \eqref{eq: Problem_Bilevel}. In the following, we introduce the notions of stationarity utilized in (non)smooth optimization literature and existing efforts that design algorithms to compute them in the context of bilevel optimization. 

First, we introduce the standard notion of stationarity.  
\begin{defn}
\label{def: std_stationary}
A point $\x$ is an $\epsilon$-stationary point of a differentiable function $\ell(\cdot)$ if it satisfies $\| \nabla \ell(\x) \| \leq \epsilon$.
\end{defn}
To minimize general Lipschitz functions that are not necessarily smooth, a widely used measure of stationarity is defined using Clarke's subdifferential. It is referred to as  Clarke's stationarity \cite{clarke1990optimization} defined below.
\begin{defn}
\label{def: Clarke}
A point $\x$ is an $\epsilon$-Clarke's stationary point of a Lipschitz function $\ell(\cdot)$ if $\text{\rm dist}(0, \partial \ell(\x)) \leq \epsilon$, where $\partial \ell(\x)$ is the Clarke subdifferential \cite{clarke1990optimization}: 
\begin{align*}
    \partial \ell(\x) \coloneqq \text{Conv}\Big\{ \lim_{i \to \infty} \nabla \ell(\x_i) : \exists \x_i \to \x, \nabla \ell(\x_i) ~\text{exists} \Big\},
\end{align*}
where $\text{Conv}(\mathcal{Z})$ is the convex hull of the set $\mathcal{Z}$.
\end{defn}
Clarke's subdifferential becomes the standard subdifferential for convex functions, while for differentiable functions $\partial \ell(\x) = \nabla \ell(\x)$. The authors in \cite{zhang2020complexity} established that achieving an $\epsilon$-Clarke's stationary point by (sub)gradient-based methods in finite time is impossible when minimizing general Lipschitz functions. {For minimizing Lipschitz functions satisfying weak convexity, the authors in \cite{davis_SIAM_2019stochastic} utilized Moreau envelope-based analysis and established that an $\epsilon$-Clarke's stationary point can be achieved in finite-time.} The ideas developed in \cite{davis_SIAM_2019stochastic} were utilized by [S]SIGD developed in \cite{khanduri2023linearly} for solving linearly constrained bilevel optimization problems. [S]SIGD provided finite time guarantees for linearly constrained bilevel optimization under the assumption that the implicit function is weakly convex. However, this is an {\it uncheckable} assumption, as verifying the weak convexity of the implicit function in bilevel optimization with LL constraints is generally not feasible.

Motivated by the need to develop (sub)gradient-based algorithms with finite time guarantees, recently, a series of works have adopted a slightly relaxed notion of stationarity referred to as $(\epsilon, \overline{\delta})$-Goldstein stationarity  \cite{goldstein1977optimization} which is defined as:
\begin{defn}
\label{Def: Gold}
A point $\x$ is an $(\epsilon, \overline{\delta})$-Goldstein's stationary point of a Lipschitz function $\ell(\cdot)$ if it satisfies: $\text{\rm dist}(0, \partial_{\overline{\delta}} \ell(\x)) \leq \epsilon$, where $\partial_{\overline{\delta}} \ell(\x)$ is the Goldstein's $\overline{\delta}$ subdifferential \cite{goldstein1977optimization} defined as ($\mathbb{B}_{\overline{\delta}}(\x)$ is the ball of radius $\overline{\delta}$ centered around $\x$):
\begin{align*}
\partial_{\overline{\delta}} \ell(\x) \coloneqq \text{Conv}\Big\{\bigcup_{\z \in \mathbb{B}_{\overline{\delta}}(\x)} \partial \ell(\z) \Big\}.
\end{align*}
\end{defn}
For minimizing general non-smooth Lipschitz functions, initial attempts to develop (sub) gradient-based methods to obtain $(\epsilon, \overline{\delta})$-Goldstein stationary points relied on gradient sampling methods \cite{burke2020gradient}, which suffered from dimension dependence, i.e., the developed convergence guarantees relied on the problem dimension. Recently, a series of works, including \cite{zhang2020complexity, davis2022gradient, tian2022finite}, have developed perturbation-based algorithms for minimizing {single-level non-smooth Lipschitz functions} and established finite-time dimension-free guarantees to achieve an $(\epsilon, \overline{\delta})$-Goldstein stationary point. {Very recently, the authors in \cite{jordan2023deterministic} considered the problem of minimizing non-smooth Lipschitz functions and established that this additional perturbation is necessary to obtain finite-time dimension-free convergence guarantees to reach an $(\epsilon, \overline{\delta})$-Goldstein stationary point.}

{First-order methods for solving linearly constrained bilevel problems of the form \eqref{eq: Problem_Bilevel} have recently been developed in \cite{kornowski2024first}, building upon ideas from \cite{zhang2020complexity, davis2022gradient, tian2022finite}.} However, since \cite{kornowski2024first} utilizes the first-order gradients to estimate the Hessian (inverses), the derived convergence guarantees either depend on the problem dimension or assume precise access and boundedness of dual parameters of the LL problem to achieve an $(\epsilon, \overline{\delta})$-Goldstein stationary point. We note that ensuring the boundedness of the optimal dual variables and the access to them is not feasible in practice. Please see Table \ref{tab:table1} for a comparison. To our knowledge, developing algorithms with dimension-free finite-time guarantees for solving problem \eqref{eq: Problem_Bilevel} is still an open problem.

{The above discussion serves as a motivation to develop algorithms for solving bilevel problems of the form \eqref{eq: Problem_Bilevel} that guarantee dimension-independent finite-time convergence guarantees and does not impose restrictive assumptions on the implicit function or the dual variables. Naturally, this leads to the development of a {\em ``doubly"} perturbation-based algorithm where one perturbation is utilized for ensuring dimension-independent guarantees \cite{jordan2023deterministic} while the second perturbation (cf. formulation \eqref{eq: Stochastic_Problem_Bilevel}) is used to guarantee differentiablity of the implicit objective circumventing the need for restrictive assumptions on the objective or the dual variables.}

\begin{table}[t]
\scriptsize
\begin{center}
\caption{Comparison of the convergence rates of existing works. ``Single-level'' refers to a standard minimization problem; ``LC Bilevel'' denotes a linearly constrained bilevel problem. The {\color{cyan}highlighted rows} represent the bilevel problems. For convergence measures, ``S'' refers to stationarity in Def. \ref{def: std_stationary}, ``G'' refers to Goldstein in Def. \ref{Def: Gold}, and ``M. env." denotes the gradient of the Moreau envelope.}
\label{tab:table1}
\renewcommand{\arraystretch}{1.3}
\resizebox{\textwidth}{!}{\begin{tabular}{|c|c|c|c|c|}
  \hline
\bf ALGORITHM & \bf PROBLEM & \bf SETTING &\bf MEASURE & \bf CONVERGENCE   \\ 
        \hline
         \hline
 INGD \cite{zhang2020complexity}, \cite{davis2022gradient} & 
Single level & Deterministic &  G &  $\widetilde{\mathcal{O}}(\epsilon^{-3}\overline{\delta}^{-1})$ \\ \hline 
 Stochastic INGD \cite{zhang2020complexity}, \cite{davis2022gradient} & 
Single level & Stochastic &  G &  $\widetilde{\mathcal{O}}(\epsilon^{-4}\overline{\delta}^{-1})$ \\ \hline 
 { Perturbed INGD \cite{tian2022finite}} &  
Single level & Deterministic & G & $\widetilde{\mathcal{O}}(\epsilon^{-3}\overline{\delta}^{-1})$   \\ \hline  
{Perturbed Stochastic INGD \cite{tian2022finite}} &  
Single level & Stochastic & G & $\widetilde{\mathcal{O}}(\epsilon^{-4}\overline{\delta}^{-1})$   \\ \hline 

{ Proximal Stochastic Subgradient$^{1}$ \cite{davis_SIAM_2019stochastic}} &  
Single level & Stochastic & M. env. & ${\mathcal{O}}(\epsilon^{-4})$   \\ \hline

\rowcolor{LightCyan} {[D]SIGD \cite{khanduri2023linearly}} & LC Bilevel & Deterministic & S &  Asymptotic  \\ \hline
\rowcolor{LightCyan} {[S]SIGD$^{2}$ \cite{khanduri2023linearly}} & LC Bilevel & Stochastic & S &  $\mathcal{O}(\epsilon^{-4})$  \\ \hline
\rowcolor{LightCyan} {\cite[Algorithm 3]{kornowski2024first}} & LC Bilevel & Deterministic & G &  $\mathcal{O} (d_u \epsilon^{-3}\overline{\delta}^{-1} )$  \\ \hline
\rowcolor{LightCyan} { Perturbed Inexact GD$^3$ \cite{kornowski2024first}} & LC Bilevel & Deterministic & G &  $\mathcal{O} (\epsilon^{-4}\overline{\delta}^{-1} )$  \\ \hline
\rowcolor{LightCyan} \algo (Ours)$^{4}$ &
 LC Bilevel & Stochastic &  G  &  $\widetilde{\mathcal{O}}(\epsilon^{-4 }\overline{\delta}^{-1})$\\ 
  \hline  
  \end{tabular}}
  \end{center}
  \begin{flushleft}
 {\footnotesize $^1$Minimizes $F(x) + R(x)$ where $F(\cdot)$ is a non-smooth Lipchitz but weakly convex function and $R(\cdot)$ is a proximable mapping. }  \\
  {\footnotesize $^2$Under the assumption that the implicit function $F(\x)$ is weakly convex.}  \\
 {\footnotesize $^3$Guarantees hold under the precise access to certain dual variables.}\\
 {\footnotesize $^4$Same guarantees hold for the perturbed \eqref{eq: Stochastic_Problem_Bilevel} and the original problem \eqref{eq: Problem_Bilevel} with $\overline{\delta}$ defined as $\overline{\delta} + \epsilon/\widetilde{L}_{F_q}$ for the original problem.}
\end{flushleft}
\end{table}

\subsection{Differentiability of \eqref{eq: Stochastic_Problem_Bilevel}}
\label{sec: pre}
Next, we will show that the perturbation $\qb^\top \y$ introduced in \eqref{eq: Stochastic_Problem_Bilevel} ensures that at a given $\x \in \du$, the strict complementarity (SC) holds for the LL problem with probability $1$ (w.p. $1$).  The SC condition will further imply the differentiability of the perturbed objective $\Fob(\x)$ (w.p.1) \cite{Friesz_Foundations_2016}.

To proceed, let us define some notations. For a given $\x \in \du$, define  $\overline{A}(\y)$ as the matrix that contains the rows $S(\y) \subseteq \{1,\ldots,k\}$ of $A$ that correspond to the  active constraints of inequality $A\y  \leq \bb -B\x $, and define $\overline{\bb}(\y), \overline{B}(\y)$ similarly. That is,  $\overline{A}(\y)\y = \overline{\bb}(\y)- \overline{B}(\y)\x$.
Also, let $\overline{\lb}_q^{\ast}(\x)$ denote the vector of Lagrange multipliers that corresponds to the active constraints at ${\y}_q^{\ast} (\x)$. We make the following assumption.
\begin{assump}\label{ass:basics_feas_rank}
The following holds for the LL constraints in \eqref{eq:sbp_ll}
    \begin{enumerate}[label=(\alph*)] 
        \item For every $\x \in \du$, there exists $\y \in \dl$ such that $A\y +B\x < \bb$.\label{ass:feas}
    \item \label{ass:rank} The constraint matrix $\overline{A}({\y}_q^{\ast}(\x))$ is full row rank 
    for every $\x \in \du$, where ${\y}_q^{\ast}(\x)$ is defined in \eqref{eq:sbp_ll}.
    \end{enumerate}
\end{assump}
Assumption \ref{ass:basics_feas_rank}\ref{ass:feas} ensures strict feasibility of the LL problem, while Assumption \ref{ass:basics_feas_rank}\ref{ass:rank} implies the linear independence constraint qualification (LICQ) is satisfied \cite{bertsekas1998nonlinear}; LICQ is commonly used in practice to ensure that the KKT conditions are satisfied by the problems's optimal solutions \cite{kornowski2024first,khanduri2023linearly}. 
The following Lemma says that for the perturbed problem \eqref{eq: Stochastic_Problem_Bilevel}, the SC condition holds for the LL problem w. p. 1.  
\begin{lem}\cite[Proposition 1]{lu2020finding}
\label{lem: SNAP}
Consider the perturbed problem \eqref{eq: Stochastic_Problem_Bilevel}. For a given $\x \in \du$, if $\y_q^{\ast}(\x)$ is a KKT point of problem $\min_{\y \in \dl} \{ g_q(\x,\y) |  A \y +B \x \leq \bb \}$, $\qb \sim \mathcal{Q}$ is generated from a continuous measure, and $\overline{A}(\y_q^{\ast}(\x))$ is full row rank, then the SC condition holds at $\x$, w.p. $1$,  i.e., $\overline{A}(\y_q^{\ast}(\x))\y_q^{\ast}(\x)=\overline{\bb}(\y_q^{\ast}(\x)) -\overline{B}(\y_q^{\ast}(\x))\implies \overline{\lb}_q^\ast(\x) >0$ where $\overline{\lb}_q^\ast(\x)$ is the vector of Lagrange multipliers corresponding to $\y_q^\ast (\x)$.
\end{lem}

This SC condition implies differentiability of $\Fob(\x)$ 
(\cite[Theorem 2.22]{friesz2015foundations}), as will be shown next. 
\begin{prop}\label{pro:diff}
  Under Assumptions \ref{ass:basics} and \ref{ass:basics_feas_rank}  the implicit functions $\Fob(\x)$ and $\Fob(\x, \xi)$ defined in problem \eqref{eq: Stochastic_Problem_Bilevel} are differentiable for every $\x \in \du$ and it holds that 
  \begin{align*}
      &\nabla \Fob(\x) \coloneqq \nabla  \Ebb_{\qb \sim \mathcal{Q}} [F_q(\x)] =    \Ebb_{\qb \sim \mathcal{Q}} [\nabla F_q(\x)] \\
      &\nabla \Fob(\x, \xi) \coloneqq \nabla  \Ebb_{\qb \sim \mathcal{Q}} [F_q(\x, \xi)] =    \Ebb_{\qb \sim \mathcal{Q}} [\nabla F_q(\x, \xi)].
  \end{align*}
Further, the gradients can be computed as:
\begin{subequations}\label{eq:stochastic:closed-form}
\begin{align}
\label{eq: imp_grad_F}
  \nabla F_q(\x) & = \nabla_{x} f(\x, \y_q^{\ast}(\x)) + [\nabla \y_q^{\ast}(\x)]^{T} \nabla_{y} f(\x, \y_q^{\ast}(\x))\\
 {\nabla} F_q(\x; \xi) & = \nabla_{x} \widetilde{f}(\x, {\y}_q^\ast(\x); \xi) +  [{\nabla } \y_q^{\ast}(\x)]^\top \nabla_{y} \widetilde{f}(\x, {\y}_q^\ast(\x);\xi).\label{eq: SG_Exact}
 \end{align}
\end{subequations}
where the Jacobian matrix $\nabla \y_q^{\ast}(\x)$ is given by:
\begin{align}
\label{eq:grad_yast1}
  \nabla \y_q^{\ast}(\x) = \big[ \nabla_{yy}^2 g(\x,\y_q^{\ast}(\x)) \big]^{-1}   
    \big[-\nabla_{xy}^2 g(\x,\y_q^{\ast}(\x)) -\overline{A}^\top \nabla \overline{\lb}_q^{\ast}(\x)\big], 
\end{align}
\begin{align}\label{eq:grad_l}
     \nabla \overline{\lb}_q^{\ast}(\x) & =   -   \big[ \overline{A}\big[\nabla_{yy}^{2}g(\x,\y_q^{\ast}(\x)) \big]^{-1}\overline{A}^\top \big]^{-1}    \big[\overline{A}\big[\nabla_{yy}^2 g(\x,\y_q^{\ast}(\x))\big]^{-1} \nabla_{xy}^2 g(\x,\y_q^{\ast}(\x))- \overline{B}\big],  
\end{align}
where $\overline{A} \coloneqq \overline{A}(\y_q^{\ast}(\x))$, $\overline{B} \coloneqq \overline{B}(\y_q^{\ast}(\x))$, $\y_q^\ast(\x)$ and $\overline{\lb}_q^{\ast}(\x)$ are defined in Lemma \ref{lem: SNAP}.
\end{prop}
\begin{proof}
    The proof follows \cite[Lemma 2.4]{khanduri2023linearly}; see Appendix \ref{app:proofs_sec2} for completeness. 
\end{proof}
In the following, we will refer to $\nabla {F}_q(\x)$ and $\nabla {F}_q(\x;\xi)$ as the {\it stochastic} and {\it sampled} implicit gradient of the perturbed function $\Fob(\x)$, respectively.
It is worth noting that without the perturbation term in \eqref{eq:sbp_ll}, the SC property would need to be {\it assumed} to guarantee the differentiability of the implicit function \cite[Theorem 2.22]{friesz2015foundations}. 
In contrast, Lemma \ref{lem: SNAP} suggests that slightly modifying the problem guarantees the SC condition, and thus the differentiability. 

From a practical standpoint, the implicit gradient computation consists of two steps: \1 obtaining an (approximate) solution to the lower-level problem, and \2 evaluating the expressions in \eqref{eq:grad_yast1} and \eqref{eq:grad_l}. For \1, there are many efficient and standard methods for solving the LL problem since it is a strongly convex problem. However, the requirement of having the {\it exact} solution $\y^*_q(\x)$ is often unrealistic. We will discuss how to relax such a requirement in the next subsection. For \2, evaluating  \eqref{eq:grad_yast1} and \eqref{eq:grad_l} can be computationally intensive, but we can leverage existing techniques such as the Neumann series approximation to reduce computational complexity; see, e.g., \cite[Lemma 3.1]{ghadimi2018approximation}.

Finally, we note that it is possible to extend this work to the case where LL objective, $g(\x,\y)$, 
 is stochastic, i.e. $g(\x,\y):=\mathbb{E}_{\zeta\sim \mathcal{D}_g}[\widetilde{g}(\x,\y;\zeta)]$. Assuming that the stochastic gradients $\nabla \widetilde{g}(\x,\y;\zeta)$ are available, then the Hessian inverse can be stochastically estimated with arbitrarily diminishing bias utilizing the procedure in \cite[Algortihm 3]{Ghadimi_BSA_Arxiv_2018}, and the variance of these estimates can also be bounded in a straightforward manner \cite{Ghadimi_BSA_Arxiv_2018, hong2020two,khanduri2021momentumassisted}. However, these errors will make the subsequent proofs unnecessarily complicated, so we do not formally consider the variation where $g(\cdot)$ is stochastic.

\subsection{Approximate stochastic gradients and properties}
Next, let us discuss how to relax the requirement of the exact solution  $\y_q^{\ast}(\x)$ for computing the gradients  \eqref{eq: Stochastic_Problem_Bilevel}. Define the approximated versions of $\nabla F_q(\cdot)$ and $\nabla F_q(\cdot;\xi)$ as:
\begin{subequations}
    \begin{align}
    \label{eq: approx_imp_grad_F}
     \widehat{\nabla} F_q(\x) &:= \nabla_{x} f(\x, \widehat{\y}_q(\x)) + [\widehat{\nabla} \y_q^{\ast}(\x)]^{T} \nabla_{y} f(\x, \widehat{\y}_q(\x)),\\
\label{eq: SG_Implicit}
    \widehat{\nabla} F_q(\x; \xi) &:= \nabla_{x} \widetilde{f}(\x, \widehat{\y}_q(\x); \xi) +  [\widehat{\nabla } \y_q^{\ast}(\x)]^\top \nabla_{y} \widetilde{f}(\x, \widehat{\y}_q(\x);\xi),
\end{align}
\end{subequations}
where $\widehat{\y}_q(\x)$ is a suitable approximate solution for the perturbed LL problem \eqref{eq:sbp_ll};
$\widehat{\nabla} \y_q^{\ast}(\x)$ is defined by substituting the approximate lower-level solution $\widehat{\y}_q(\x)$ in place of the exact solution $\y_q^\ast(\x)$ in the expressions \eqref{eq:grad_yast1} and \eqref{eq:grad_l}. To ensure that \eqref{eq: approx_imp_grad_F} provides a meaningful approximation of the exact implicit gradient, we impose several assumptions on the quality of the estimate $\widehat{\y}_q(\x)$. 
\begin{assump}\label{ass:Approx_y}  
    The approximate solution of the (perturbed) LL problem   \eqref{eq:sbp_ll}, $\widehat{\y}_q(\x)$, satisfies the following $\forall \x \in \du$ and $\qb \sim \mathcal{Q}$
    \begin{enumerate}[label=(\alph*)] \label{ass:alg_cond}
    \item\label{ass:3error}  $\|\y_q^{\ast}(\x) - \widehat{\y}_q(\x) \| \leq \delta$ for some $\delta > 0$,  
    \item \label{ass:3feasible} $\widehat{\y}_q(\x)$ is a feasible point, i.e., $A\widehat{\y}_q(\x) +B \x \leq \bb$,   
    \item \label{ass:3active} It holds that $\overline{A}(\y_q^{\ast}(\x))=\overline{A}(\widehat{\y}_q(\x))$.  
    \end{enumerate}
\end{assump}
Assumptions \ref{ass:Approx_y}\ref{ass:3error}, \ref{ass:3feasible} can be easily satisfied by leveraging standard algorithms such as the projected gradient method to solve a strongly convex problem. It is important to note that Assumption \ref{ass:Approx_y}\ref{ass:3active} can be satisfied for a ``sufficiently accurate'' solution $\widehat{\y}_q(\x)$ produced by algorithms such as projected gradient; see the remark below. 

\begin{rem} From a well-known result \cite[Theorem 4.1]{calamai1987projected}, if $\widehat{\y}_q^{k}(\x) \in \mathcal{Y}$ is a sequence that converges to a non-degenerate (i.e., SC and Assumption \ref{ass:basics_feas_rank}\ref{ass:rank} holds) stationary point $\y_q^{\ast}(\x)$, then there exists an integer $k_{0}$ such that $\overline{A}(\y_q^{\ast}(\x))=\overline{A}(\widehat{\y}_q^{k}(\x))$, $\forall k>k_{0}$. Further, in some cases, an upper bound for $k_{0}$ can be obtained. For instance, consider the case of non-negative constraints $\y \geq {\bf 0}$. Assuming that $\nabla_{y_{i}} g(\x,\y_q^{\ast}(\x))>0, \forall i \in S(\y_q^{\ast}(\x))$ it can be shown that after $\frac{L_{g}}{\mu_{g}} \log \left( 2L_{g} \| \y_q^{0}-\y_q^{\ast}(\x)\| /\tau \right)$ iterations of projected gradient descent the active set of the approximate solution $\widehat{\y}_q(\x)$ is the same with the active set of the exact one $\y_q^{\ast}(\x)$ (see \cite[Corollary 1]{nutini2019active}), where $\tau = \min_{i \in S(\y_q^{\ast}(\x))} \nabla_{y_{i}} g(\x,\y_q^{\ast}(\x))$ and $\y_q^{0}$ is the initialization. A similar result can be obtained for the case with bound constraints $\ab \leq \y \leq  \bb$.
\end{rem}   
The next three lemmas analyze  the properties of $\widehat{\nabla} F_q (\x)$ and $\widehat{\nabla} F_q (\x; \xi)$. The proofs of these lemmas are provided in Appendix \ref{app: lem:F} -- \ref{sub:stoch_grad}.
\begin{lem}[{\bf Approximation bias}]
\label{lem:F}
\hspace{-0.1cm}Under Assumptions \ref{ass:basics}-\ref{ass:Approx_y} the following holds
\begin{align*}
 \|\widehat{\nabla} F_q(\x) - \nabla F_q(\x) \| \leq L_{F}\cdot \delta \triangleq \delta_\y,\quad\forall~\x \in \du .
 \end{align*}
Moreover, the following holds: 
$$
\|\nabla F_q(\x)\| \leq \overline{L}_{F}, ~~
\|\widehat{\nabla} F_q(\x)\| \leq \overline{L}_{F}  
~ ~\text{and}~~ \|\nabla \Fob (\x)\| \leq \overline{L}_{F}, \quad\forall~\x \in \du  $$
where $L_{F} := L_{f} + L_{\y^{\ast}} \overline{L}_{f} +L_{f}\overline{L}_{\y^{\ast}}$, and $\overline{L}_{F} := \left(1+\overline{L}_{\y^{\ast}}\right)\overline{L}_{f}$; the constants $\overline{L}_{\y^{\ast}}, L_{\y^{\ast}}$ are defined in Lemmas \ref{lem:lambda_y_bound}, \ref{lem:y_err}, respectively, provided in the Appendix.
\end{lem}

\begin{lem}[{\bf Variance from $\qb$}]
\label{lem:F_Variance}
\hspace{-0.1cm} Under Assumptions \ref{ass:basics}-\ref{ass:Approx_y}, the following holds
\begin{align*}
    \Ebb_{\qb \sim \mathcal{Q}} \|\nabla F_q(\x) - \nabla \Fob(\x)  \|^2 \leq \delta_{\vb}^2:=4\overline{L}_{F}^{2},\quad \quad\forall~\x \in \du.  
\end{align*} 
\end{lem}

We note that the above variance bound is independent of $\qb \sim \mathcal{Q}$, as bounding this variance with some constant is sufficient in the subsequent convergence analysis. It is possible to derive a bound for the variance that depends on $\sup_{\qb \in \mathcal{Q}} \|\qb\|$. However, as the analysis is lengthy and ultimately unnecessary, we choose not to include it.
Next, we characterize the variance incurred because of stochastic sampling of the UL objective. 
To this end, we introduce the following standard assumptions.
\begin{assump}\label{ass:SG}
For the UL objective $f(\x,\y)$ the following hold for some $\sigma_f > 0$:
    \begin{align*}
         \mathbb{E}_{\xi \sim \mathcal{D}_f}[ \nabla \widetilde{f}(\x, \y ; \xi)  ] = \nabla f(\x, \y) ,\quad 
         \mathbb{E}_{\xi \sim \mathcal{D}_f} \| \nabla \widetilde{f}(\x, \y ; \xi) - \nabla f(\x, \y) \|^2 = \sigma_f^2.
    \end{align*}
\end{assump}

\begin{lem}[{\bf Variance from stochastic sampling}]
\label{lem: SG_IG}
\hspace{-0.1cm}
Under Assumptions \ref{ass:basics}, \ref{ass:Fn_UL_LL}, \ref{ass:Approx_y} and \ref{ass:SG}, the stochastic gradient in \eqref{eq: SG_Exact} satisfies the following:
\begin{align*}
 \mathbb{E}_{\xi \sim \mathcal{D}_f}[\widehat{\nabla} F_q(\x; \xi)]  =  \widehat{\nabla} F_q(\x)~\text{and}~ ~ \mathbb{E}_{\qb \sim \mathcal{Q}, \xi \sim   \mathcal{D}_f}\| {\nabla} F_q(\x; \xi) - \nabla \Fob (\x) \|^2  \leq  \delta_\vb^2,
\end{align*}%
where $\delta_\vb^2 \coloneqq   \sigma_F^2 + \delta_\vb^2$; $\sigma_F^2 \coloneqq 2 \sigma_f^2 + 2 \overline{L}_{\y^\ast} \sigma_f^2$;  $\overline{L}_{\y^\ast}$ is defined in Lemma \ref{lem:lambda_y_bound}.
\end{lem}
Note that in Lemmas \ref{lem:F_Variance} and \ref{lem: SG_IG}, to simplify notation we have used the same $\delta_\vb$ to denote the variance of $F_q(\x)$ and $F_q(\x;\xi)$.

\section{A doubly stochastic algorithm for linearly constrained bilevel optimization (\algo)}
\label{sec:Algorithm}
In this section, we develop an algorithm to solve the stochastically perturbed problem \eqref{eq: Stochastic_Problem_Bilevel}. Note that even though perturbed problem \eqref{eq: Stochastic_Problem_Bilevel} is now differentiable, designing an efficient algorithm for it is still challenging since the implicit function $\Fob(\cdot)$ is not Lipschitz smooth. This can be observed from Proposition \ref{pro:diff} which shows that $\nabla \y_q^\ast(\x)$ is not a continuous function of $\x$ because the active set $\overline{A}$ is discontinuous in $\x$. However, we know from Lemma \ref{lem:F} that $\Fob(\cdot)$ is a Lipschitz continuous function. This implies that conventional gradient-based algorithms are no longer suitable for \eqref{eq: Stochastic_Problem_Bilevel}. 

To proceed, we develop a novel doubly stochastic algorithm \algo, which optimizes the {\it perturbed} problem \eqref{eq: Stochastic_Problem_Bilevel}. \algo~relies on two kinds of perturbations, one is the vector $\mathbf{q} \sim \mathcal{Q}$ introduced in Section \ref{sec:stochastic_formulation} to ensure the differentiability of the stochastic objective.
The second perturbation is necessary to achieve dimension-free convergence guarantees as recently established in \cite{jordan2023deterministic}. In contrast to \cite{kornowski2024first} where the authors directly tackle the non-smooth problem \eqref{eq: Problem_Bilevel}, the perturbed reformulation in \eqref{eq: Stochastic_Problem_Bilevel} is differentiable. Even though the function is differentiable, the function may still be non-Lipschitz smooth, therefore, we rely on the standard notion of $(\epsilon, \overline{\delta})$-Goldstein stationarity to provide the convergence guarantees.   

\begin{algorithm}[t]
\caption{Doubly Stochastic Algorithm for Bilevel Optimization (\algo)}
\begin{algorithmic}[1]
\State{\bf Input.}  $\x_1$, $T > K$, where $K =\frac{1}{\ln \frac{1}{\beta}} \ln \frac{32 (  \delta_\vb + 2 \overline{L}_F )  }{\epsilon}$, $\gamma_1 = \frac{K}{\overline{\delta}}$, $\gamma_2 = 4 \gamma_1  (  \delta_\vb + 2\overline{L}_F) 
$, LL solution accuracy $\delta_\y =   \min \Big\{ \frac{ \epsilon^2}{1280   (\delta_\vb + 2 \overline{L}_F)} , \frac{2  \epsilon}{3  }  , \overline{L}_F, L_F \cdot \delta \Big\}$, $\beta = 1 - \frac{\epsilon^2}{960 (  \delta_\vb^2 + 2 \overline{L}_F^2)}$
\State{Sample $\qb_1 \sim \mathcal{Q}$ and perturb the LL problem in \eqref{eq: Problem_Bilevel}};
 \State{Find an approximate solution $\widehat{\y}_{q_1}(\x_1)$ of \eqref{eq:sbp_ll} s.t. Assumption \ref{ass:alg_cond} is satisfied}
 \State{\vspace{-4 mm}\begin{align*}\text{Compute $\m_1 = \g_1$ with}
    \begin{cases}
     \textbf{Opt I}:   & \g_1 = \widehat{\nabla} F_{q_1}(\x_1) ~\text{using \eqref{eq: approx_imp_grad_F}}  \\
      \textbf{Opt II}:  & \g_1 = \widehat{\nabla} F_{q_1}(\x_1; \xi_1);~\text{using \eqref{eq: SG_Implicit}}
    \end{cases}
      \end{align*}}
\For{$t \in [T]$}
\State{Update $\x_{t + 1} = \x_t - \eta_t \m_t$ where $\eta_t = \frac{1}{\gamma_1 \|\m_t \| + \gamma_2 }$};
\State{Sample $\overline{\x}_{t+1} \sim \mathcal{U}[\x_t, \x_{t+1}]$};
\State{Sample $\qb_{t+1} \sim \mathcal{Q}$ independently from $\qb_{t}$ and perturb the LL problem in \eqref{eq: Problem_Bilevel}};
    \State{Find an approx. solution $\widehat{\y}_{q_{t+1}}(\overline{\x}_{t+1})$ of \eqref{eq:sbp_ll} s.t. Assumption \ref{ass:alg_cond} is satisfied}
    \State{\vspace{-4 mm}\begin{align*}\text{Compute}
    \begin{cases}
     \textbf{Opt I}:   & \g_{t + 1}  = \widehat{\nabla} F_{q_{t+1}}(\overline{\x}_{t+1})~\text{using \eqref{eq: approx_imp_grad_F}}\\
      \textbf{Opt II}:  & \g_{t + 1} = \widehat{\nabla} F_{q_{t+1}}(\overline{\x}_{t+1}; \xi_{t+1})~\text{using} ~\eqref{eq: SG_Implicit};
    \end{cases}
      \end{align*}}  
    \State{Update $\m_{t+1} = \beta \m_t + (1 - \beta) \g_{t+1}$}. 
\EndFor
\end{algorithmic}
\label{Algo: DS-BLO}
\end{algorithm}

We state the algorithm steps in Algorithm \ref{Algo: DS-BLO}. Specifically, the algorithm relies on two perturbations: \1 In Step 8 of Algorithm \ref{Algo: DS-BLO} for ensuring (almost sure) differentiability of the implicit function in problem \eqref{eq: Problem_Bilevel}, and \2 In Step 7 of Algorithm \ref{Algo: DS-BLO} where $\overline{\x}_{t+1}$ is chosen uniformly randomly from range $[\x_t, \x_{t+1}]$ to ensure sufficient (expected) descent in each iteration for non-Lipschitz smooth objectives. 
Compared to \algo, the algorithms developed in \cite{kornowski2024first} approximate the (implicit) gradients using \1 The zeroth order oracle where the gradient approximation quality degrades as the problem dimension increases, and \2 A first-order oracle which requires precise knowledge of the dual parameters of the LL problem. In contrast, \algo~avoids dimension dependence while not requiring any knowledge of the dual parameters of the LL problem (cf. Proposition \ref{pro:diff} for the gradient expression employed by \algo).

\subsection{Convergence guarantees}
\label{sec:convergence}
This section studies the convergence behavior of  \algo. Specifically, we aim to characterize the iterations required to achieve an $(\epsilon, \overline{\delta})$-Goldstein stationary point for some $\epsilon,  \overline{\delta} > 0$ as given in Definition \ref{Def: Gold}. First, we establish an intermediate result for \algo~ which will then be utilized to establish the convergence guarantees of \algo. 

\begin{lem}
\label{Thm:Main}
Under Assumptions \ref{ass:basics}-\ref{ass:SG} the iterates generated by Algorithm \ref{Algo: DS-BLO} for both {\bf Options I} and {\bf II} satisfy the following:
    \begin{align*}
  \frac{1}{T} \sum_{t = 1}^T   \Ebb \Big\|\!\!  \sum_{i = t - K + 1}^t  \hspace{-2mm} \alpha_i    {\gr}  \Fob(\overline{\x}_i)   \Big\| \!  \leq \!    \epsilon   ~\text{for} ~T = \widetilde{\mathcal{O}} \bigg\{ \frac{1}{\overline{\delta}  \cdot \epsilon^4} \max\big\{(\delta_\vb + 2 \overline{L}_F)  \Delta ,  (\delta_\vb^2 + 2\overline{L}_F^2)^2\big\} \!\! \bigg\},
    \end{align*}
where $\alpha_i := \frac{\beta^{t - i} (1 - \beta)}{1 - \beta^K}$ with $\beta$ defined in Algorithm \ref{Algo: DS-BLO} and $\Delta := \Fob(\x_1) - \Fob^\ast$ with $\Fob^\ast$ defined as $\Fob^\ast \coloneqq \inf_{\x \in \du} \Fob(\x)$. 
\end{lem}
The proof of the above result extends techniques used in \cite{zhang2020complexity,tian2022finite} where the authors establish that if the iterates $\overline{\x}_i$ for $i \in [t-K + 1, t]$ are sufficiently close to $\x_{t - K}$ then the convex combination of the (stochastic) gradients computed at iterations $i \in [t-K + 1, t]$ must belong to the Goldstein $ \overline{\delta}$ subdifferential set. However, our proof and the presented result have a few major differences compared to the ones in \cite{zhang2020complexity,tian2022finite} and other works. First, we must deal with the bias that arises in the implicit gradient estimates from approximate LL solutions. Second, our analysis needs to tackle the stochastic noises arising because of random perturbation $\qb \sim \mathcal{Q}$ and the randomization introduced in sampling the iterates $\overline{\x}_{t+1}$ in Step 7 of Algorithm \ref{Algo: DS-BLO}. Third, in our case, the objective function of \eqref{eq: Stochastic_Problem_Bilevel} is differentiable (albeit not Lipschitz smooth) and we have access to closed-form expressions to compute this (biased) gradient, unlike \cite{zhang2020complexity,tian2022finite, davis2022gradient, kornowski2024first,khanduri2023linearly} where the objective may be non-differentiable.

Note that the above result in Lemma \ref{Thm:Main} is non-standard in the sense that we have established that a convex combination of the implicit gradients is small. Next, we show that this implies $(\epsilon, \overline{\delta})$-Goldstationarity of the perturbed objective. 

\begin{theorem}[Convergence of \algo]
\label{Thm:Main_I}
Under Assumptions \ref{ass:basics}-\ref{ass:SG} the iterates generated by Algorithm \ref{Algo: DS-BLO} for both {\bf Options I} and {\bf II} satisfies 
\begin{align*}
 \frac{1}{T} \sum_{t = 1}^T  \Ebb [\text{dist}(0, \partial_{\overline{\delta}} \Fob(\x_{t - K}))]   \leq   \epsilon   ~\text{for}~T = \widetilde{\mathcal{O}} \bigg\{ \frac{1}{\overline{\delta} \cdot  \epsilon^4} \max\big\{(\delta_\vb + 2 \overline{L}_F)  \Delta ,  (\delta_\vb^2 + 2\overline{L}_F^2)^2\big\} \!\! \bigg\}
\end{align*}
\end{theorem}
\begin{proof}
To establish the stated result, we first show that $\x_{t-K}$ will be close to $\overline{\x}_i$ for $i \in [t- K +1, t]$ as below:
\begin{align*}
    \|\x_{t - K} - \overline{\x}_i \| & \overset{(a)}{\leq}  \|\x_{t - K} -  {\x}_{i-1}  \| + \|  \overline{\x}_i -  {\x}_{i-1} \| \\
    & \overset{(b)}{\leq}  \|\x_{t - K} -  {\x}_{i-1}  \| + \|   {\x}_i -  {\x}_{i-1} \| \\
    &\overset{(c)}{\leq} \sum_{j = t - K}^{i-1} \eta_j \| \m_j \| \overset{(d)}{\leq} \frac{K}{\gamma_1} = \overline{\delta},
\end{align*}
where $(a)$ follows from the triangle inequality; $(b)$ results from the fact that $\overline{\x}_i$ is sampled from the range $[\x_{i-1},\x_{i}]$; $(c)$ utilizes the update rule of Step 6 in \algo~and the triangle inequality; inequality $(d)$ uses the definition of $\eta_j = \frac{1}{\gamma_1 \| \m_j\| + \gamma_2}$ and the fact that the summation has less than $K$ terms, and the final equality results from the choice of $\gamma_1 =  {K}/{\overline{\delta}}$ from Algorithm \ref{Algo: DS-BLO}. 

Note that this implies that $\overline{\x}_i$ for $i \in [t- K +1, t]$ stays within a $\overline{\delta}$ neighborhood of $\x_{t - K}$. This implies that the convex combination $\sum_{i = t - K + 1}^t  \alpha_i    {\gr}  \Fob(\overline{\x}_i) \in \partial_{\overline{\delta}} \Fob(\x_{t - K})$. Therefore, we have
\begin{align*}
    \text{dist}(0, \partial_{\overline{\delta}} \Fob(\x_{t - K})) \leq \bigg\| \sum_{i = t - K + 1}^t  \alpha_i    {\gr}  \Fob(\overline{\x}_i) \bigg\|,
\end{align*}
Taking the expectation on both sizes, summing over $t \in [T]$ and multiplying both sides by $1/T$, we get
\begin{align*}
  \frac{1}{T}  \sum_{t = 1}^T \Ebb [\text{dist}(0, \partial_{\overline{\delta}} \Fob(\x_{t - K}))] \leq  \frac{1}{T}  \sum_{t = 1}^T \Ebb\bigg\| \sum_{i = t - K + 1}^t  \alpha_i    {\gr}  \Fob(\overline{\x}_i) \bigg\| \leq \epsilon,
\end{align*}
where the last inequality follows from Lemma \ref{Thm:Main} and with the choice of $T$ given in the statement of the Theorem. 

Therefore, we have the proof. 
\end{proof}

We make the following observations about the result of Theorem \ref{Thm:Main}. First, the proof utilizes the differentiability of the implicit function $\Fob(\cdot)$ and the definition of the $\overline{\delta}$-Goldstein subgradient (Definition \ref{Def: Gold}) to establish the following fact: the convex combination of (sub) gradients derived in Lemma \ref{Thm:Main} belongs to the $\overline{\delta}$-Goldstein subgradient set if $\x_i$'s are sufficiently close to $\x_{t-K}$ for $i \in [t-K +1, t]$.
This notion is similar to the one utilized in \cite{zhang2020complexity,tian2022finite} for single-level and \cite{kornowski2024first} for linearly constrained bilevel optimization where objective functions are non-differentiable. {We next present an important result that links the solution of the perturbed problem \eqref{eq: Stochastic_Problem_Bilevel} with the Goldstein stationary solution of the original problem \eqref{eq: Problem_Bilevel}. To do so, we first introduce some additional assumptions.
\begin{assump} 
\label{ass: additional}
We assume that the following hold
    \begin{enumerate}[label=(\alph*)]
    \item \label{ass: additional_LL}The LL constraint set $\mathcal{Y}$ is $\mathcal{Y} = \{\y \in \dl ~|~ A\y \leq \bb \}$ and the constraint matrix $A$ is full row rank. 
      \item  \label{ass: additional_g_lip} {$g$ is Lipschitz in $\x$, i.e., $\| g(\x, \y) - g(\overline{\x}, \y)\| \leq \overline{L}_g \|\x - \overline{\x} \|$.}
    \item \label{ass: Jacobian_LL}For any $\overline{\x} 
    \in \du$, $\y \in \dl$, $\overline{\delta} > 0$ and a uniformly randomly chosen $\mathbf{u} \in \du$ over a unit ball centered at $\mathbf{0}$, define $\boldsymbol{\vartheta}$ as: $\boldsymbol{\vartheta} \coloneqq \nabla_y g(\overline{\x} + \overline{\delta} \mathbf{u}, \y) - \nabla_y g(\overline{\x}, \y)$.
    We assume that $\boldsymbol{\vartheta}$ is a continuous random variable in $\dl$ space. 
    \end{enumerate}
\end{assump}
Assumption \ref{ass: additional}\ref{ass: additional_LL} is rather standard in bilevel optimization \cite{khanduri2023linearly,tsaknakis2022implicit}. Assumption \ref{ass: additional}\ref{ass: additional_g_lip} is relatively mild, and in fact it is similar to the set of assumptions introduced in Assumption \ref{ass:Fn_UL_LL}. Together, Assumptions \ref{ass: additional}\ref{ass: additional_LL}, \ref{ass: additional_g_lip} are used to establish Lipschitz continuity of $\y^\ast(\x)$ which further implies Lipschitz continuity of the implicit function $F(\x)$ (see Lemma \ref{lem: y_Lip}). Notably, the Lipschitzness of $F(\x)$ allows us to define the notion of $\overline{\delta}$-Goldstein subgradient for the original problem \eqref{eq: Problem_Bilevel}. We note that in \cite{kornowski2024first} the Lipschitz continuity of $\y^\ast(\x)$ is directly assumed.

Assumption \ref{ass: additional}\ref{ass: Jacobian_LL} is somewhat unconventional. Note that $\boldsymbol{\vartheta}$ is a continuous random variable since it is a continuous function (since $\nabla_y g(\x, \y)$ is continuous, see Assumption \ref{ass:Fn_UL_LL}\ref{ass:g_lip_grad}) of a continuous random variable, $\mathbf{u}$. The assumption implies that $\boldsymbol{\vartheta}$ has a continuous measure in $\dl$ space.  
For example, it is easy to verify that Assumption \ref{ass: additional}\ref{ass: Jacobian_LL} is satisfied by functions of the form $g(\x,\y)=g_1(\x) + \x^{\top}Q\y + g_2(\y)$ with $\boldsymbol{\vartheta} = \overline{\delta} \cdot Q^\top \mathbf{u}$ for full row rank $Q$; note that this example is similar to the ones we used in the numerical experiments for specific choices of $g_1(\cdot)$ and $g_2(\cdot)$ (see Section \ref{sec:numerical}). Similarly, for general objectives, one can establish that for $\overline{\delta} \to 0$ we have $\boldsymbol{\vartheta} \approx \overline{\delta} \cdot [\nabla_{xy}^2 g (\overline{\x},\y) \big]^\top \mathbf{u}$, therefore, Assumption \ref{ass: additional}\ref{ass: Jacobian_LL} is (asymptotically) satisfied if $\nabla_{xy}^2 g (\overline{\x},\y) $ is full row rank.
This assumption allows us to (i) establish the SC property for the original implicit function $F(\x)$ (see Lemma \ref{lem: SC}) which is further utilized to compute the implicit gradient of $\y^\ast(\x)$, and (ii) to bound the distance of the implicit gradient of $\y^\ast(\x)$ from the (stochastic) gradients of the stochastically perturbed problem \eqref{eq: Stochastic_Problem_Bilevel}. Overall, Assumption \ref{ass: additional}\ref{ass: Jacobian_LL} allow us to establish the connection between Goldstein stationary solutions of the perturbed bilevel problem \eqref{eq: Stochastic_Problem_Bilevel} and the original bilevel problem \eqref{eq: Problem_Bilevel}; see the next result with proof provided in Appendix \ref{app:equivalence}.

\begin{prop}\label{prop:equivalence}
There exists a continuous measure $\mathcal{Q}$ such that for $\qb \sim \mathcal{Q}$ with $\|\qb \| \leq \mathcal{O}(\epsilon)$ and $\qb \in U_q$, where $U_q$ is defined in Appendix \ref{app:equivalence}
and under Assumptions \ref{ass:basics}-\ref{ass: additional}, an $(\epsilon, \overline{\delta})$-Goldstein stationary point for the perturbed bilevel problem defined in \eqref{eq: Stochastic_Problem_Bilevel} is also a $(3\epsilon, \overline{\delta} + \epsilon/\widetilde{L}_{F_q})$-Goldstein stationary point for the original bilevel problem \eqref{eq: Problem_Bilevel}.
\end{prop}}
Next, we compare the presented results with those of the earlier works \cite{kornowski2024first} and \cite{khanduri2023linearly}. The guarantees presented in \cite{kornowski2024first} and \cite{khanduri2023linearly} only hold for deterministic bilevel problems, in contrast, \algo~ provides finite-time guarantees for stochastic bilevel problems. Observe from Table \ref{tab:table1} that the convergence guarantees of \cite[Algorithm 3]{kornowski2024first} depend on the problem dimension, in Perturbed Inexact GD \cite{kornowski2024first} they rely on the precise access to the dual variables, and in [D]SIGD they are asymptotic.  \algo~is free of all these limitations while being able to tackle stochastic bilevel problems. It is also worth noting that \algo~matches the performance of Perturbed Inexact GD \cite{kornowski2024first} while solving a more challenging stochastic problem. 
Moreover, \algo~also matches the performance of algorithms designed to optimize single-level non-Lipschitz smooth stochastic objectives \cite{zhang2020complexity, tian2022finite, davis2022gradient}.  

Please see Table \ref{tab:table1} for a detailed comparison. 

\section{Experiments} 
\label{sec: exp}
In this section, we evaluate the numerical performance of \algo~and compare it against the state-of-the-art. 
\subsection{Numerical examples} \label{sec:numerical} We first consider a bilevel problem of the form \eqref{eq: Problem_Bilevel}
\begin{align*}
    &f(\x,\y) = \|\x\|^{2} + 0.1\x^{T}Q_{1}\y + \|\y\|^{2} + \mathbf{e}_{1}^{T}\x + \mathbf{e}_{2}^{T}\y, \\
    &g(\x,\y)= \|\x\|^{2} + \x^{T}Q_{2}\y + \|\y\|^{2}
\end{align*}
The elements of the matrices $Q_{1}, Q_{2}$ are generated at random from a uniform distribution, while $\mathbf{e}_{1},\mathbf{e}_{2}$ denotes vectors of all $1$s (of appropriate dimensions). The elements of the constraints matrices $A,B,\bb$ are also selected from the random uniform distribution. We solve the above problem using three different algorithms: \1 [S]SIGD \cite{khanduri2023linearly}, \2 Perturbed Inexact GD (PIGD) \cite{kornowski2024first}, \3 \algo~(proposed algorithm). The parameters of these algorithms are selected using a hyperparameter optimization procedure, where the optimal stepsize is selected in the range $[10^{-3},10^{-1}]$. The number of outer iterations is set to $100$ for all algorithms, while for the solution of the LL problem, we utilize a convex optimization solver. 

In Figure \ref{fig:numerical}, we present the objective value $F(\x)=f(\x,\y^{\ast}(\x))$  as a function of time for different values of the dimensions $d_{u}, d_{l}$ and $k$. We observe that the [S]SIGD and the proposed DS-BLO algorithm converge faster than the PIGD. The difference in runtime can be attributed to the different ways the reformulated problem of each algorithm is solved. While [S]SIGD and DS-BLO compute the implicit gradient and perform a number of gradient steps, the PIGD solves a Lagrangian reformulated problem using a convex optimization solver (cvxpy), which can be slower in practice.
\begin{figure}[t]
    \centering
    \begin{subfigure}[b]{0.49\textwidth}
        \centering
        \includegraphics[width=0.9\linewidth]{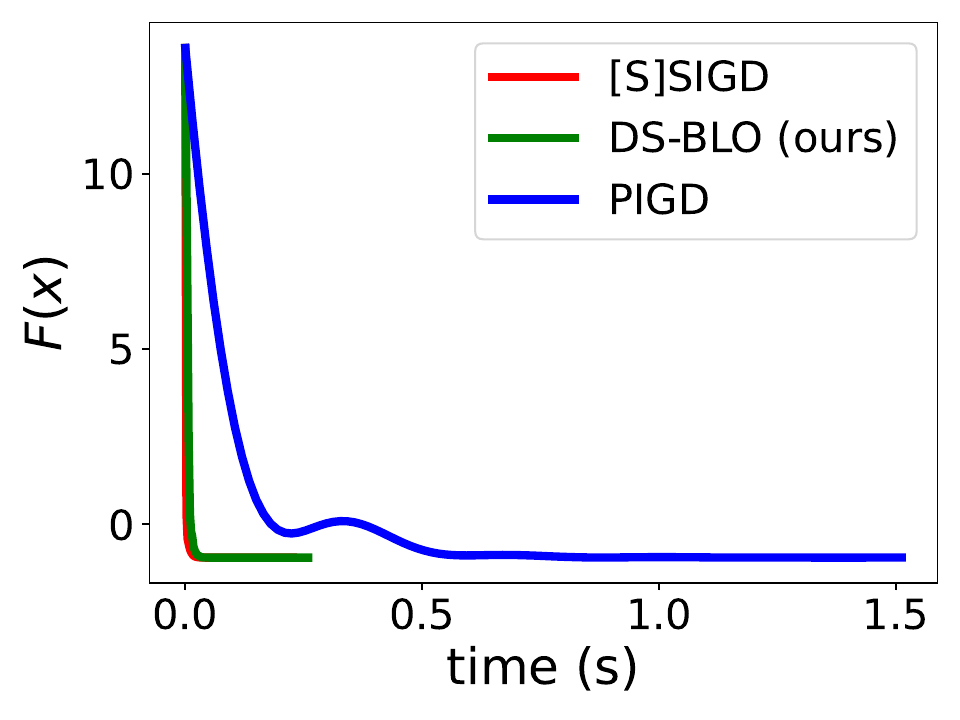}
        \caption{$d_{u}=d_{l}=10, k=5$} 
        \label{fig:numerical1}
    \end{subfigure}
    \hfill
    \begin{subfigure}[b]{0.49\textwidth}
        \centering
        \includegraphics[width=0.9\linewidth]{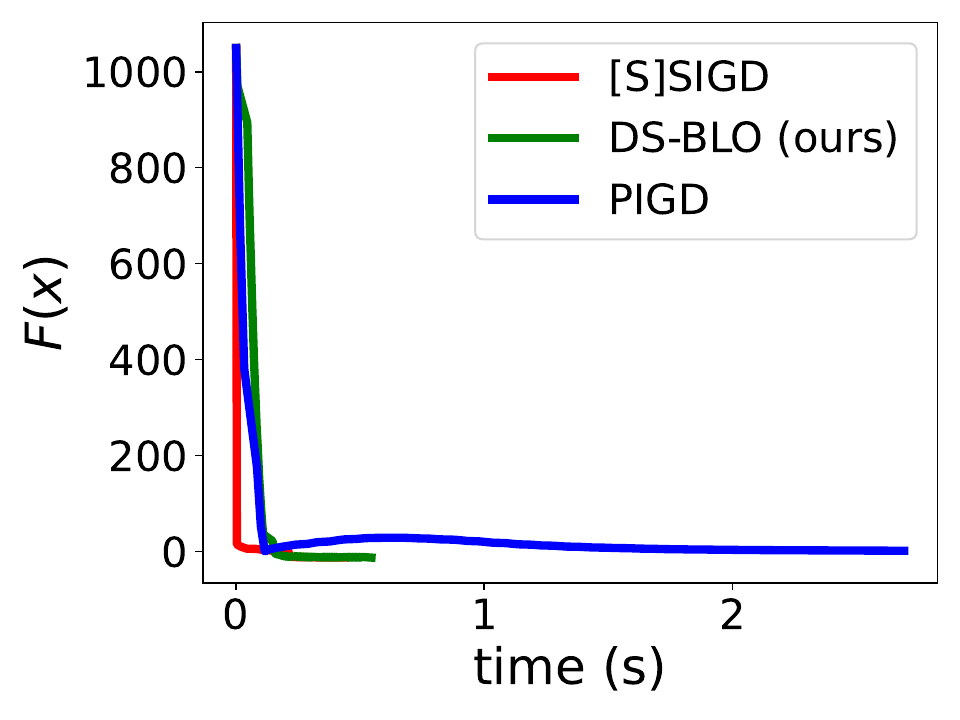}
        \caption{$d_{u}=d_{l}=50, k=10$} 
        \label{fig:numerical2}
    \end{subfigure}
    \caption{The evolution of the objective value $F(\x)=f(\x,\y^{\ast}(\x))$  as a function of time for different values of the dimensions $d_{u}, d_{l}$ and $k$ and three different algorithms.}
    \label{fig:numerical}
\end{figure}
\subsection{Adversarial Training}
\label{subsec: AT}
In this section, we evaluate the performance of \algo~via numerical experiments conducted on the widely accepted adversarial training problem \cite{madry2018towards, zhang2022revisiting} and compare the performance of \algo~with the recently proposed bilevel optimization methods, including [S]SIGD~\cite{khanduri2023linearly}, Perturbed Inexact GD \cite{kornowski2024first}, and other commonly used adversarial training methods, such as AT~\cite{madry2018towards} and TRADES~\cite{zhang2019theoretically}. Adversarial training aims to improve the robustness of deep neural networks against adversarial attacks, which are bounded by the $\epsilon$-tolerant $\ell_\infty$-norm attack constraint. This can be easily expressed as a bilevel problem with linear inequality constraints in the LL as follows: Consider the problem of robustly training a model $\phi(\x; \cb)$, where $\x$ denotes the model parameters and $\cb$ the input to the model; let's denote with $\{(\cb_{i},\db_{i}) \}_{i=1}^{N}$ the training set where $\cb_{i} \in \mathbb{R}^{d_{\ell_i}}, \db_{i} \in \R$  \cite{zhang2022revisiting, goodfellow2014explaining}. Then the adversarial training problem can be posed in the form of \eqref{eq: Problem_Bilevel} as 
\begin{align}
\label{eq:adversarial}
\displaystyle \min_{\x \in \du}  \sum_{i=1}^{N} f_i( \phi(\x;\cb_i +  {\y}_i^\ast(\x)) , \db_i)~\text{ s.t.} \; 
{\y}^\ast(\x)  
\in \begin{cases}
\displaystyle \argmin_{\y_i \in \mathbb{R}^{d_{\ell_i}}}  \sum_{i=1}^{N} g_i( \phi(\x;\cb_i + \y_i) , \db_i) \nonumber\\
\text{s.t.} \;\; 
-\epsilon \cdot \mathbf{1} \leq \y \leq \epsilon \cdot\mathbf{1} \nonumber
\end{cases},
\end{align}
where $\y = [\y_1^T, \ldots, \y_N^T]^T \in \dl$; $\y_i \in \mathbb{R}^{d_{\ell_i}}$ denotes the attack on the $i^{\text{th}}$ example and we have $\sum_{i = 1}^N   {d_{\ell_i}} = d_\ell$. Moreover, 
$f_i: \mathbb{R} \times \mathbb{R} \to \mathbb{R}$ denotes the loss function for learning the model parameter $\x$, while $g_i: \mathbb{R} \times \mathbb{R} \to \mathbb{R}$ denotes the adversarial objective used to design the optimal attack $\y$. Note that the linear constraints in the LL problem $-\epsilon \cdot\mathbf{1}  \leq \y \leq \epsilon \cdot\mathbf{1}$ model the $\epsilon$-attack budget.

Following prior work~\cite{zhang2022revisiting, khanduri2023linearly}, we consider two representative datasets CIFAR-10/100~\cite{krizhevsky2009learning} and adopt the ResNet-18~\cite{he2016deep} model. Regarding the perturbation, we studied two widely used~\cite{madry2018towards,wong2020fast,zhang2019theoretically,zhang2022revisiting} attack budget choices $\epsilon \in \{8/255, 16/255\}$. 
\begin{table}[t]
\centering
\caption{Performance overview of different methods on CIFAR-10 and CIFAR-100 \cite{krizhevsky2009learning} datasets with ResNet-18 \cite{he2016deep}. The result $a\pm b$ represents the mean value $a$ with a standard deviation of $b$ over $5$ random trials. Bilevel methods are denoted with $^\dagger$. ``PIGD" refers to the Perturbed Inexact GD \cite{kornowski2024first}. The results of [S]SIGD \cite{khanduri2023linearly} and DS-BLO are reported with different Gaussian variance ($\sigma^2$) values. ``PM" denotes the performance metrics. Performance is evaluated in terms of standard accuracy (SA) and robust accuracy (RA).}
\label{tab: at_results}
\resizebox{\textwidth}{!}{%
\begin{tabular}{c|c|c|c|ccc|ccc}
\toprule[1pt]
\midrule
\multirow{2}{*}{\bf PM} & \multirow{2}{*}{\bf AT} & \multirow{2}{*}{\bf TRADES} & \multirow{2}{*}{\bf PIGD$^\dagger$} & \multicolumn{3}{c|}{{\bf {[}S{]}SIGD}$^\dagger$ ($\sigma^2$)} & \multicolumn{3}{c}{{\bf DS-BLO}$^\dagger$ ($\sigma^2$)} \\
 &  &  &  & $2e-5$ & $6e-5$ & $8e-5$ & $1e-5$ & $5e-5$ & $1e-4$ \\
 \midrule
\multicolumn{10}{c}{{\bf CIFAR-10}, $\epsilon$ = 8/255} \\
\midrule
{\bf SA} & $80.83\pm0.24$ & $80.88\pm0.18$ & $79.43\pm0.21$ & $80.49\pm0.31$ & $81.59\pm0.21$ & $83.33\pm0.15$ & $80.86\pm0.24$ & $81.11\pm0.12$ & $81.19\pm0.32$ \\
{\bf RA} & $50.63\pm0.24$ & $51.34\pm0.24$ & $51.55\pm0.42$ & $50.79\pm0.24$ & $50.66\pm0.16$ & $49.91\pm0.21$ & $50.86\pm0.21$ & $51.79\pm0.15$ & $51.59\pm0.23$ \\
\midrule
\multicolumn{10}{c}{{\bf CIFAR-10}, $\epsilon$ = 16/255} \\
\midrule
{\bf SA} & $70.19\pm0.12$ & $70.25\pm0.33$ & $70.11\pm0.19$ & $71.35\pm0.23$ & $73.72\pm0.23$ & $73.14\pm0.19$ 
& $71.49\pm0.21$ & $70.32\pm0.18$ & $70.85\pm0.25$ \\
{\bf RA} & $32.31\pm0.11$ & $33.39\pm0.29$ & $33.75\pm0.14$ & $32.73\pm0.21$ & $29.19\pm0.54$ & $29.88\pm0.31$ 
& $32.29\pm0.29$ & $33.99\pm0.17$ & $33.13\pm0.09$ \\
\midrule
\multicolumn{10}{c}{{\bf CIFAR-100}, $\epsilon$ = 8/255} \\
\midrule
{\bf SA} & $53.79\pm0.14$ & $53.27\pm0.19$ & $53.77\pm0.23$ & $53.79\pm0.18$ & $53.47\pm0.23$ & $54.34\pm0.14$ 
& $53.84\pm0.21$ & $53.79\pm0.21$ & $53.94\pm0.31$ \\
{\bf RA} & $27.39\pm0.11$ & $28.49\pm0.24$ & $28.19\pm0.26$ & $27.98\pm0.32$ & $28.13\pm0.13$ & $27.99\pm0.22$ 
& $28.29\pm0.27$ & $28.83\pm0.17$ & $28.15\pm0.21$ \\
\midrule
\multicolumn{10}{c}{{\bf CIFAR-100}, $\epsilon$ = 16/255} \\
\midrule
{\bf SA} & $42.12\pm0.21$ & $42.33\pm0.17$ & $41.12\pm0.31$ & $44.14\pm0.21$ & $46.69\pm0.15$ & $47.05\pm0.28$ & $42.32\pm0.12$ & $42.08\pm0.18$ & $42.25\pm0.24$ \\
{\bf RA} & $15.33\pm0.11$ & $16.68\pm0.13$ & $16.66\pm0.12$ & $15.63\pm0.09$ & $13.32\pm0.44$ & $13.72\pm0.22$ & $15.32\pm0.13$ & $15.46\pm0.21$ & $14.93\pm0.11$ \\
\midrule
\bottomrule[1pt]
\end{tabular}%
}
\end{table}
Table \ref{tab: at_results} shows the performance overview of our experiments. We make the following observations. First, bilevel methods (Perturbed Inexact GD, [S]SIGD, and DS-BLO) in general achieve better results than non-bilevel methods (AT, TRADES). This is demonstrated by the higher RA and the competitive SA score. Second, DS-BLO (our proposed method) demonstrates a better RA-SA trade-off. For example, in the setting of (CIFAR-10, $\epsilon=8/255$), DS-BLO achieved an RA of $51.79$ with $\sigma^2=5e-5$, which is the highest value among all the methods. Third, compared to [S]SIGD, DS-BLO demonstrates a better robustness against the choice of the Gaussian variance $\sigma^2$. As we can see, within a larger range of $\sigma^2$, the RA of DS-BLO (the objective of the bilevel optimization) oscillates less than [S]SIGD. For example, in the setting of (CIFAR-10, $\epsilon=16/255$) the RA of [S]SIGD changes from $29.19\%$ to $32.73\%$ with a difference of over $3\%$, while DS-BLO remains the in range from $32.29\%$ to $33.39\%$. Therefore, DS-BLO demonstrates a better algorithmic stability.

\section{Conclusion}
In this work, we developed a framework for solving a special class of constrained bilevel problems where the LL task has linear constraints and a strongly convex objective. The key challenge we tackled for solving this class of problems is the non-differentiability of the implicit function, which is addressed with the use of a perturbation-based smoothing technique. This allows us to compute the gradient of the implicit function in closed form, and develop first-order algorithms to compute its stationary points in finite time. In the future, we would be interested in studying other special classes of constrained bilevel problems (i.e., problems with different constraint sets) and their properties.

\newpage

\bibliographystyle{IEEEtran}
\bibliography{abrv, References, ref-bi}

\newpage

\appendix

\section{Appendix}

\subsection{Proofs of Section  \ref{sec:stochastic_formulation}}\label{app:proofs_sec2}

\begin{proof}[Proof of proposition \ref{pro:diff}]
We will mainly leverage \cite[Theorem 7.44]{stoch_prog} to establish the differentiability of the stochastic implicit function $$\Fob(\x) \coloneqq \Ebb_{\qb \sim \mathcal{Q}} [F_q (\x)] \coloneqq \Ebb_{\qb \sim \mathcal{Q}} [f(\x, \y_q^\ast(\x))].$$ To this end, we are going to show that all the relevant assumptions of \cite[Theorem 7.44]{stoch_prog}  are satisfied by $\Fob(\x)$. 

First, the function $\Fob(\x)$ is well-defined and takes finite values for every $\x \in \du$. Next, the function $F_q (\x)$ is continuous and has bounded gradients (see Lemma \ref{lem:F}). It then follows that $F_q (\x)$ is locally Lipschitz continuous at any given $\x \in \du$. Finally, as we discussed in Sec. \ref{sec: pre}, Lemma \ref{lem: SNAP} implies that the function $F_q (\x)$ is almost surely differentiable at every given $\x \in \du$.

The conditions of \cite[Theorem 7.44]{stoch_prog} are satisfied. As a result, the function $\Fob(\x)$ is differentiable at every given $\x \in \du$ and it holds that 
  \begin{align}
     \nabla \Fob(\x) \coloneqq \nabla  \Ebb_{\qb \sim \mathcal{Q}} [F_q(\x)] =    \Ebb_{\qb \sim \mathcal{Q}} [\nabla F_q(\x)].
  \end{align}
  The result for the differentiability of $\nabla \Fob(\x, \xi)$ follows similarly.

Next, let us derive the closed-form expressions. First, consider the Lagrangian of the LL problem \eqref{eq:sbp_ll}, i.e., 
 $$L(\x,\y,\lb) = g_{q}(\x,\y) + \lb^{T}\left(A\y+B \x-\bb \right) .$$ 
 Then, for some fixed $\x \in \X$, consider a KKT point  $(\y^{\ast}_{q}(\x), \lb_q^{\ast}(\x))$ of \eqref{eq:sbp_ll}, for which it holds that, 
\begin{itemize}[leftmargin = 5mm]
    \item $\nabla_{y} L(\x,\y^{\ast}_{q}(\x),\lb_q^{\ast}(\x)) = \nabla_{y} g_{q}(\x,\y^{\ast}_{q}(\x)) + A^{T} \lb_q^{\ast}(\x)  = 0$
    \item $\left[\lb_q^{\ast}(\x)\right]^{T} \left(A\y^{\ast}_{q}(\x)+ B \x-\bb \right)=0$
    \item $\lb_q^{\ast}(\x) \geq 0$
    \item $A\y^{\ast}_{q}(\x) + B \x-\bb  \leq 0.$
\end{itemize}
Now, consider the active constraints at $(\y^{\ast}_{q}(\x),\lb_q^{\ast}(\x))$, and to simplify notation let us set $\overline{A} :=\overline{A}(\y^{\ast}_{q}(\x))$. Using the notations defined in Section \ref{sec: pre} and the SC property, the KKT conditions given above can be equivalently rewritten as
\begin{align}\label{eq:kkt}
    &\nabla_{y} g_{q}(\x,\y^{\ast}_{q}(\x)) + \overline{A}^{T} \overline{\lb}_q^{\ast}(\x)   = 0, 
    \quad \overline{A} \y^{\ast}_{q}(\x) + \overline{B} \x - \overline{\bb} = 0, \quad   \overline{\lb}_q^{\ast}(\x) >0,
\end{align}
where $\overline{\lb}_q^{\ast}(\x)$ is the subvector of $\lb_q^{\ast}(\x)$ that contains only the elements whose indices correspond to the active constraints at $\y=\y^{\ast}_{q}(\x)$.
Moreover, notice that the point $(\y^{\ast}_{q}(\x), \lb_q^{\ast}(\x))$ is unique. 
The uniqueness of $\y^{\ast}_{q}(\x)$ follows from the strong convexity of $g_{q}(\x, \cdot)$; the uniqueness of $\lb_q^{\ast}(\x)$ results from the fact that matrix $\overline{A}$ has full row rank (which guarantees regularity, e.g., see \cite{bertsekas1998nonlinear}). 

 As mentioned in section \ref{sec: pre}, the SC condition (from Lemma \ref{lem: SNAP}) combined with Assumption \ref{ass:basics} implies that the mapping $\y^\ast(\x)$ is differentiable almost surely  \cite[Theorem 2.22]{friesz2015foundations}. As a result, at any given point $\x$, we can consider a sufficiently small neighborhood around it, such that the active constraints $\overline{A}$ remain unchanged. Then, we can compute the gradient of \eqref{eq:kkt} using the implicit function theorem as follows
\begin{align}
    &\nabla_{xy}^2 g_{q}(\x,\y^{\ast}_{q}(\x)) + \nabla_{yy}^2 g_{q}(\x,\y^{\ast}_{q}(\x)) \nabla\y^{\ast}_{q}(\x) + \overline{A}^{T} \nabla \overline{\lb}_q^{\ast}(\x) = 0  \label{eq:eq11}\\
    &\overline{A} \nabla \y^{\ast}_{q}(\x)+\overline{B} = 0. \label{eq:eq22}
\end{align}
Solving the \eqref{eq:eq11} for $\nabla \y^{\ast}_{q}(\x)$ yields
\begin{align}\label{eq:grad_yast}
    \nabla \y^{\ast}_{q}(\x) = \big[ \nabla_{yy}^2 g_{q}(\x,\y^{\ast}_{q}(\x)) \big]^{-1} \left[-\nabla_{xy}^2 g_{q}(\x,\y^{\ast}_{q}(\x)) -\overline{A}^{T} \nabla \overline{\lb}_q^{\ast}(\x) \right],
\end{align}
where we exploited the fact that the Hessian matrix $\nabla_{yy}^2 g(\x,\y^{\ast}_{q}(\x))$ is positive definite and thus invertible.
By substituting \eqref{eq:grad_yast} into \eqref{eq:eq22} and noting that $\nabla^{2}_{yy}g_{q}(\x,\y)=\nabla^{2}_{yy}g(\x,\y)$ and $\nabla^{2}_{xy}g_{q}(\x,\y)=\nabla^{2}_{xy}g(\x,\y)$, we obtain the following.
\begin{align*}
    &\overline{A} \big[ \nabla_{yy}^2 g(\x,\y^{\ast}_{q}(\x)) \big]^{-1} \left[-\nabla_{xy}^2 g(\x,\y^{\ast}_{q}(\x)) -\overline{A}^{T} \nabla \overline{\lb}_q^{\ast}(\x) \right] + \overline{B} =0  \Longrightarrow\\
     &\nabla \overline{\lb}_q^{\ast}(\x)   =   -   \left[ \overline{A} \big[ \nabla_{yy}^2 g(\x,\y^{\ast}_{q}(\x))^{-1} \big]\overline{A}^{T} \right]^{-1} \\
     & \qquad \qquad \qquad \qquad \qquad \qquad
     \left[\overline{A}\big[ \nabla_{yy}^2 g(\x,\y^{\ast}_{q}(\x)) \big]^{-1} \nabla_{xy}^2 g(\x,\y^{\ast}_{q}(\x)) - \overline{B}\right].
\end{align*}
Finally, note that the KKT point $\y^{\ast}_{q}(\x)$ corresponds to the unique global minimum of  \eqref{eq:sbp_ll}, due to the strong convexity of $g_{q}(\x,\cdot)$. 
The proof is now complete.
\end{proof}
\subsubsection{The proof of Lemma \ref{lem:F}}\label{app: lem:F}
We provide several intermediate results below that are necessary to prove Lemma \ref{lem:F}. Under Assumption \ref{ass:Approx_y}\ref{ass:3active} we have that $\overline{A}(\y_q^{\ast}(\x))=\overline{A}(\widehat{\y}_q(\x))$. For simplicity, we will refer to these matrices simply as $\overline{A}$ throughout this subsection. Additionally, for any given matrix $A$ we denote by $L_{A}$ the maximum value of $\|\overline{A}\left(\widehat{\y}_q(\x) \right)\|$, over all $\x \in \X$, $\overline{L}_{A}:=\frac{L_{y}}{\lambda_{\text{min}}\left(\overline{A}\left(\widehat{\y}_q(\x)\right)\overline{A}\left(\widehat{\y}_q(\x)\right)^{T}\right)}$ and $\|B\|:=L_{B}$.
\begin{lem}\label{lem:prelim1}
Suppose that Assumptions \ref{ass:basics}, \ref{ass:Fn_UL_LL}, \ref{ass:basics_feas_rank} and \ref{ass:Approx_y} hold. Then for any $\x \in \X$, we have:
\begin{enumerate}[label=(\alph*)]
    \item\label{lem:prelim1a}  $
    \left\| \left[\nabla_{yy}^2 g(\x,\y)\right]^{-1}\right\|  \leq \frac{1}{\mu_{g}}, \; \forall \y \in \dl.
    $
    \item\label{lem:prelim1b} $
    \left\|\left[\nabla_{yy}^2 g(\x,\y_{q}^{\ast}(\x))\right]^{-1} - \left[ \nabla_{yy}^2 g(\x,\widehat{\y}_{q}(\x))\right]^{-1} \right\|
    \leq \left(\frac{1}{\mu_{g}}\right)^{2} L_{g_{yy}} \delta.
   $
    \item\label{lem:prelim1c} $
    \left\| \left[\overline{A} \left[ \nabla_{yy}^2 g(\x,\y)\right]^{-1} \overline{A}^{T}\right]^{-1}\right\| \leq  \overline{L}_{A}, ~\forall \y \in \dl.
    $
    \item\label{lem:prelim1d} 
    $ \left\| \left[\overline{A} \left[\nabla_{yy}^2 g(\x,\y^{\ast}_{q}(\x))\right]^{-1} \overline{A}^{T}\right]^{-1}\!\!\!\!\! - \left[\overline{A} \left[\nabla_{yy}^2 g(\x,\widehat{\y}_{q}(\x))\right]^{-1} \overline{A}^{T}\right]^{-1} \right\|   \leq L_{A}^{2} \overline{L}_{A}^{2}   \frac{1}{\mu_{g}^2}  L_{g_{yy}} \delta.
    $
    
\end{enumerate}
\end{lem}
\begin{proof}
See proof of \cite[Lemma D.6]{khanduri2023linearly}
\end{proof}
Now let us bound the norm of the gradients of the mappings $\lb_q^{\ast}(\x)$ and $\y_{q}^{\ast}(\x)$.
\begin{lem}\label{lem:lambda_y_bound}
Under Assumptions \ref{ass:basics}, \ref{ass:Fn_UL_LL}, \ref{ass:basics_feas_rank} and \ref{ass:Approx_y}, the gradients of the mappings $\lb_q^{\ast}(\x)$ and  $\y_q^\ast(\x)$ satisfy the following bounds for every $\x \in \X$,
\label{Lem: Bounded_Grad_Y}
\begin{align*}
    &\|\nabla \overline{\lb}_q^{\ast}(\x) \| \leq \overline{L}_{\lb^{\ast}}, \;\;  \left\| \widehat{\nabla}  \overline{\lb}_q^\ast(\x) \right\| \leq \overline{L}_{\lb^{\ast}} \\
    &\|\nabla \y_{q}^{\ast}(\x) \| \leq \overline{L}_{\y^{\ast}}, \;\; \|\widehat{\nabla} {\y}_q^\ast(\x) \| \leq \overline{L}_{\y^{\ast}}
\end{align*}
where $\overline{L}_{\lb^{\ast}} = \frac{1}{\mu_{g}} \overline{L}_{A} L_{A}  \overline{L}_{g_{xy}}$ and $\overline{L}_{\y^{\ast}} = \frac{1}{\mu_y} \left( \overline{L}_{g_{xy}} + L_{A}\overline{L}_{\lb^{\ast}}\right)$. Note that $\widehat{\nabla}  \overline{\lb}_q^\ast(\x)$ and $\widehat{\nabla} {\y}_{q}^\ast(\x)$ are obtained by substituting the estimate $\widehat{\y}_{q}(\x)$ in place of $\y_{q}^\ast(\x)$ in the expressions $\nabla \overline{\lb}_q^{\ast}(\x)$ and $\nabla \y_{q}^{\ast}(\x)$, respectively (Please see Proposition \ref{pro:diff}). 
\end{lem}
\begin{proof}
From Proposition \ref{pro:diff} we have
\begin{align*}
    \nabla \overline{\lb}_q^{\ast}(\x)   =   -  \left[ \overline{A} \left[ \nabla_{yy}^2 g(\x,\y_{q}^{\ast}(\x))\right]^{-1} \overline{A}^{T} \right]^{-1} 
     \left[\overline{A}\left[ \nabla_{yy}^2 g(\x,\y_{q}^{\ast}(\x)) \right]^{-1} \nabla_{xy}^2 g(\x,\y_{q}^{\ast}(\x))-\overline{B}\right]
\end{align*}
Then, taking the norm of this quantity, we get,
\begin{align*}
    &\left\| \nabla \overline{\lb}_q^{\ast}(\x) \right\| 
    \hspace{-1mm}=\hspace{-1mm}\left\| \left[ \overline{A} \left[ \nabla_{yy}^2 g(\x,\y_{q}^{\ast}(\x))\right]^{-1} \overline{A}^{T} \right]^{-1} 
     \left[\overline{A}\left[ \nabla_{yy}^2 g(\x,\y_{q}^{\ast}(\x)) \right]^{-1} \nabla_{xy}^2 g(\x,\y_{q}^{\ast}(\x))\right]-\overline{B} \right\| \\
    &\leq \left\| \left[ \overline{A} \left[ \nabla_{yy}^2 g(\x,\y_{q}^{\ast}(\x))\right]^{-1} \overline{A}^{T} \right]^{-1} \right\|\left[ \left\| \overline{A} \right\| \left\|\big[ \nabla_{yy}^2 g(\x,\y_{q}^{\ast}(\x)) \big]^{-1}  \right\| \left\|\nabla_{xy}^2 g(\x,\y_{q}^{\ast}(\x)) \right\|+\|B\|\right] \\
     &\leq \overline{L}_{A} \left[ L_{A}  \frac{1}{\mu_{g}} \overline{L}_{g_{xy}}+L_{B} \right] :=\overline{L}_{\lb^{\ast}},
\end{align*}
where in the last inequality we used Lemma \ref{lem:prelim1}\ref{lem:prelim1a},  \ref{lem:prelim1}\ref{lem:prelim1c} and Assumption \ref{ass:Fn_UL_LL}\ref{ass:Hes_bound_gxy}.

Similarly, we have $\left\| \widehat{\nabla}  \overline{\lb}_q^{\ast}(\x) \right\| \leq \overline{L}_{\lb^{\ast}}$ we have that 

Moving to the bound of $\|\nabla \y^{\ast}(\x) \|$, we know from Lemma \ref{pro:diff} that the formula of the gradient of $\y_q^{\ast}(\x)$ is
\begin{align}
      \nabla \y_{q}^{\ast}(\x) = \left[ \nabla_{yy}^2 g(\x,\y_{q}^{\ast}(\x)) \right]^{-1}\left[-\nabla_{xy}^2 g(\x,\y_{q}^{\ast}(\x)) -\overline{A}^{T} \nabla \overline{\lb}_q^{\ast}(\x)\right]. 
\end{align}
Then, we have that  
\begin{align*}
    \left\|\nabla \y_{q}^{\ast}(\x) \right\| &= \left\|\big[ \nabla_{yy}^2 g(\x,\y_{q}^{\ast}(\x)) \big]^{-1} \left[-\nabla_{xy}^2 g(\x,\y_{q}^{\ast}(\x)) -\overline{A}^{T} \nabla \overline{\lb}_q^{\ast}(\x)\right] \right\| \\
    &\leq \left\|\big[ \nabla_{yy}^2 g(\x,\y_{q}^{\ast}(\x)) \big]^{-1} \right\| \left\|\left[-\nabla_{xy}^2 g(\x,\y_{q}^{\ast}(\x)) -\overline{A}^{T} \nabla \overline{\lb}_q^{\ast}(\x)\right] \right\| \\
    &\leq \frac{1}{\mu_{g}} \left( \left\|\nabla_{xy}^2 g(\x,\y_{q}^{\ast}(\x)) \right\| + \left\|\overline{A} \right\| \left\|\nabla \overline{\lb}_q^{\ast}(\x) \right\| \right) \\
    &\leq \frac{1}{\mu_{g}} \left( \overline{L}_{g_{xy}} + L_{A} \overline{L}_{\lb^{\ast}} \right):=\overline{L}_{\y^{\ast}},
\end{align*}
where in the second inequality we used we used Lemma \ref{lem:prelim1}\ref{lem:prelim1a}; the third inequality follows from Assumption \ref{ass:Fn_UL_LL}\ref{ass:Hes_bound_gxy} and the bound for $\left\| \nabla \overline{\lb}_q^{\ast}(\x) \right\|$ we derived above. 

Similarly, we have  $\left\|\widehat{\nabla} \y_q(\x) \right\| \leq \overline{L}_{\y^{\ast}}$ we can obtain the following bound

The proof is now complete.
\end{proof}

In the next two results, we are going to present bounds for the difference of the exact and approximate gradients of the mappings $\overline{\lb}_q^{\ast}(\x)$ and $\y_q^{\ast}(\x)$.
\begin{lem}\label{lem:lambda_err}
Suppose that Assumptions \ref{ass:basics}, \ref{ass:Fn_UL_LL}, \ref{ass:basics_feas_rank} and \ref{ass:Approx_y} hold. Then, the following bound holds
\begin{align*}
    \|\nabla \overline{\lb}_q^{\ast}(\x) -\widehat{\nabla} \overline{\lb}_q^{\ast}(\x) \| \leq L_{\lb^{\ast}}\delta,
\end{align*}
where 
$L_{\lb^{\ast}} =\left( \frac{1}{\mu_{g}} \right)^{3} \overline{L}_{A}^{2} L_{A}^{3}  L_{g_{yy}}  \overline{L}_{g_{xy}} + 
    \frac{1}{\mu_{g}} \overline{L}_{A} L_{A}  L_{g_{xy}} + \left(\frac{1}{\mu_{g}}\right)^{2} \overline{L}_{A} L_{A}  L_{g_{yy}} \overline{L}_{g_{xy}}$.
\end{lem}
\begin{proof}
Using the derivation of $\nabla \overline{\lb}_q^{\ast}(\x)$ from Proposition \ref{pro:diff}, and its approximation $\widehat{\nabla}  \overline{\lb}_q^{\ast}(\x)$ where we substitute $\y_{q}^{\ast}(\x)$ with $\widehat{\y}_{q}(\x)$ in the formula of the former, that is,
\begin{align*}
    \widehat{\nabla}  \overline{\lb}_q^{\ast}(\x)   =   -  \left[ \overline{A} \left[ \nabla_{yy}^2 g(\x,\widehat{\y}_{q}(\x))\right]^{-1} \overline{A}^{T} \right]^{-1} 
     \left[\overline{A}\left[ \nabla_{yy}^2 g(\x,\widehat{\y}_{q}(\x)) \right]^{-1} \nabla_{xy}^2 g(\x,\widehat{\y}_{q}(\x))\right],
\end{align*}
we obtain 
\begin{align*}
    &\left\|\nabla \overline{\lb}_q^{\ast}(\x) - \widehat{\nabla}  \overline{\lb}_q^{\ast}(\x) \right\| \\
    &= \bigg\| \left[ \overline{A} \left[ \nabla_{yy}^2 g(\x,\y_{q}^{\ast}(\x))\right]^{-1} \overline{A}^{T} \right]^{-1} 
     \left[\overline{A}\left[ \nabla_{yy}^2 g(\x,\y_{q}^{\ast}(\x)) \right]^{-1} \nabla_{xy}^2 g(\x,\y_{q}^{\ast}(\x))-\overline{B}\right]\\
     &\hspace{3mm}-\left[ \overline{A} \left[ \nabla_{yy}^2 g(\x,\widehat{\y}_{q}(\x))\right]^{-1} \overline{A}^{T} \right]^{-1} 
     \left[\overline{A}\left[ \nabla_{yy}^2 g(\x,\widehat{\y}_{q}(\x)) \right]^{-1} \nabla_{xy}^2 g(\x,\widehat{\y}_{q}(\x))-\overline{B}\right] 
     \bigg\|.
\end{align*}
Below, we use the following notation in order to simplify the derivations. 
\begin{align*}
&H = \left[ \overline{A} \left[ \nabla_{yy}^2 g(\x,\y_{q}^{\ast}(\x))\right]^{-1} \overline{A}^{T} \right]\hspace{-1mm},
G = \left[ \nabla_{yy}^2 g(\x,\y_{q}^{\ast}(\x)) \right]^{-1}\hspace{-1mm}, M = \nabla_{xy}^2 g(\x,\y_{q}^{\ast}(\x)) \\
&\widehat{H} = \left[ \overline{A} \left[ \nabla_{yy}^2 g(\x,\widehat{\y}_{q}(\x))\right]^{-1} \overline{A}^{T} \right]\hspace{-1mm},
\widehat{G} = \big[ \nabla_{yy}^2 g(\x,\widehat{\y}_{q}(\x)) \big]^{-1}\hspace{-1mm},
\widehat{M} = \nabla_{xy}^2 g(\x,\widehat{\y}_{q}(\x))
\end{align*}
Then, we have 
\begin{align}\label{eq:nabla_lambda_diff}
    &\|\nabla \overline{\lb}_q^{\ast} (\x) - \widehat{\nabla}  \overline{\lb}_q^{\ast}(\x) \| \nonumber\\
    &= \left\| H^{-1} \left[ \overline{A}G M -\overline{B}\right] -\widehat{H}^{-1}  \left[ \overline{A}\widehat{G} \widehat{M} -\overline{B} \right]\right\| \nonumber\\
    &\stackrel{(a)}{\leq} \left\|H^{-1} \left[ \overline{A}G M-\overline{B}\right] - \widehat{H}^{-1} \left[ \overline{A}G M -\overline{B}\right]\right\| \nonumber\\
    &+\left\|\widehat{H}^{-1} \left[ \overline{A}G M -\overline{B}\right] -\widehat{H}^{-1} \left[ \overline{A}\widehat{G} \widehat{M} -\overline{B}\right] \right\| \nonumber\\
    &\leq \left\|H^{-1}-\widehat{H}^{-1} \right\| \left[\left\| \overline{A} \right\| \left\|  G \right\| \left\| M\right\| -\|B\|\right]
    + \left\|\widehat{H}^{-1}\right\| \left\| \overline{A} \right\| \left\|  G M - \widehat{G} \widehat{M} \right\| \nonumber\\
    &\stackrel{(b)}{\leq}\hspace{-1mm} \left\|H^{-1}-\widehat{H}^{-1} \right\| \left[\left\| \overline{A}\right\| \left\|  G \right\| \left\| M \right\| -\|B\|\right]
    \nonumber\\
    &+\left\|\widehat{H}^{-1} \right\| \left\| \overline{A} \right\| 
    \left[ \left\| G M - G \widehat{M} \right\| \hspace{-1mm}+\hspace{-1mm} \left\| G \widehat{M} - \widehat{G} \widehat{M} \right\| \right] \nonumber\\
     &\leq \left\|H^{-1}-\widehat{H}^{-1} \right\| \left[\left\| \overline{A} \right\| \left\|  G \right\| \left\| M\right\| -\|B\|\right]\nonumber\\
    &+ \left\|\widehat{H}^{-1} \right\| \left\| \overline{A} \right\| \left\| G\right\| \left\|  M - \widehat{M} \right\|   + 
    \left\|\widehat{H}^{-1}\right\| \left\| \overline{A} \right\| \left\| G - \widehat{G} \right\| \left\|\widehat{M} \right\|  \nonumber\\
    &\stackrel{(c)}{\leq}\hspace{-1mm} \left(\overline{L}_{A}^{2} L_{A}^{2} \left( \frac{1}{\mu_{g}} \right)^{2} L_{g_{yy}} \left(L_{A} \frac{\overline{L}_{g_{xy}}}{\mu_{g}}  +L_{B}\right) \hspace{-1mm}+\hspace{-1mm} 
    \frac{1}{\mu_{g}} \overline{L}_{A} L_{A}  L_{g_{xy}} \hspace{-1mm}+\hspace{-1mm} \left(\frac{1}{\mu_{g}}\right)^{2} \hspace{-1mm}\overline{L}_{A} L_{A}  L_{g_{yy}} \overline{L}_{g_{xy}}  \right)\delta.
\end{align}
In (a) we add and subtract the term $\widehat{H}^{-1}  \overline{A}G M$ and apply the triangle inequality. In (b) we add and subtract the term $ G \widehat{M}$ and apply the triangle inequality. In (c) we use Lemma \ref{lem:prelim1}\ref{lem:prelim1d} for $ \|H^{-1}-\widehat{H}^{-1} \|$, the bound $\|\overline{A}\| \leq L_{A}$, Lemma \ref{lem:prelim1}\ref{lem:prelim1a} for $\| G\|$, Lemma \ref{lem:prelim1}\ref{lem:prelim1c} for $\|H^{-1}\|$ and $\|\widehat{H}^{-1}\|$, Assumption \ref{ass:Fn_UL_LL}\ref{ass:Hes_bound_gxy} for $\| M\|$ and $\| \widehat{M}\|$, Assumption \ref{ass:Fn_UL_LL}\ref{ass:Hes_lip_gxy} for $\|  M - \widehat{M} \|$,  and finally Lemma \ref{lem:prelim1}\ref{lem:prelim1b} for $\| G - \widehat{G} \|$. 

The proof is now complete.
\end{proof}
\begin{lem}\label{lem:y_err}
Suppose that Assumptions \ref{ass:basics}, \ref{ass:Fn_UL_LL}, \ref{ass:basics_feas_rank} and \ref{ass:Approx_y} hold. Then, the following bound holds
\begin{align*}
    \|\nabla \y_{q}^{\ast}(\x) - \widehat{\nabla} \y_{q}(\x) \| \leq L_{\y^{\ast}} \delta,
\end{align*}
where $L_{\y^{\ast}} = \left(\frac{1}{\mu_{g}} \right)^{2} L_{g_{yy}} \overline{L}_{g_{xy}} + \frac{1}{\mu_{g}} L_{g_{xy}} + \left( \frac{1}{\mu_{g}} \right)^{2} L_{g_{yy}} L_{A} \overline{L}_{\lb^{\ast}} + \frac{1}{\mu_{g}} L_{A} L_{\lb^{\ast}}$.
\end{lem}
\begin{proof}
See proof of \cite[Lemma D.9]{khanduri2023linearly}.
\end{proof}
Now we have all the results needed to prove Lemma \ref{lem:F}.
\begin{proof}[Proof of Lemma \ref{lem:F}]
See \cite[Lemma 2.8]{khanduri2023linearly} for the proof of the first three expressions. Moreover, for the fourth expression, we apply Jensen's inequality and obtain 
$$\|\nabla \Fob(\x)\| =  \|\Ebb_{\qb \sim \mathcal{Q}} [\nabla F_q(\x)]\| \leq \Ebb_{\qb \sim \mathcal{Q}} [\|\nabla F_q(\x)\|] \leq \overline{L}_{F}.$$
\end{proof}
\subsubsection{Proof of Lemma \ref{lem:F_Variance}}
\label{app: lem:var}
\begin{proof}[Proof of Lemma \ref{lem:F_Variance}]
From Lemma \ref{lem:F} it follows that 
\begin{align*}
    \Ebb_{\qb \sim \mathcal{Q}} \|\nabla F_q(\x) - \nabla \Fob(\x)  \|^2 \leq 
    \Ebb_{\qb \sim \mathcal{Q}} \left[2\|\nabla F_q(\x)\|^{2} + 2\|\nabla \Fob(\x)\|^2 \right] 
    \leq 4\overline{L}_{F}^{2}:=\delta_{\vb}^2
\end{align*}
\end{proof}

\subsubsection{Proof of Lemma \ref{lem: SG_IG}}\label{sub:stoch_grad}
\begin{proof}[Proof of Lemma \ref{lem: SG_IG}]
    From the definition of the stochastic gradient in \eqref{eq: SG_Implicit}, we have  
$$ \widehat{\nabla} F_{q}(\x; \xi)=\nabla_{x} \widetilde{f}(\x, \widehat{\y}_{q}(\x); \xi) +  [\widehat{\nabla } \y_{q}^{\ast}(\x)]^T \nabla_{y} \widetilde{f}(\x, \widehat{\y}_{q}(\x);\xi).$$
Taking expectation on both sides and utilizing the unbiasedness of the stochastic gradient of the UL objective $f(\x,\y)$ (Assumption \ref{ass:SG}), we get
\begin{align*}
  \mathbb{E}_{\xi} [\widehat{\nabla} F_{q}(\x; \xi)] & =\mathbb{E}_{\xi} \big[\nabla_{x} \widetilde{f}(\x, \widehat{\y}_{q}(\x); \xi) +  [\widehat{\nabla } \y_{q}^{\ast}(\x)]^T \nabla_{y} \widetilde{f}(\x, \widehat{\y}_{q}(\x);\xi)\big] \\
  & =\mathbb{E}_{\xi} \big[\nabla_{x} \widetilde{f}(\x, \widehat{\y}_{q}(\x); \xi) \big] +  [\widehat{\nabla } \y_{q}^{\ast}(\x)]^T \mathbb{E}_{\xi} \big[ \nabla_{y} \widetilde{f}(\x, \widehat{\y}_{q}(\x);\xi)\big] \\
  &  = \nabla_{x} f(\x, \widehat{\y}_{q}(\x))   +  [\widehat{\nabla } \y_{q}^{\ast}(\x)]^T   \nabla_{y} f(\x, \widehat{\y}_{q}(\x)) \\
  & = \widehat{\nabla} F_{q}(\x).
\end{align*}
Similarly, for the variance of the stochastic implicit gradient, we have
\begin{align*}
    &\mathbb{E}_{\qb \sim \mathcal{Q}, \xi \sim \mathcal{D}_f}\| {\nabla} F_{q}(\x; \xi) - \nabla \Fob (\x) \|^2 \\
   & =  \mathbb{E}_{\qb \sim \mathcal{Q}, \xi \sim \mathcal{D}_f}\| {\nabla} F_{q}(\x; \xi) - {\nabla} F_{q}(\x) + {\nabla} F_{q}(\x) - \nabla \Fob (\x) \|^2  \\
   & \overset{(a)}{\leq} 2~ \mathbb{E}_{\qb \sim \mathcal{Q}, \xi \sim \mathcal{D}_f}\| {\nabla} F_{q}(\x; \xi) - {\nabla} F_{q}(\x) \|^2 +  2~ \mathbb{E}_{\qb \sim \mathcal{Q}} \| {\nabla} F_{q}(\x) - \nabla \Fob (\x) \|^2 
    \\
    & \leq 2~ \Ebb_{\qb \sim \mathcal{Q}} \big[ \mathbb{E}_{\xi \sim \mathcal{D}_f} \big[ \| {\nabla} F_{q}(\x; \xi) - {\nabla} F_{q}(\x) \|^2 | \qb  \big] \big] +  2~ \mathbb{E}_{\qb \sim \mathcal{Q}} \| {\nabla} F_{q}(\x) - \nabla \Fob (\x) \|^2 ,
    \end{align*}
where $(a)$ follows from $\| \x + \y \|^2 \leq 2 \| \x \|^2 + 2 \|\y \|^2$. We know from Lemma \ref{lem:F_Variance} that the second term $\mathbb{E}_{\qb \sim \mathcal{Q}} \| {\nabla} F_{q}(\x) - \nabla \Fob (\x) \|^2 \leq \delta_{\vb}^2$. Next, upper bounding the first term in the above expression.
    \begin{align*}
   &  \mathbb{E}_{\xi \sim \mathcal{D}_f}   \| {\nabla} F_{q}(\x; \xi) - {\nabla} F_{q}(\x) \|^2     \\
    & =  \mathbb{E}_{\xi} \big\| \nabla_{x} \widetilde{f}(\x, {\y}^\ast_{q}(\x); \xi) +  [{\nabla} \y_{q}^{\ast}(\x)]^T \nabla_{y} \widetilde{f}(\x, {\y}^\ast_{q}(\x);\xi) \\
    & \qquad \qquad \qquad \qquad \qquad -  \big[\nabla_{x}  {f}(\x, {\y}^\ast_{q}(\x)) +  [{\nabla} \y_{q}^{\ast}(\x)]^T \nabla_{y}  {f}(\x, {\y}^\ast_{q}(\x)) \big] \big\|^2 \\
    & \overset{(a)}{\leq} 2~ \mathbb{E}_{\xi} \|\nabla_{x} \widetilde{f}(\x, {\y}_q^\ast(\x); \xi) - \nabla_{x}  {f}(\x,  {\y}_q^\ast(\x))  \|^2 \\
    & \qquad \qquad \qquad \qquad \qquad + 2 ~\| {\nabla } \y_q^{\ast}(\x) \|^2~ \mathbb{E}_{\xi} \| \nabla_{y} \widetilde{f}(\x,  {\y}_q^\ast(\x);\xi)  - \nabla_{y} {f}(\x, {\y}_q^\ast(\x) ) \|^2 \\
    & \overset{(b)}{\leq} 2 \sigma_f^2 + 2 \overline{L}_{\y^\ast} \sigma_f^2 \coloneqq \sigma_F^2,
    \end{align*}
    where $(a)$ follows from $\| \x + \y \|^2 \leq 2 \| \x \|^2 + 2 \|\y \|^2$ and $(b)$ results from Assumption \ref{ass:SG} and the application of Lemma \ref{Lem: Bounded_Grad_Y}. Combining the upper bounds of the two terms and defining new $\delta_\vb^2$ as $\sigma_F^2 + \delta_\vb^2$. The proof is completed.
\end{proof}

\subsubsection{Local Lipschitz smoothness of the perturbed LL solution map}
\begin{lem}
\label{lem: grad_y_lip}
   Let $\x$ be an arbitrary point in the domain $\du$. Then, there exists a neighborhood $U$ of $\x$, such that the following holds:
   \begin{align*}
       \left\|\nabla \y_q^{\ast}(\x)-\nabla \y_q^{\ast}(\xo) \right\| \leq \widetilde{L}_{\y^{\ast}} \| \x-\xo\|, \; \forall\; \xo \in U, 
   \end{align*}
   where $\widetilde{L}_{\y^{\ast}}$ is defined in eq. \eqref{eq:lip_y_const} and 
    \begin{align*}
       \left\|\nabla F_q(\x)-\nabla F_q(\xo) \right\| \leq \widetilde{L}_{F_{q}} \| \x-\xo\|, \; \forall\; \xo \in U, 
   \end{align*}
   where $\widetilde{L}_{F_{q}}$ is defined in eq. \eqref{eq:lip_F_const}.
\end{lem}
\begin{proof}
Let $\x$ be an arbitrary point in the domain $\du$. We are going to show that there exists a neighborhood around $\x$ such that the mapping $\nabla \y_{q}^{\ast}(\x)$ is Lipschitz continuous. Indeed, \cite[Theorem 2.22]{Friesz_Foundations_2016} establishes that for any given $\x \in \du$, there exists a neighborhood around $\x$ in which the active constraints at $\y_{q}^{\ast}(\x)$ remain unchanged. It is in this neighborhood $U$ that our local Lipschitz continuity property is established.
To formalize this result, for any given point $\x \in \du$, we pick an arbitrary point $\xo \in U$. Then, it holds that 
$\overline{A}(\y_{q}^{\ast}(\x))=\overline{A}(\y_{q}^{\ast}(\xo)):=A$ and $\overline{B}(\y_{q}^{\ast}(\x))=\overline{B}(\y_{q}^{\ast}(\xo)):=B$.
Moreover, to simplify the calculations we introduce the following notation.
\begin{align*}
&G_{yy}=\nabla_{yy}^2 g(\x,\y_q^{\ast}(\x)), \overline{G}_{yy}=\nabla_{yy}^2 g(\xo,\y_q^{\ast}(\xo)) \\
&G_{xy}=\nabla_{xy}^2 g(\x,\y_q^{\ast}(\x)), \overline{G}_{xy}=\nabla_{xy}^2 g(\xo,\y_q^{\ast}(\xo))
\end{align*}  
Then, the implicit gradient formulas of Proposition \ref{pro:diff} take the following form.
\begin{align}
\begin{split}
\nabla \y_q^{\ast}(\x) & = G_{yy}^{-1}   
\big[-G_{xy} -A^\top \nabla \overline{\lb}_q^{\ast}(\x)\big],  \\
\nabla \y_q^{\ast}(\xo) & = \overline{G}_{yy}^{-1}  
\big[-\overline{G}_{xy} -A^\top \nabla \overline{\lb}_q^{\ast}(\xo)\big] \label{eq:yq_yqbar}
\end{split}
\end{align}
\begin{align}
\begin{split}
\nabla \overline{\lb}_q^{\ast}(\x) & =  - \big[ A G_{yy}^{-1}A^\top \big]^{-1} \big[A G_{yy}^{-1} G_{xy}- B\big],  \\
 \nabla \overline{\lb}_q^{\ast}(\xo)  & =   - \big[ A \overline{G}_{yy}^{-1}A^\top \big]^{-1} \big[A \overline{G}_{yy}^{-1} \overline{G}_{xy}- B\big].\label{eq:lamq_lamqbar}
 \end{split}
\end{align}
We begin our derivations with the $\nabla \y_q^{\ast}(\x)$ term.
\begin{align}\label{eq:local_lip_y}
&\left\|\nabla \y_q^{\ast}(\x) - \nabla \y_q^{\ast}(\xo) \right\| \nonumber \\
&\overset{(a)}{\leq} \left\| G_{yy}^{-1} G_{xy} - \overline{G}_{yy}^{-1} \overline{G}_{xy}\right\| + \left\|G_{yy}^{-1}A^\top\nabla \overline{\lb}_q^{\ast}(\x) - \overline{G}_{yy}^{-1}A^\top\nabla \overline{\lb}_q^{\ast}(\xo)\right\| \nonumber \\
&\overset{(b)}{\leq} \left\| G_{yy}^{-1} G_{xy} -G_{yy}^{-1} \overline{G}_{xy}  +G_{yy}^{-1} \overline{G}_{xy} - \overline{G}_{yy}^{-1} \overline{G}_{xy}\right\| \nonumber \\
&+ \left\|G_{yy}^{-1} A^\top\nabla \overline{\lb}_q^{\ast}(\x) 
-\overline{G}_{yy}^{-1}A^\top\nabla \overline{\lb}_q^{\ast}(\x)
+\overline{G}_{yy}^{-1}A^\top\nabla \overline{\lb}_q^{\ast}(\x)
- \overline{G}_{yy}^{-1}A^\top\nabla \overline{\lb}_q^{\ast}(\xo)\right\| \nonumber \\
&\overset{(c)}{\leq}\left[\left\|G_{yy}^{-1}\right\| \left\|G_{xy}-\overline{G}_{xy}\right\| + \left\|\overline{G}_{xy}\right\| \left\|G_{yy}^{-1}-\overline{G}_{yy}^{-1} \right\| \right] \nonumber\\
&+ \left\|G_{yy}^{-1}-\overline{G}_{yy}^{-1} \right\| \left\|A^\top\right\| \left\|\nabla \overline{\lb}_q^{\ast}(\x)\right\| 
+\left\|\overline{G}_{yy}^{-1}\right\| \left\|A^\top \right\| \left\|\nabla \overline{\lb}_q^{\ast}(\x)-\overline{\lb}_q^{\ast}(\xo)\right\| \nonumber\\ 
&\overset{(d)}{\leq} \left[\frac{1}{\mu_{g}}L_{g_{xy}} + \overline{L}_{g_{xy}} \left(\frac{1}{\mu_{g}}\right)^{2}L_{g_{yy}} \right] \left[ \left\|\x-\xo\right\| + \left\|\y_{q}^{\ast}(\x)-\y_{q}^{\ast}(\xo)\right\| \right] \nonumber\\
&+ \overline{L}_{\lambda^{\ast}}L_{A}\left(\frac{1}{\mu_{g}}\right)^{2}L_{g_{yy}}\left[ \left\|\x-\xo\right\| + \left\|\y_{q}^{\ast}(\x)-\y_{q}^{\ast}(\xo)\right\| \right] 
+\frac{L_{A}}{\mu_{g}} \left\|\nabla \overline{\lb}_q^{\ast}(\x)-\nabla \overline{\lb}_q^{\ast}(\xo)\right\| \nonumber\\
&\overset{(e)}{\leq} \left[\frac{1}{\mu_{g}}L_{g_{xy}} + \overline{L}_{g_{xy}} \left(\frac{1}{\mu_{g}}\right)^{2}L_{g_{yy}} + \overline{L}_{\lambda^{\ast}}L_{A}\left(\frac{1}{\mu_{g}}\right)^{2}L_{g_{yy}}\right] 
\left[ 1 + \overline{L}_{\y^{\ast}} \right] \left\|\x-\xo\right\| \nonumber\\
& +\frac{L_{A}}{\mu_{g}} \left\|\nabla \overline{\lb}_q^{\ast}(\x)-\nabla \overline{\lb}_q^{\ast}(\xo)\right\|  
\end{align}
In (a) we used the definition of $\nabla \y_q^{\ast}(\x)$; in (b) In we add a subtract a number of terms; in (c) we use the triangle inequality and the submultiplicative property of matrix norms; in (d) we used Lemmas \ref{lem:prelim1}, \ref{lem:lambda_y_bound} and Assumption \ref{ass:Fn_UL_LL}; in (e) we applied Lemma \ref{lem:lambda_y_bound} where the boundedness of $\nabla \y_{q}^{\ast}(\x)$ implies the (local) Lipschitz continuity of $\y_{q}^{\ast}(\x)$. 

Next we work on the term $\left\|\nabla \overline{\lb}_q^{\ast}(\x)-\nabla \overline{\lb}_q^{\ast}(\xo)\right\|$, which appears in the above expression. Specifically, we have 
\begin{align}\label{eq:local_lip_lam}
&\|\nabla \overline{\lb}_q^{\ast}(\x)-\nabla \overline{\lb}_q^{\ast}(\xo)\| \nonumber\\
&\overset{(a)}{\leq}\left\|- \big[ A G_{yy}^{-1}A^\top \big]^{-1} \big[A G_{yy}^{-1} G_{xy}- B\big]
+ \big[ A \overline{G}_{yy}^{-1}A^\top \big]^{-1} \big[A\overline{G}_{yy}^{-1} \overline{G}_{xy}- B\big] \right\|\nonumber\\
& \overset{(b)}{\leq} \left\|- \big[ A G_{yy}^{-1}A^\top \big]^{-1} \big[AG_{yy}^{-1} G_{xy}- B\big]
+\big[ A \overline{G}_{yy}^{-1}A^\top \big]^{-1} \big[A G_{yy}^{-1} G_{xy}- B\big] \right\| \nonumber\\
&\hspace{5mm}\left\|-\big[ A \overline{G}_{yy}^{-1}A^\top \big]^{-1} \big[A G_{yy}^{-1} G_{xy}- B\big]
+ \big[ A \overline{G}_{yy}^{-1}A^\top \big]^{-1} \big[A\overline{G}_{yy}^{-1} \overline{G}_{xy}- B\big] \right\|\nonumber\\
& \overset{(c)}{\leq} \left\|\big[A G_{yy}^{-1} G_{xy}- B\big]\right\| \left\|\big[ A G_{yy}^{-1}A^\top \big]^{-1}-\big[ A \overline{G}_{yy}^{-1}A^\top \big]^{-1} \right\| \nonumber\\
&+\left\|\big[ A \overline{G}_{yy}^{-1}A^\top \big]^{-1}\right\| \left\|\big[A G_{yy}^{-1} G_{xy}- B\big]-\big[A \overline{G}_{yy}^{-1} \overline{G}_{xy}- B\big] \right\| \nonumber\\
& \overset{(d)}{\leq} \left\|\big[A G_{yy}^{-1} G_{xy}- B\big]\right\| \left\|\big[ A G_{yy}^{-1}A^\top \big]^{-1}-\big[ A \overline{G}_{yy}^{-1}A^\top \big]^{-1} \right\| \nonumber\\
&+\left\|\big[ A \overline{G}_{yy}^{-1}A^\top \big]^{-1}\right\| 
\left\|A \right\| \left[ \left\|G_{yy}^{-1}\right\| \left\|G_{xy}-\overline{G}_{xy}\right\| + \left\|G_{xy}\right\| \left\|G_{yy}^{-1}-\overline{G}_{yy}^{-1} \right\| \right] \nonumber\\
& \overset{(e)}{\leq} \left(\frac{L_{A}\overline{L}_{g_{xy}}}{\mu_{g}} + L_{B} \right)L_{A}^{2}\overline{L}_{A}^{2}\frac{1}{\mu_{g}^{2}}L_{g_{yy}}\left[ \left\|\x-\xo\right\| + \left\|\y_{q}^{\ast}(\x)-\y_{q}^{\ast}(\xo)\right\| \right] \nonumber\\
&+ \overline{L}_{A}L_{A} \left[\frac{1}{\mu_{g}}L_{g_{xy}} + \overline{L}_{g_{xy}} \left(\frac{1}{\mu_{g}}\right)^{2}L_{g_{yy}} \right] \left[ \left\|\x-\xo\right\| + \left\|\y_{q}^{\ast}(\x)-\y_{q}^{\ast}(\xo)\right\| \right]\nonumber\\
&  \overset{(f)}{\leq} \left[\left(\frac{L_{A}\overline{L}_{g_{xy}}}{\mu_{g}} + L_{B} \right)L_{A}^{2}\overline{L}_{A}^{2}\frac{1}{\mu_{g}^{2}}L_{g_{yy}} + \overline{L}_{A}L_{A} \left(\frac{1}{\mu_{g}}L_{g_{xy}} + \overline{L}_{g_{xy}} \left(\frac{1}{\mu_{g}}\right)^{2}L_{g_{yy}} \right)\right] \cdot \nonumber\\
&\hspace{4mm}\cdot \left[ 1 + \overline{L}_{\y^{\ast}} \right] \left\|\x-\xo\right\|
\end{align}
In (a) we used the definition of $\nabla \overline{\lb}_q^{\ast}(\x)$; in (b) we add and subtract a term, and use the triangle inequality; in (c) we use the submultiplicative property of matrix norms; in (d) we add and subtract a single term and apply the triangle inequality; in (e) we used Lemmas \ref{lem:prelim1}, \ref{lem:lambda_y_bound} and Assumption \ref{ass:Fn_UL_LL}; in (f) we  applied Lemma \ref{lem:lambda_y_bound} where the boundedness of $\nabla \y_{q}^{\ast}(\x)$ implies the (local) Lipschitz continuity of $\y_{q}^{\ast}(\x)$. 

Then, by combining \eqref{eq:local_lip_y} and \eqref{eq:local_lip_lam} we can show that 
\begin{align}\label{eq::lip_grad_y_eq}
\left\|\nabla \y_q^{\ast}(\x)-\nabla \y_q^{\ast}(\xo) \right\| \leq \widetilde{L}_{\y^{\ast}} \| \x-\xo\|
\end{align}
The Lipschitz constant is defined as 
\begin{align}\label{eq:lip_y_const}
&\widetilde{L}_{\y^{\ast}} := \left[\frac{1}{\mu_{g}}L_{g_{xy}} + \overline{L}_{g_{xy}} \left(\frac{1}{\mu_{g}}\right)^{2}L_{g_{yy}} + \overline{L}_{\lambda^{\ast}}L_{A}\left(\frac{1}{\mu_{g}}\right)^{2}L_{g_{yy}}\right] 
\left[ 1 + \overline{L}_{\y^{\ast}} \right] \nonumber\\
&\hspace{-1mm}+\frac{L_{A}}{\mu_{g}} \left[\left(\frac{L_{A}\overline{L}_{g_{xy}}}{\mu_{g}} \hspace{-1mm}+\hspace{-1mm} L_{B} \right)L_{A}^{2}\overline{L}_{A}^{2}\frac{1}{\mu_{g}^{2}}L_{g_{yy}} + \overline{L}_{A}L_{A} \left(\frac{1}{\mu_{g}}L_{g_{xy}} + \overline{L}_{g_{xy}} \left(\frac{1}{\mu_{g}}\right)^{2}L_{g_{yy}} \right)\right]\left[ 1 + \overline{L}_{\y^{\ast}} \right].
\end{align}

Finally, we can show the Lipschitz smoothness of $\left\|\nabla F_q(\x)-\nabla F_q(\xo) \right\|$. 
\begin{align*}
&\left\|\nabla F_q(\x)-\nabla F_q(\xo) \right\|  \\
&\overset{(a)}{\leq} \left\|\nabla_{x} f(\x, \y_q^{\ast}(\x)) -\nabla_{x} f(\xo, \y_q^{\ast}(\xo)) \right\| \\
&+ \left\|[\nabla \y_q^{\ast}(\x)]^{T} \nabla_{y} f(\x, \y_q^{\ast}(\x))-
[\nabla \y_q^{\ast}(\xo)]^{T} \nabla_{y} f(\xo, \y_q^{\ast}(\xo))\right\| \\
&\overset{(b)}{\leq} L_{f} \left[ \left\|\x-\xo\right\| + \left\|\y_{q}^{\ast}(\x)-\y_{q}^{\ast}(\xo)\right\| \right] \\
&+\left\|\nabla_{y} f(\x, \y_q^{\ast}(\x)) \right\| \left\| [\nabla \y_q^{\ast}(\x)]^{T} -[\nabla \y_q^{\ast}(\xo)]^{T}\right\| \\
&+\left\|[\nabla \y_q^{\ast}(\xo)]^{T} \right\| \left\|\nabla_{y} f(\x, \y_q^{\ast}(\x)) -\nabla_{y} f(\xo, \y_q^{\ast}(\xo))\right\| \\
&\overset{(c)}{\leq} \left[L_{f} \left( 1 + \overline{L}_{\y^{\ast}} \right) \left\|\x-\xo\right\| +\overline{L}_{f} \widetilde{L}_{\y^{\ast}}  +\overline{L}_{\y^{\ast}} L_{f} \left( 1 + \overline{L}_{\y^{\ast}} \right) \right] \| \x-\xo\|
\end{align*}
In (a) we used the definition of $\nabla F_q(\x)$ and the triangle inequality; in (b) we used Assumption \ref{ass:Fn_UL_LL}, added and subtracted a term, and used the triangle inequality; in (c) we applied Assumption \ref{ass:Fn_UL_LL}, eq. \eqref{eq::lip_grad_y_eq} and Lemma \ref{lem:lambda_y_bound}.
We set 
\begin{align}\label{eq:lip_F_const}
    \widetilde{L}_{F_{q}} = L_{f} \left( 1 + \overline{L}_{\y^{\ast}} \right) \left\|\x-\xo\right\| +\overline{L}_{f} \widetilde{L}_{\y^{\ast}}  +\overline{L}_{\y^{\ast}} L_{f} \left( 1 + \overline{L}_{\y^{\ast}} \right).
\end{align}
The proof is complete.
\end{proof}

\subsection{Lipschitzness of the LL solution map}
{\begin{lem}
\label{lem: y_Lip}
Under Assumptions \ref{ass:basics}\ref{ass:diff}, \ref{ass:basics}\ref{ass:str_cvx_h}, \ref{ass:Fn_UL_LL} and \ref{ass: additional}, the mapping $\y^\ast(\x)$ defined for problem \eqref{eq: Problem_Bilevel} is Lipschitz continuous with constant $\frac{2 \overline{L}_g}{\mu_g}$ and mapping $\lb^\ast(\x)$ is Lipschitz continuous with constant $\frac{L_g}{\sigma_{\text{min}}(A)} \Big(1 + \frac{2 \overline{L}_g}{\mu_g} \Big)$. 
\end{lem}
\begin{proof}
Note that from the definition of $\y^\ast(\x)$, for $\x_1, \x_2 \in \du$ we have
\begin{align}
\label{eq: opt_g_1}
\big\langle \nabla_\y g(\x_1, \y^\ast(\x_1)), \overline{\y}_1 - \y^\ast(\x_1) \big\rangle & \geq 0 \qquad \forall~ \overline{\y}_1 \in \mathcal{Y}, \\
\label{eq: opt_g_2}
\big\langle \nabla_\y g(\x_2, \y^\ast(\x_2)), \overline{\y}_2 - \y^\ast(\x_2) \big\rangle & \geq 0 \qquad \forall~ \overline{\y}_2 \in \mathcal{Y},
\end{align}
where $\mathcal{Y} \coloneqq \{\y: A \y   \leq \bb \}$. 

Moreover, from the strong-convexity of $g(\x, \cdot)$, we have
\begin{align}
\label{eq: g_sc_1} g(\x_1, \y^\ast(\x_2)) & \geq g(\x_1, \y^\ast(\x_1)) + \big\langle \nabla_\y g(\x_1, \y^\ast(\x_1)), \y^\ast(\x_2) - \y^\ast(\x_1) \big\rangle \\
& \qquad \qquad \qquad \qquad  \qquad  + \frac{\mu_g}{2} \| \y^\ast(\x_1) - \y^\ast(\x_2) \|^2 \nonumber \\
\label{eq: g_sc_2}  g(\x_2, \y^\ast(\x_1)) & \geq g(\x_2, \y^\ast(\x_2)) + \big\langle \nabla_\y g(\x_2, \y^\ast(\x_2)), \y^\ast(\x_1) - \y^\ast(\x_2) \big\rangle \\
& \qquad \qquad \qquad \qquad  \qquad  + \frac{\mu_g}{2} \| \y^\ast(\x_1) - \y^\ast(\x_2) \|^2. \nonumber
\end{align}
Adding \eqref{eq: g_sc_1} and \eqref{eq: g_sc_2}, and rearranging the terms, we get
\begin{align*}
 \mu_g  \| \y^\ast(\x_1) - \y^\ast(\x_2) \|^2 & \overset{(a)}{\leq} \big[  g(\x_1, \y^\ast(\x_2)) -    g(\x_2, \y^\ast(\x_2)) \big] \\
 & \qquad \qquad   + \big[ g(\x_2, \y^\ast(\x_1)) -  \geq g(\x_1, \y^\ast(\x_1)) \big]   \\
 & \| \y^\ast(\x_1) - \y^\ast(\x_2) \|^2 \overset{(b)}{\leq} \frac{2 \overline{L}_g}{\mu_g} \|\x_1 - \x_2\|.
\end{align*}
where $(a)$ follows from the fact that $\y^\ast(\x_1), \y^\ast(\x_2) \in \mathcal{Y}$ and substituting $\overline{\y}_1 = \y^\ast(\x_2)$ in \eqref{eq: opt_g_1} and $\overline{\y}_2 =\y^\ast(\x_1)$ in \eqref{eq: opt_g_2}; and $(b)$ results from Assumption \ref{ass: additional}\ref{ass: additional_g_lip}.

Next, using the Lipschitzness of $\y^\ast(\cdot)$, we establish Lipschitzness of $\lb^\ast(\cdot)$. 

First, note that from Lipschitzness of mapping $\y^\ast(\x)$ combined with Assumptions \ref{ass:basics}\ref{ass:diff} and \ref{ass:Fn_UL_LL}\ref{ass:grad_f_bound} we  have that the implicit function $F(\x) \coloneqq f(\x, \y^\ast(\x))$ is also Lipschitz continuous. Next, from the KKT stationary condition of the LL problem \eqref{eq: Problem_Bilevel}, we have
\begin{align*}
    \nabla_{y} g(\x,\y^{\ast}(\x)) + {A}^\top  {\lb}^{\ast}(\x)   = 0 ~\Longrightarrow ~ {A}^\top  {\lb}^{\ast}(\x)  =  \nabla_{y} g(\x,\y^{\ast}(\x)),
\end{align*}
 Using the above we have for $\x_1, \x_2 \in \du$
 \begin{align*}
    {A}^\top  \big[ {\lb}^{\ast}(\x_1) -  {\lb}^{\ast}(\x_2) \big] =  \nabla_{y} g(\x_1,\y^{\ast}(\x_1)) - \nabla_{y} g(\x_2,\y^{\ast}(\x_2)),
\end{align*}
which implies that we have 
\begin{align*}
 \big\|   {A}^\top  \big[ {\lb}^{\ast}(\x_1) -  {\lb}^{\ast}(\x_2) \big] \big\| & =  \|\nabla_{y} g(\x_1,\y^{\ast}(\x_1)) - \nabla_{y} g(\x_2,\y^{\ast}(\x_2)) \| \\
\big\|      {\lb}^{\ast}(\x_1) -  {\lb}^{\ast}(\x_2)   \big\| & \overset{(a)}{\leq} \frac{1}{\sigma_{\text{min}}(A)}  \cdot \|\nabla_{y} g(\x_1,\y^{\ast}(\x_1)) - \nabla_{y} g(\x_2,\y^{\ast}(\x_2)) \| \\
 & \overset{(b)}{\leq} \frac{L_g}{\sigma_{\text{min}}(A)}  \cdot \|[\x_1 ; \y^{\ast}(\x_1))] - [\x_2;\y^{\ast}(\x_2)] \| \\
 & \overset{(c)}{\leq} \frac{L_g}{\sigma_{\text{min}}(A)} \bigg(1 + \frac{2 \overline{L}_g}{\mu_g} \bigg)  \cdot \|\x_1 - \x_2 \|,
\end{align*} 
where $(a)$ follows from the definition of $\sigma_{\text{min}}(A)$ which is the minimum singular value of the matrix $A$ and Assumption \ref{ass: additional}\ref{ass: additional_LL}; $(b)$ utilizes Assumption \ref{ass:Fn_UL_LL}\ref{ass:g_lip_grad}, and $(c)$ results from the Lipschitzness of $\y^\ast(\x)$.

Therefore, the map ${\lb}^\ast(\x)$ is also Lipschitz continuous. 
\end{proof}}
{
\subsection{The SC property of the LL Problem in \eqref{eq: Problem_Bilevel}}
\begin{lem}
\label{lem: SC}
    Under Assumptions \ref{ass:basics}-\ref{ass: additional}, given $\overline{\x} \in \du$, for a randomly generated $\x \in \du$ such that $\x \coloneqq \overline{\x} + \overline{\delta} \mathbf{u}$ for some $\overline{\delta} > 0$ and $\mathbf{u} \in \du$ uniformly randomly chosen over a unit ball centered at $\mathbf{0}$, the SC property is satisfied for the LL problem of the original bilevel problem \eqref{eq: Problem_Bilevel}. 
\end{lem}
\begin{proof}
First, let us consider the KKT conditions for the LL problem in \eqref{eq: Problem_Bilevel}. 
\begin{align}
   \label{eq: KKT_LL}
    \nabla_\y g(\x, \y^\ast(\x)) & + \sum_{j \in S(\y^\ast(\x))} \mu^\ast_j A_j^\top  = 0 \\
\nonumber  A_j \y^\ast(\x) & = b_j,  \mu^\ast_j \geq 0,~~\forall j \in S(\y^\ast(\x)) \\
 \nonumber A_j \y^\ast(\x) & < b_j,  \mu^\ast_j = 0,~~\forall j \notin S(\y^\ast(\x)) 
\end{align}
where the set $ S(\y^\ast(\x)) \subset \{1, \ldots, k\}$ denotes the set of active constraints at $\y^\ast(\x)$, $A_j$ refers to the $j^\text{th}$ row of the constraint matrix $A$, and $\mu^\ast_j$ refers to the dual variables. Using $\x \coloneqq \overline{\x} + \overline{\delta} \mathbf{u}$, and the definition of $\boldsymbol{\vartheta}$ from Assumption \ref{ass: additional}\ref{ass: Jacobian_LL}, from the first KKT condition in \eqref{eq: KKT_LL}, we get 
\begin{align}
 \nonumber \nabla_\y g (\overline{\x}  , \y^\ast(\x))    + \sum_{j \in S(\y^\ast(\x))} \mu^\ast_j A_j^\top  & = -  \nabla_\y g (\overline{\x}   + \overline{\delta} \mathbf{u}, \y^\ast(\x))  + \nabla_\y g (\overline{\x}  , \y^\ast(\x)) , \\
 \nabla_\y g (\overline{\x}  , \y^\ast(\x))    + \sum_{j \in S(\y^\ast(\x))} \mu^\ast_j A_j^\top  & = -  \boldsymbol{\vartheta},
    \label{eq: KKT_MVT}
\end{align}
where the first inequality follows from adding and subtracting $\nabla_\y g (\overline{\x}  , \y^\ast(\x))$ to the first KKT condition in \eqref{eq: KKT_LL}. 
Note from Assumption \ref{ass: additional}\ref{ass: Jacobian_LL} that the vector  $\boldsymbol{\vartheta}$ belongs to a continuous measure in $\dl$. 

Beyond this point, the proof follows a similar structure to the proof of \cite[Proposition 1]{lu2020finding}, where SC is established by contradiction. Specifically, without loss of generality, assuming that SC does not hold for the first constraint, i.e., we have $A_1\y^\ast(\x) = b_1$  and $\mu_1 = 0$. Using the same argument as \cite[Proposition 1]{lu2020finding}, we consider the map from the left hand side of \eqref{eq: KKT_MVT}
\begin{align*}
  \Phi(\y^\ast(\x),\boldsymbol{\mu}^\ast) =    \nabla_\y g (\overline{\x}, \y^\ast(\x))   + \sum_{j \in S(\y^\ast(\x))} \mu^\ast_j A_j^\top ,
\end{align*}
that maps the points of form $(\y^\ast(\x), \boldsymbol{\mu}^\ast)$ satisfying the KKT conditions with the first constraint satisfying $A_1\y^\ast(\x) = b_1$  and $\mu^\ast_1 = 0$ to a continuous measure in $d_\ell$ dimensional space since from Assumption \ref{ass: additional}\ref{ass: Jacobian_LL}, $\boldsymbol{\vartheta}$ is a continuous measure in $\dl$ space. Let us define the domain of $ \Phi$ as $\mathcal{T}$, then we have 
\begin{align*}
    \mathcal{T} & = \Big\{ (\y^\ast(\x), \boldsymbol{\mu}^\ast)~ \Big|~A_j \y^\ast(\x) = b_j, j \in S(\y^\ast(\x)), A_j\y^\ast(\x) < b_j, \mu_j^\ast = 0,  j \notin S(\y^\ast(\x)), \\
    & \quad \qquad  \qquad \qquad\qquad \qquad\qquad \qquad\qquad \qquad\qquad \quad  \boldsymbol{\mu}^\ast \geq 0, A_1 \y^\ast(\x) = 0, \mu_1 = 0 \Big\}
\end{align*}
Now let us quantify the dimension of $\mathcal{T}$. Recall that all $A_j$'s with $j \in S(\y^\ast(\x))$ are linearly independent since $\overline{A}$ is a full row rank matrix. We then have from \cite[Proof of Proposition 1]{lu2020finding} that the dimension of free space\footnote{For the definitions of the free and active space, please see \cite[Section 4]{lu2020finding}.} of $\y^\ast(\x)$ is the rank of $\overline{A}$, i.e., $d_\ell - |S(\y^\ast(\x))|$. Note that $|S(\y^\ast(\x))|$ is the number of active constraints, and we have assumed that $\mu^\ast_1 = 0$. This implies that the dimension of free space of vector $\boldsymbol{\mu}^\ast$ is $|S(\y^\ast(\x))| - 1$. Therefore, $\Phi$ is a mapping that maps elements from a $(d_\ell - 1)$-dimensional subspace to a $d_\ell$-dimensional space, implying that the image of the mapping is zero-measure in $\dl$. However, from Assumption \ref{ass: additional}\ref{ass: Jacobian_LL} we know that $\boldsymbol{\vartheta}$ on the right hand side of \eqref{eq: KKT_MVT} is a continuous measure in $\dl$ space which results in a contradiction of the assumption that the strict complementarity condition does not hold. Therefore, the SC property holds. 
\end{proof}
}
\subsubsection{Proof of Proposition \ref{prop:equivalence}} 
\label{app:equivalence}

\begin{proof}
{
Let us assume that the point $\overline{\x} \in \du$ is $(\epsilon, \overline{\delta})$-Goldstein stationary for the implicit function $\Fob(\cdot)$ defined for the perturbed problem \eqref{eq: Stochastic_Problem_Bilevel}. 

Then from  Lemma \ref{Thm:Main} we know that there exists a sequence of points $\{ \overline{\x}_i\}_{i = 1}^N$ for some $N \in \mathbb{N}$ with $\overline{\x}_i \in \mathbb{B}_{\overline{\delta}}(\overline{\x})$, and $\{\alpha_i\}_{i = 1}^N$ with $\sum_{i = 1}^N \alpha_i = 1$ such that we have 
\begin{align}
  \label{eq: Gold_Perturbed}  \bigg\| \sum_{i= 1}^N \alpha_i \nabla \Fob (\overline{\x}_i)  \bigg\| \leq \epsilon.
\end{align}
In the following, we will establish that \eqref{eq: Gold_Perturbed} implies that $\big\| \sum_{i= 1}^N \alpha_i \nabla F ({\x}_i)   \big\| \leq 3\epsilon$, where $F(\cdot)$ is the implicit function of the original bilevel problem in \eqref{eq: Problem_Bilevel} and where each $\x_i$ is chosen randomly from a neighborhood of $\overline{\x}_i$ for all $i \in [N]$. This implies that $\overline{\x}$ is, in fact, also an $(3 \epsilon, \widetilde{\delta})$-Goldstein stationary point for the original bilevel problem \eqref{eq: Problem_Bilevel} for some $\widetilde{\delta} > 0$ to be defined later. Next, we proceed with the proof. 

From Lemma \ref{lem: grad_y_lip}, we know that for each $\overline{\x}_i$ with $i \in [N]$, there exists a neighborhood $U(\overline{\x}_i)$ of $\overline{\x}_i$ such that the stochastic function $F_q(\cdot)$ is Lipschitz smooth with constant $\widetilde{L}_{F_{q}}$ within the set $U(\overline{\x}_i)$. This implies that corresponding to each $\overline{\x}_i$ there exists a randomly chosen point $\x_i \in U(\overline{\x}_i)$ such that $\|\x_i - \overline{\x}_i \| \leq \epsilon/\widetilde{L}_{F_{q}}$, which implies that, we have
\begin{align}
  \label{eq: Lip_Smoothness_Fq}  \big\|   \nabla F_q(\overline{\x}_i) - \nabla F_q(\x_i) \big\| \leq \epsilon~~\text{for all $i \in [N]$ from Lemma \ref{lem: grad_y_lip}.}
\end{align}
Moreover, using Rademacher's theorem  \cite{Rudin_Book}, we know that the maps $\y^\ast(\x)$, $F(\x)$, and ${\lb}^\ast(\x)$ are differentiable almost everywhere since they are Lipschtiz continuous (see Lemma \ref{lem: y_Lip}).  Therefore, for $\x_i$ chosen randomly the mappings $\y^\ast(\x_i)$, $\overline{\lb}^\ast(\x_i)$ are differentiable and $\nabla F(\x_i)$ exists almost surely for all $i \in [N]$. Therefore, we evaluate
\begin{align}
   \nonumber  \bigg\| \sum_{i = 1}^N \alpha_i \nabla F(\x_i) \bigg\| & \overset{(a)}{\leq}   \bigg\| \sum_{i = 1}^N \alpha_i \big[ \nabla F(\x_i) - \nabla \Fob(\overline{\x}_i) \big] \bigg\| +   \bigg\| \sum_{i = 1}^N \alpha_i  \nabla \Fob(\overline{\x}_i)   \bigg\| \\
   \nonumber  & \overset{(b)}{\leq} \sum_{i = 1}^N \alpha_i  \big\|   \nabla F(\x_i) -  \Ebb_{\qb} [\nabla F_q(\overline{\x}_i)]  \big\| +  \epsilon\\
   \nonumber  & \overset{(c)}{\leq} \sum_{i = 1}^N \alpha_i  \Ebb_{\qb}  \big\|   \nabla F(\x_i) -  \nabla F_q(\overline{\x}_i)  \big\| +  \epsilon\\
   \nonumber  & \overset{(d)}{\leq} \sum_{i = 1}^N \alpha_i  \Ebb_{\qb}  \big\|   \nabla F(\x_i) - \nabla F_q(\x_i) \big\|  \\
  \nonumber  & \qquad \qquad \qquad \qquad + \sum_{i = 1}^N \alpha_i  \Ebb_{\qb} \big\| \nabla F_q(\x_i) - \nabla F_q(\overline{\x}_i)  \big\| + \epsilon \\
     & \overset{(e)}{\leq} \sum_{i = 1}^N \alpha_i  \Ebb_{\qb}  \big\|   \nabla F(\x_i) - \nabla F_q(\x_i) \big\|  +   2  \epsilon
     \label{eq: Norm_Grad_F}
\end{align}
where $(a)$ follows from the application of Triangle inequality; $(b)$ uses the first equality in Proposition \ref{pro:diff} and the fact that $\big\| \sum_{i= 1}^N \alpha_i \nabla \Fob (\overline{\x}_i)  \big\| \leq \epsilon$ from \eqref{eq: Gold_Perturbed}; $(c)$ utilizes Jensen's inequality; $(d)$ results from Triangle inequality, and $(e)$ results from \eqref{eq: Lip_Smoothness_Fq}.

Next, we consider the term $\big\| \nabla F(\x_i) - \nabla F_q({\x}_i)  \big\|$ in \eqref{eq: Norm_Grad_F}. Let us bound this term for a generic $\x \in \du$ randomly selected in the neighborhood of a corresponding point $\overline{\x} \in \du$. Note that the same result will hold for all $\x_i$'s for $i \in [N]$ under the same set of assumptions.
\begin{align}
    \|\nabla F(\x) - 
 \nabla F_q(\x)  \| & =    \big\|\nabla_\x f(\x ,\y^\ast(\x)) + [\nabla \y^\ast(\x)]^\top \nabla_\y f(\x ,\y^\ast(\x)) \nonumber \\
& \qquad \qquad \quad - 
\nabla_\x f(\x ,\y_q^\ast(\x)) + [\nabla \y_q^\ast(\x)]^\top \nabla_\y f(\x ,\y_q^\ast(\x))  \big\| \nonumber  \\
& \overset{(a)}{\leq}    \big\|\nabla_\x f(\x ,\y^\ast(\x)) - \nabla_\x f(\x ,\y_q^\ast(\x)) \| \nonumber \\
&  +  \big\|  [\nabla \y^\ast(\x)]^\top \nabla_\y f(\x ,\y^\ast(\x))  -
  [\nabla \y_q^\ast(\x)]^\top \nabla_\y f(\x ,\y_q^\ast(\x))  \big\| \nonumber  \\
  & \overset{(b)}{\leq} L_f \|\y(\x) - \y_q(\x) \| \nonumber \\
   & + \|[\nabla \y_q^\ast(\x)]^\top \nabla_\y f(\x ,\y^\ast(\x)) - [\nabla \y_q^\ast(\x)]^\top \nabla_\y f(\x ,\y_q^\ast(\x))   \|  \nonumber \\
  & + \|  [\nabla \y^\ast(\x)]^\top \nabla_\y f(\x ,\y^\ast(\x)) - [\nabla \y_q^\ast(\x)]^\top \nabla_\y f(\x ,\y^\ast(\x)) \| \nonumber 
 \\
  & \overset{(c)}{\leq} (L_f + \overline{L}_{\y^\ast}   L_f) \|\y^\ast(\x) - \y^\ast_q(\x) \|  + \overline{L}_f \|\nabla \y^\ast(\x) - \nabla \y_q^\ast(\x) \| \nonumber  \\
    & \overset{(d)}{\leq} \frac{(L_f + \overline{L}_{\y^\ast}   L_f)}{\mu_g}   \| \qb \|   + \overline{L}_f \|\nabla \y^\ast(\x) - \nabla \y_q^\ast(\x) \|. 
    \label{eq: Norm_Grad_Diff_F}
\end{align}
where $(a)$ results from the application of the Triangle inequality; $(b)$ uses the Triangle inequality and Assumption \ref{ass:Fn_UL_LL}\ref{ass:grad_f_lip}; $(c)$ again utilizes Assumption \ref{ass:Fn_UL_LL}\ref{ass:grad_f_lip} and Lemma \ref{Lem: Bounded_Grad_Y}; $(d)$ utilizes the result derived in \cite[Appendix D.2.2]{khanduri2023linearly}. 

The above inequality implies that for a differentiable point $\x \in \du$, the distance between the gradient of the original implicit function, $F(\cdot)$ in \eqref{eq: Problem_Bilevel} and the perturbed stochastic objective $F_q(\cdot)$ in \eqref{eq: Stochastic_Problem_Bilevel} depends on the size of the perturbation $\| \qb\|$ and $\|\nabla \y^\ast(\x) - \nabla \y_q^\ast(\x) \|$. Note that we can choose the perturbation $\qb \sim \mathcal{Q}$ with a small enough support so that $\|\qb\|$ is bounded. Therefore, to bound $\|\nabla F(\x) - \nabla F_q(\x)  \|$ we need to bound $ \|\nabla \y^\ast(\x) - \nabla \y_q^\ast(\x) \|$ in \eqref{eq: Norm_Grad_Diff_F}. 
 
Next, we bound the term $ \|\nabla \y^\ast(\x) - \nabla \y_q^\ast(\x) \|$ in \eqref{eq: Norm_Grad_Diff_F}. 

Recall from Lemma \ref{lem: SC} that given a point $\overline{\x} \in \du$,  SC is satisfied for $\y^\ast(\x)$ where $\x \in \du$ is a random point evaluated in the neighborhood of $\overline{\x}$. First, note that for $\qb = 0$, under SC of $\y^\ast(\x)$ and the set of Assumptions \ref{ass:basics}-\ref{ass: additional}, the conditions of \cite[Theorem 2.22]{Friesz_Foundations_2016} are satisfied for the LL problem in \eqref{eq:sbp_ll}. Then from \cite[Statement (3) of Theorem 2.22]{Friesz_Foundations_2016}, it follows that the set of active constraints in the neighborhood of $\qb = 0$ is unchanged and SC holds within this neighborhood. Let us define this neighborhood of $\qb = 0$ as $U_q$.  
This implies that there exists a distribution $\mathcal{Q}$ with support small enough such that for every $\qb \sim \mathcal{Q}$ we have $\qb \in U_q$ and the LL solutions $\y^\ast(\x)$ and $\y_q^\ast(\x)$ share the same set of active constraints. 

Therefore, the reduced KKT system for $\y^\ast(\x)$ will be similar to $\y_q^\ast(\x)$ and can be written as
\begin{align} 
    &\nabla_{y} g(\x,\y^{\ast}(\x)) + \overline{A}^\top \overline{\lb}^{\ast}(\x)   = 0, 
    \quad \overline{A} \y^{\ast}(\x) + \overline{B} \x - \overline{\bb} = 0.   
    \label{eq: KKT_Original}
\end{align}
To differentiate the above set of equations in order to obtain a closed form expression for $\nabla \y^{\ast}(\x)$, we need to ensure differentiability of $\y^{\ast}(\x)$ and $\overline{\lb}^{\ast}(\x)$. Recall, from the discussion after \eqref{eq: Lip_Smoothness_Fq} that differentiability of $\y^\ast(\x)$ and $\overline{\lb}^\ast(\x)$ at a randomly chosen $\x \in \du$ follow from the Rademacher's theorem  \cite{Rudin_Book}.  
Now, utilizing the fact that that $\y^\ast(\x)$ and $\overline{\lb}^\ast(\x)$ are differentiable, i.e., $\nabla \y^\ast(\x)$ and $\nabla \overline{\lb}^\ast(\x)$ exists, we can derive a closed-form expression for $\nabla \y^\ast(\x)$ in the same manner as we did for $\nabla \y_q^\ast(\x)$ in \eqref{eq:grad_yast1} of Proposition \ref{pro:diff}. More precisely, following the derivation of Proposition \ref{pro:diff} presented in Section \ref{app:proofs_sec2}, we have}
\begin{align*}
    \nabla \y^\ast(\x) =  \big[ \nabla_{yy}^2 g(\x,\y^{\ast}(\x)) \big]^{-1} \big[-\nabla_{xy}^2 g(\x,\y^{\ast}(\x)) - \overline{A}^{T} \nabla \overline{\lb}^{\ast}(\x)\big],
\end{align*}
where 
\begin{align*} 
     \nabla \overline{\lb}^{\ast}(\x) & =   -   \big[ \overline{A}\big[\nabla_{yy}^{2}g(\x,\y^{\ast}(\x)) \big]^{-1}\overline{A}^\top \big]^{-1}     \big[\overline{A}\big[\nabla_{yy}^2 g(\x,\y^{\ast}(\x))\big]^{-1} \nabla_{xy}^2 g(\x,\y^{\ast}(\x)) \big],  
\end{align*}
Next, utilizing the expressions of $\nabla \y_{q}^{\ast}(\x)$ and $\nabla \y^{\ast}(\x)$, we get 
\begin{align*}
&\left\|  \nabla \y_{q}^{\ast}(\x) -  \nabla \y^{\ast}(\x)  \right\| \\
&\stackrel{(a)}{=} \big\|  \big[ \nabla_{yy}^2 g(\x,\y_{q}^{\ast}(\x)) \big]^{-1} \big[-\nabla_{xy}^2 g(\x,\y_{q}^{\ast}(\x)) - \overline{A}^{T} \nabla \overline{\lb}_{q}^{\ast}(\x)\big]  \\
& \qquad \qquad -\big[ \nabla_{yy}^2 g(\x,\y^{\ast}(\x)) \big]^{-1}\big[-\nabla_{xy}^2 g(\x,\y^{\ast}(\x)) -\overline{A}^{T} \nabla \overline{\lb}^{\ast}(\x)\big]  \big\| \\
&\leq \big\|  \big[ \nabla_{yy}^2 g(\x,\y_{q}^{\ast}(\x)) \big]^{-1} \nabla_{xy}^2 g(\x,\y_{q}^{\ast}(\x)) - \big[\nabla_{yy}^2 g(\x,\y^{\ast}(\x)) \big]^{-1} \nabla_{xy}^2 g(\x,\y^{\ast}(\x)) \big\| \nonumber\\
&\qquad \qquad \quad + \big\|  \big[ \nabla_{yy}^2 g(\x,\y_{q}^{\ast}(\x)) \big]^{-1} \big[ \overline{A}^{T} \nabla \overline{\lb}_{q}^{\ast}(\x)\big] -\big[ \nabla_{yy}^2 g(\x,\y^{\ast}(\x)) \big]^{-1}\big[\overline{A}^{T} \nabla \overline{\lb}^{\ast}(\x)\big] \big\| \nonumber \\
&\stackrel{(b)}{\leq} \big \|\big[ \nabla_{yy}^2 g(\x,\y_{q}^{\ast}(\x)) \big]^{-1}  -  \big[ \nabla_{yy}^2 g(\x,\y^{\ast}(\x)) \big]^{-1} \big\| 
\big\| \nabla_{xy}^2 g(\x,\y_{q}^{\ast}(\x)) \big\| \nonumber \\
& \qquad + \big\|\big[ \nabla_{yy}^2 g(\x,\y^{\ast}(\x)) \big]^{-1} \big\| 
\left\|\nabla_{xy}^2 g(\x)(\x,\y_{q}^{\ast}(\x))  -  \nabla_{xy}^2 g(\x,\y^{\ast}(\x))\right\| \nonumber \\
&\qquad \qquad+ \left\|\big[ \nabla_{yy}^2 g(\x,\y_{q}^{\ast}(\x)) \big]^{-1}  -  \big[ \nabla_{yy}^2 g(\x,\y^{\ast}(\x)) \big]^{-1} \right\| 
\left\|  \overline{A}^{T} \nabla \overline{\lb}_{q}^{\ast}(\x)\right\| \nonumber\\
&\qquad \qquad \qquad + \big\| \big[ \nabla_{yy}^2 g(\x,\y^{\ast}(\x)) \big]^{-1}  \overline{A}^{T} \big\| 
\big\| \nabla \overline{\lb}_{q}^{\ast}(\x) -  \nabla \overline{\lb}^{\ast}(\x) \big\|.
\end{align*}
In (a) we used the gradient formula of Proposition \ref{pro:diff}. In (b) we added and subtracted certain terms, applied the triangle inequality, and finally used the submultiplicative property of matrix norms. 
 By applying Lemmas \ref{lem:prelim1}, \ref{lem:lambda_y_bound} and Assumption \ref{ass:Fn_UL_LL} to the above derivation, we obtain the following:  
\begin{align}\label{eq:nabla_y_diff}
& \|  \nabla \y_{q}^{\ast}(\x) -  \nabla \y^{\ast}(\x)   \| \nonumber\\ 
& \quad \leq
\left(\frac{1}{\mu_{g}} \right)^{2} L_{g_{yy}} \overline{L}_{g_{xy}}  \|\y_{q}^{\ast}(\x) - \y^{\ast}(\x)\|
+ \frac{1}{\mu_{g}} L_{g_{xy}}  \|\y_{q}^{\ast}(\x) - \y^{\ast}(\x)\| \nonumber \\
& \quad \qquad + \left( \frac{1}{\mu_{g}} \right)^{2} L_{g_{yy}}  L_{A} \overline{L}_{\lb^{\ast}} \|\y_{q}^{\ast}(\x) - \y^{\ast}(\x)\|
+ \frac{1}{\mu_{g}} L_{A} \big\| \nabla \overline{\lb}_{q}^{\ast}(\x) -  \nabla \overline{\lb}^{\ast}(\x)\big\|.
\end{align}

To proceed, it is necessary to bound the term $\| \nabla \overline{\lb}_{q}^{\ast}(\x) -  \nabla \overline{\lb}^{\ast}(\x)\|$, as for the term $\|\y_{q}^{\ast}(\x) - \y^{\ast}(\x)\|$, we have $\|\y_{q}^{\ast}(\x) - \y^{\ast}(\x)\| \leq \| \qb\|/\mu_g$ \cite{khanduri2023linearly}. Towards this end, we define the following notation to simplify our calculations. 
\begin{align*}
&H_q = \left[ \overline{A} \left[ \nabla_{yy}^2 g(\x,\y_{q}^{\ast}(\x))\right]^{-1} \overline{A}^{T} \right]\hspace{-1mm},
G_g = \left[ \nabla_{yy}^2 g(\x,\y_{q}^{\ast}(\x)) \right]^{-1}\hspace{-1mm}, M_q = \nabla_{xy}^2 g(\x,\y_{q}^{\ast}(\x)) \\
& {H} = \left[ \overline{A} \left[ \nabla_{yy}^2 g(\x,\y^{\ast}(\x))\right]^{-1} \overline{A}^{T} \right]\hspace{-1mm},
 {G} = \big[ \nabla_{yy}^2 g(\x,\y^{\ast}(\x)) \big]^{-1}\hspace{-1mm},
 {M} = \nabla_{xy}^2 g(\x,\y^{\ast}(\x))
\end{align*}
Then, we note that the expression $\|\nabla \overline{\lb}^{\ast}(\x) - \widehat{\nabla}  \overline{\lb}^{\ast}(\x) \|$ has the same exact form as the one in \eqref{eq:nabla_lambda_diff}. Therefore, we can immediately conclude that
\begin{align}\label{eq:nabla_lamda_diff2}
    &\big\| \nabla \overline{\lb}_{q}^{\ast}(\x) -  \nabla \overline{\lb}^{\ast}(\x)\big\| \nonumber\\
     &\leq \|H_q^{-1}- {H}^{-1} \| \left[\| \overline{A} \|  \|  G_q  \|  \| M_q \| -\|B\|\right]\nonumber\\
    & \qquad + \| {H}^{-1}  \|  \| \overline{A}  \|  \| G_q \|  \|  M_q - {M}  \|   + 
   \| {H}^{-1} \|  \| \overline{A}  \|  \| G_q -  {G}  \|  \| {M}  \|  \nonumber\\
    &\leq \hspace{-1mm} \left(\overline{L}_{A}^{2} L_{A}^{2} \left( \frac{1}{\mu_{g}} \right)^{2} L_{g_{yy}} \left(L_{A} \frac{\overline{L}_{g_{xy}}}{\mu_{g}}  +L_{B}\right) \hspace{-1mm}+\hspace{-1mm} 
    \frac{1}{\mu_{g}} \overline{L}_{A} L_{A}  L_{g_{xy}} \hspace{-1mm}+\hspace{-1mm} \left(\frac{1}{\mu_{g}}\right)^{2} \hspace{-1mm}\overline{L}_{A} L_{A}  L_{g_{yy}} \overline{L}_{g_{xy}}  \right)  \nonumber\\
    &\qquad \qquad \cdot \|  \y_{q}^{\ast}(\x) - \y^{\ast}(\x)  \| \nonumber\\
    & \qquad \qquad = L_{\lb^{\ast}} \|  \y_{q}^{\ast}(\x) - \y^{\ast}(\x)   \|.
\end{align}
 In the second inequality we use Lemma \ref{lem:prelim1}\ref{lem:prelim1d}
 for $ \|H^{-1}-\widehat{H}^{-1} \|$, the bound $\|\overline{A}\| \leq L_{A}$, Lemma \ref{lem:prelim1}\ref{lem:prelim1a} for $\| G\|$, Lemma \ref{lem:prelim1}\ref{lem:prelim1c} for $\|H^{-1}\|$ and $\|\widehat{H}^{-1}\|$, Assumption \ref{ass:Fn_UL_LL}\ref{ass:Hes_bound_gxy} for $\| M\|$ and $\| \widehat{M}\|$, Assumption \ref{ass:Fn_UL_LL}\ref{ass:Hes_lip_gxy} for $\|M - \widehat{M} \|$,  and finally Lemma \ref{lem:prelim1}\ref{lem:prelim1b} for $\| G - \widehat{G} \|$. 

Then combining \eqref{eq:nabla_y_diff} with \eqref{eq:nabla_lamda_diff2} results in
\begin{align}\label{eq:nabla_y_diff_bound}
    &\left\|  \nabla \y_{q}^{\ast}(\x) -  \nabla \y^{\ast}(\x)  \right\| \nonumber \\
    & \qquad \leq  
    \left[\left(\frac{1}{\mu_{g}} \right)^{2} \hspace{-1mm} L_{g_{yy}} \overline{L}_{g_{xy}} + \frac{1}{\mu_{g}} L_{g_{xy}}
    \hspace{-1mm}+\hspace{-1mm} \left( \frac{1}{\mu_{g}} \right)^{2} L_{g_{yy}}  L_{A} \overline{L}_{\lb^{\ast}}
    \hspace{-1mm}+\hspace{-1mm} \frac{1}{\mu_{g}} L_{A} L_{\lb^{\ast}}\right]  \nonumber\\
    & \qquad \qquad \qquad \qquad\qquad \qquad \cdot  \left\|  \y_{q}^{\ast}(\x) - \y^{\ast}(\x)  \right\|. \nonumber \\
 &   \leq  
    \left[\left(\frac{1}{\mu_{g}} \right)^{2} \hspace{-1mm} L_{g_{yy}} \overline{L}_{g_{xy}} + \frac{1}{\mu_{g}} L_{g_{xy}}
    \hspace{-1mm}+\hspace{-1mm} \left( \frac{1}{\mu_{g}} \right)^{2} L_{g_{yy}}  L_{A} \overline{L}_{\lb^{\ast}}
    \hspace{-1mm}+\hspace{-1mm} \frac{1}{\mu_{g}} L_{A} L_{\lb^{\ast}}\right] \cdot  \frac{\|\qb \|}{\mu_g}.
\end{align}
Substituting in the expression of $ \|\nabla F(\x)- 
 \nabla F_q(\x)  \|$ in \eqref{eq: Norm_Grad_Diff_F}, we get
 \begin{align*}
\|\nabla F(\x) - 
 \nabla F_q(\x)  \| \leq \mathcal{C}_{\qb} \cdot  \| \qb \|, 
 \end{align*}
where 
$$\mathcal{C}_{\qb} \coloneqq \bigg[ \frac{(L_f + \overline{L}_{\y^\ast}   L_f)}{\mu_g}  + \overline{L}_f \bigg[ \frac{L_{g_{yy}} \overline{L}_{g_{xy}}}{\mu_{g}^2}   \hspace{-1mm}  + \frac{L_{g_{xy}}}{\mu_{g}} 
    \hspace{-1mm}+\hspace{-1mm}   \frac{L_{g_{yy}}  L_{A} \overline{L}_{\lb^{\ast}}}{\mu_{g}^2}   
    \hspace{-1mm}+\hspace{-1mm} \frac{ L_{A} L_{\lb^{\ast}}}{\mu_{g}}\bigg] \bigg].  $$
Next, we choose $\qb \sim \mathcal{Q}$ s.t. 
 $$\| \qb \| \leq   \frac{\epsilon}{ \mathcal{C}_\qb} ~\Longrightarrow~\|\nabla F(\x) - 
 \nabla F_q(\x)  \| \leq  {\epsilon} .$$
{Recall that in addition to $\| \qb \| \leq   \frac{\epsilon}{ \mathcal{C}_\qb}$ above we also require $\qb \in U_q$ (see discussion before \eqref{eq: KKT_Original}).} Finally, substituting everything in the bound for $ \big\| \sum_{i = 1}^N \alpha_i \nabla F(\x_i) \big\|$ in \eqref{eq: Norm_Grad_F} and utilizing the above result for individual $\x_i$'s for $i \in [N]$, we get

\begin{align*}
 \bigg\| \sum_{i = 1}^N \alpha_i \nabla F(\x_i) \bigg\| \leq 3 \epsilon. 
\end{align*}
This further implies that $\overline{\x}$ will be a $(3\epsilon, \overline{\delta} +  {\epsilon}/{\widetilde{L}_{F_q}})$-Goldstein stationary solution for the original problem \eqref{eq: Problem_Bilevel} since $\x_i$'s are at most $\overline{\delta} + \epsilon/\widetilde{L}_{F_q}$ distance away from the original $(\epsilon, \overline{\delta})$-Goldstein stationary point $\overline{\x}$ of the perturbed problem \eqref{eq: Stochastic_Problem_Bilevel}.

Hence, the proof is complete. 
\end{proof}

\subsection{Proofs of Section \ref{sec:Algorithm}}
\begin{theorem}[Convergence of \algo]
\label{App:Thm:Main}
Under Assumptions \ref{ass:basics}-\ref{ass:Approx_y} the iterates generated by Algorithm \ref{Algo: DS-BLO} for both {\bf Options I} and {\bf II} satisfies
    \begin{align*}
        \frac{1}{T} \sum_{t = 1}^T   \Ebb \bigg\|  \sum_{i = t - K + 1}^t  \alpha_i    {\gr}  \Fob(\overline{\x}_i)   \bigg\|  \leq    \epsilon  ~\text{for}~ \widetilde{\mathcal{O}} \bigg\{ \frac{1}{\overline{\delta} \cdot \epsilon^4} \max\big\{(\delta_\vb + 2 \overline{L}_F)  \Delta ,  (\delta_\vb^2 + 2\overline{L}_F^2)^2\big\}  \bigg\},
    \end{align*}
where $\alpha_i = \frac{\beta^{t - i} (1 - \beta)}{1 - \beta^K}$, $K =\frac{1}{\ln \frac{1}{\beta}} \ln \frac{32 \big( \delta_\y + \delta_\vb + \overline{L}_F \big)  }{\epsilon}$, and $\Delta = \Fob(\x_1) - \Fob^\ast$. 
\end{theorem}
\begin{proof}
From the definition of $\m_{t
+1} = \beta \m_t + (1 - \beta) \g_{t + 1}$, we have
\begin{align}
   \| \m_{t+1}\|^2  =  \beta^2 \| \m_t\|^2 + 2 \beta (1 - \beta) \langle \g_{t+1}, \m_t \rangle + (1 - \beta)^2 \| \g_{t+1}\|^2  
\end{align}
Multiplying the step-size $\eta_t$ on both sides, taking expectation, and rearranging the terms, we get the following
\begin{align}
   &  \Ebb [\eta_t   \| \m_{t+1}\|^2  - \beta^2 \eta_t \| \m_t\|^2]  \nonumber\\
    &  \qquad  \qquad \qquad = 2 \beta (1 - \beta) \Ebb [\langle \g_{t+1}, \eta_t \m_t \rangle] + (1 - \beta)^2 \Ebb[\eta_t \| \g_{t+1}\|^2]
    \label{Eq: Difference_mt}
\end{align}
Next, we define 
\begin{align}
    T_1 & \coloneqq  \Ebb [\langle \g_{t+1}, \eta_t \m_t \rangle] \label{Eq:T1}\\
    T_2 & \coloneqq  \Ebb[\eta_t \| \g_{t+1}\|^2] \label{Eq:T2}
\end{align}
Next, we bound each term above separately. First, let us consider $T_2$ in \eqref{Eq:T2}, we have
\begin{align}
 T_2 & \coloneqq    \Ebb[\eta_t \| \g_{t+1}\|^2] \nonumber  \\
 & \overset{(a)}{=} \Ebb[\eta_t \| \widehat{\gr} F_{q_{t+1}}(\overline{\x}_{t+1}) \|^2] \nonumber  \\
 & \overset{(b)}{\leq} \frac{1}{\gamma_2}\Ebb\| \widehat{\gr} F_{q_{t+1}}(\overline{\x}_{t+1}) \|^2 \nonumber \\
 & \overset{(c)}{\leq} \frac{3}{\gamma_2}\Ebb \| \widehat{\gr} F_{q_{t+1}}(\overline{\x}_{t+1}) -  {\gr} F_{q_{t+1}}(\overline{\x}_{t+1}) \|^2 \nonumber \\
& \qquad + \frac{3}{\gamma_2}\Ebb \| {\gr} F_{q_{t+1}}(\overline{\x}_{t+1}) - {\gr} \Fob_{q_{t+1}}(\overline{\x}_{t+1}) \|^2 
  + \frac{3}{\gamma_2}\Ebb \| {\gr} \Fob_{q_{t+1}}(\overline{\x}_{t+1})\|^2 \nonumber  \\
  & \overset{(d)}{\leq}  \frac{3}{\gamma_2} \delta_\y^2  + \frac{3}{\gamma_2} \delta_\vb^2  + \frac{3}{\gamma_2} \overline{L}_F^2     
  \label{eq: Bound_T2}
\end{align}
where $(a)$ above follows from the definition of $\g_{t+1}$; $(b)$ utilizes the fact that by definition of the step-size $\eta_t$, we have $\eta_t = \frac{1}{\gamma_1 \|\m_t \| + \gamma_2} \leq \frac{1}{\gamma_2}$; $(c)$ utilizes $\| \sum_{k = 1}^K  \ab_k \|^2 \leq K \sum_{k =  1}^K \|\ab_k \|^2$ for a set of vectors $\{ \ab_k\}_{k =1}^K \in \du $; and $(d)$ results from the application of Lemmas \ref{lem:F} and \ref{lem:F_Variance}.

Next, let us consider the term $T_1$ in \eqref{Eq:T1}, we have
\begin{align}
     T_1 & \coloneqq  \Ebb [\langle \g_{t+1}, \eta_t \m_t \rangle]  \nonumber\\
    & \overset{(a)}{=} \Ebb \langle \widehat{\gr} F_{q_{t+1}}(\overline{\x}_{t+1}), \eta_t \m_t \rangle \nonumber\\
     & \overset{(b)}{\leq} \Ebb \big[\| \widehat{\gr} F_{q_{t+1}}(\overline{\x}_{t+1}) -  {\gr} F_{q_{t+1}}(\overline{\x}_{t+1}) \|  \eta_t  \|\m _t\| \big] +
     \Ebb \langle  {\gr} F_{q_{t+1}}(\overline{\x}_{t+1}), \eta_t \m_t \rangle \nonumber \\
      & \overset{(c)}{\leq} \frac{\delta_\y}{\gamma_1} -
     \Ebb \langle  {\gr} F_{q_{t+1}}(\overline{\x}_{t+1}), \x_{t+1} - \x_t \rangle
     \label{eq: Bound_T1}
\end{align}
where $(a)$ follows from the definition of $\g_{t+1}$; $(b)$ results by adding and subtracting ${\gr} F_{q_{t+1}}(\overline{\x}_{t+1})$ to the left term in the inner-product, utilizing the linearity of inner-product, and the Cauchy-Schwarz inequality; and $(c)$ utilizes Lemma \ref{lem:F}, the fact that $\eta_t \|\m_t \| = \frac{1}{\gamma_1 \|\m_t \| + \gamma_2} \| \m_t\| \leq \frac{1}{\gamma_1}$, and the update rule $\x_{t+1} = \x_t - \eta_t \m_t$.

We next define $T_{11}$ as
\begin{align}
\label{Eq:T11}
    T_{11} \coloneqq  \Ebb \langle  {\gr} F_{q_{t+1}}(\overline{\x}_{t+1}), \x_{t+1} - \x_t \rangle.
\end{align}
Before, bounding $T_{11}$, we define some notations. We define the sigma-algebra generated by the sequence of random variables as $\mathcal{G}_t \coloneqq \sigma(\g_1, \ldots, \g_t) $ and $\XCalOb_{t+1} \coloneqq \sigma(\g_1, \ldots, \g_t, \overline{\x}_{t+1})$.
Next, we bound term $T_{11}$ in the following 
\begin{align}
      T_{11} & \coloneqq  \Ebb \langle  {\gr} F_{q_{t+1}}(\overline{\x}_{t+1}), \x_{t+1} - \x_t \rangle \nonumber \\
     &  = \Ebb \bigg[    \Ebb_{\{\overline{\x}_{t+1}, q_{t+1}\}} \Big[ \Big\langle  {\gr} F_{q_{t+1}}(\overline{\x}_{t+1}), \x_{t+1} - \x_t \Big\rangle \Big| \mathcal{G}_t \Big] \bigg] \nonumber \\
      & =  \Ebb \bigg[ \Ebb_{\{\overline{\x}_{t+1}\} } 
      \Big[ \Ebb_{\{q_{t+1}\}} \Big[ \Big\langle  {\gr} F_{q_{t+1}}(\overline{\x}_{t+1}), \x_{t+1} - \x_t \Big\rangle \Big| \XCalOb_{t+1}\Big] \Big| \mathcal{G}_t \Big] \bigg] \nonumber\\
      & \overset{(a)}{=} \Ebb \bigg[ \Ebb_{\{\overline{\x}_{t+1}\} }
      \Big[   \Big\langle  {\gr} \Fob(\overline{\x}_{t+1}), \x_{t+1} - \x_t \Big\rangle \Big|   \mathcal{G}_t \Big] \bigg] \nonumber \\
      & \overset{(b)}{=} \Ebb \big[ \Fob(\x_{t+1}) - \Fob(\x_t) \big]
      \label{Eq:T11_Bound}
\end{align}
where the equality $(a)$ follows from Proposition \ref{pro:diff}; and $(b)$ follows from the application of the fundamental theorem of calculus as: Denote $\boldsymbol{\alpha}(\lambda) \coloneqq (1 - \lambda) \x_t + \lambda \x_{t+1}$
\begin{align*}
\Ebb_{\{ \overline{\x}_{t+1}\}}\Big[   \Big\langle  {\gr} \Fob(\overline{\x}_{t+1}), \x_{t+1} - \x_t \Big\rangle \Big|   \mathcal{G}_t \Big] & \overset{(b_1)}{=}    \int_{0}^1 \Fob'(\boldsymbol{\alpha}(\lambda)) \mathrm{d}\lambda   \overset{(b_2)}{=} \Fob(\x_{t+1}) - \Fob(\x_t)     
\end{align*}
where $(b_1)$ above follows from the fact that $\Fob'(\boldsymbol{\alpha}(\lambda)) = \langle  {\gr} \Fob(\overline{\x}_{t+1}), \x_{t+1} - \x_t \rangle$ and the expectation over uniformly randomly chosen $\overline{\x}_{t+1} \in [\x_t, \x_{t+1}]$ is same as uniformly randomly chosen $\lambda \in [0,1]$ by definition of $\boldsymbol{\alpha}(\lambda)$; and $(b_2)$ follows from the application of the second fundamental theorem of calculus utilizing the fact that $\Fob(\cdot)$ is a continuous function. 

Finally, substituting the bound derived in \eqref{Eq:T11_Bound} in \eqref{eq: Bound_T1}, we get
\begin{align}
    T_1 \leq \frac{\delta_\y}{\gamma_1} + \Ebb \big[ \Fob(\x_{t}) - \Fob(\x_{t+1}) \big]
    \label{Eq: Bound_T1_Final}
\end{align}
Next, substituting the bounds derived in \eqref{eq: Bound_T2} and \eqref{Eq: Bound_T1_Final} in \eqref{Eq: Difference_mt}, we get
\begin{align*}
    & \Ebb [\eta_t   \| \m_{t+1}\|^2  - \beta^2 \eta_t \| \m_t\|^2] \\
     & \qquad \leq 2 \beta (1 - \beta) \frac{\delta_\y}{\gamma_1} + 2 \beta (1 - \beta) \Ebb \big[ \Fob(\x_{t}) - \Fob(\x_{t+1}) \big]   + \frac{3(1 - \beta)^2}{\gamma_2} (\delta_\y^2 + \delta_\vb^2 + \overline{L}_F^2),
\end{align*}
Summing over $t \in [T]$ and multiplying by $\frac{1}{T}$, we get 
\begin{align}
    & \frac{1}{T} \sum_{t = 1}^T  \Ebb [\eta_t   \| \m_{t+1}\|^2  - \beta^2 \eta_t \| \m_t\|^2]  \nonumber\\
    &  \qquad \overset{(a)}{\leq} 2 \beta (1 - \beta) \frac{\delta_\y}{\gamma_1} + 2 \beta (1 - \beta) \frac{\Delta}{T}  + \frac{3(1 - \beta)^2}{\gamma_2} (\delta_\y^2 + \delta_\vb^2 + \overline{L}_F^2),
    \label{eq: Difference_mt_sum}
\end{align}
where inequality $(a)$ follows from telescoping the sum $\Ebb \big[ \Fob(\x_{t}) - \Fob(\x_{t+1}) \big] $ over $t \in [T]$ and defining $\Delta \coloneqq \Fob(\x_1) - \Fob^\ast$ where $\Fob^\ast \triangleq \min_{\x \in \du} \Fob(\x)$.

Next, substituting the bound derived in Lemma \ref{lem: Term_T3} into \eqref{eq: Difference_mt_sum}, we get
\begin{align}
  \nonumber  & \frac{\beta - \beta^2}{2} 
  \frac{1}{T} \sum_{t = 1}^{T+1} \Ebb \bigg[ \frac{\| \m_t\|^2}{\gamma_1 \| \m_t\| + \gamma_2} \bigg]    -    \frac{3 \beta}{\gamma_2} \frac{1}{T} (\delta_\y^2 + \delta_\vb^2 + \overline{L}_F^2) \\
   \nonumber  &  \qquad \qquad \leq  2 \beta (1 - \beta) \frac{\delta_\y}{\gamma_1} + 2 \beta (1 - \beta) \frac{\Delta}{T}  + \frac{3(1 - \beta)^2}{\gamma_2} (\delta_\y^2 + \delta_\vb^2 + \overline{L}_F^2)
\end{align}
Rearranging the terms and multiplying both sides by $\frac{2 \gamma_2}{\beta - \beta^2}$, we get
\begin{align}
  \nonumber  &   
  \frac{1}{T} \sum_{t = 1}^{T} \Ebb \bigg[ \frac{ \gamma_2\| \m_t\|^2}{\gamma_1 \| \m_t\| + \gamma_2} \bigg]     \\
   \nonumber  &  \qquad  \leq  \frac{4 \gamma_2 \delta_\y}{\gamma_1} +   \frac{4 \gamma_2 \Delta}{T}  + \frac{6(1 - \beta)}{\beta} (\delta_\y^2 + \delta_\vb^2 + \overline{L}_F^2) +   \frac{6}{1 - \beta} \frac{1}{T} (\delta_\y^2 + \delta_\vb^2 + \overline{L}_F^2).
\end{align}
\begin{itemize}
   \item[\RT] Recall that $\delta_\y \coloneqq L_F\cdot \delta$ where $\delta$ defined in Assumption \ref{ass:Approx_y} is a system parameter. With some overloading of notation, here we redefine $\delta_\y$ to be
$$\delta_\y \! \coloneqq  \min \bigg\{ \! \frac{ \epsilon^2}{1280 (\delta_\vb + 2 \overline{L}_F) \!},  L_F\cdot \delta, \overline{L}_F\bigg\} \Rightarrow \delta_\y \! \leq \frac{ \epsilon^2}{1280 (\delta_\vb + 2 \overline{L}_F)} ~\text{and}~\delta_\y \! \leq \overline{L}_F.$$ 
    \end{itemize}
    \begin{align}
  \nonumber    
  \frac{1}{T} \sum_{t = 1}^{T} \Ebb \bigg[ \frac{ \gamma_2\| \m_t\|^2}{\gamma_1 \| \m_t\| + \gamma_2} \bigg]    \!   \leq \! \frac{4 \gamma_2 \delta_\y}{\gamma_1} +   \frac{4 \gamma_2 \Delta}{T}  + \frac{6(1 - \beta)}{\beta} (  \delta_\vb^2 + 2 \overline{L}_F^2) +   \frac{6}{1 - \beta} \frac{\delta_\vb^2 + 2\overline{L}_F^2}{T}.
\end{align}
Recall from Theorem \ref{Thm:Main_I} that we have $\gamma_1 = K/\overline{\delta}$, where $K$ is defined in Algorithm \ref{Algo: DS-BLO} as 
$$K =\frac{1}{\ln \frac{1}{\beta}} \ln \frac{32 (   \delta_\vb + 2 \overline{L}_F)  }{\epsilon} = \mathcal{O} \bigg(\frac{1}{\epsilon^2} \ln \frac{32 (   \delta_\vb + 2 \overline{L}_F)}{\epsilon} \bigg).$$
Moreover, from \eqref{eq: Bound_T3_Intermediate} in Lemma \ref{lem: Term_T3} we have $\gamma_2 =4  \gamma_1 (\delta_{\vb} + 2 \overline{L}_F)$.

    \begin{itemize}
    \item[\RT] Choosing $\beta = 1 - \frac{\epsilon^2}{960 (\delta_\vb^2 + 2\overline{L}_F^2)}$. 

 \item[\RT] Moreover, assuming $\epsilon^2 \leq 480 (  \delta_\vb^2 + 2\overline{L}_F^2)$ s.t. $ \frac{1}{2} \leq \beta \leq 1$.

 \item[\RT] Choosing the number of iterations $T$ large enough such that

\begin{align*}
T &\! =\! \mathcal{O}\bigg(\! \max \bigg\{\frac{1280 (\delta_{\vb} + 2 \overline{L}_F) \Delta}{ \overline{\delta} \cdot \epsilon^4} \ln \bigg( \frac{32 (   \delta_\vb +2\overline{L}_F)}{\epsilon} \bigg), \frac{ 460,800 (  \delta_\vb^2 + 2\overline{L}_F^2)^2} {\epsilon^4} \bigg\} \!\bigg)\\
& = \widetilde{\mathcal{O}} \bigg\{ \frac{1}{\overline{\delta} \cdot \epsilon^4} \max\big\{(\delta_\vb + 2 \overline{L}_F)  \Delta ,  (\delta_\vb^2 + 2\overline{L}_F^2)^2\big\}  \bigg\},
\end{align*}
where $\widetilde{\mathcal{O}}(\cdot)$ hides the log terms. 
\end{itemize}
The choice of parameters above implies that we have
\begin{align*}
    \frac{1}{T} \sum_{t = 1}^{T} \Ebb \bigg[ \frac{ \gamma_2\| \m_t\|^2}{\gamma_1 \| \m_t\| + \gamma_2} \bigg]      \leq \frac{\epsilon^2}{20} \leq \frac{\epsilon^2}{17} .
\end{align*}
From the choice of $\gamma_2 = 4 \gamma_1 ( \delta_\vb + 2\overline{L}_F)$ we get
\begin{align*}
    \frac{1}{T} \sum_{t = 1}^{T} \Ebb \bigg[ \frac{ 4   (  \delta_\vb + 2 \overline{L}_F)\| \m_t\|^2}{  \| \m_t\| + 4   (  \delta_\vb + 2\overline{L}_F)} \bigg]        \leq \frac{\epsilon^2}{17}. 
\end{align*}
Next, following the analysis of \cite[Equation (15)]{zhang2020complexity} and using $\epsilon \leq   \delta_\vb + 2 \overline{L}_F$ we deduce that 
\begin{align}
    \frac{1}{T} \sum_{t = 1}^T \Ebb \|\m_t\| \leq \frac{\epsilon}{4}.
    \label{Eq: Final_Bound_m}
\end{align}
Next, we show that $\frac{1}{T} \sum_{t = 1}^T \Ebb \|\m_t\| $ being small implies that a certain convex combination of the gradients of the implicit function $\Fob(\cdot)$ will also be small. We have from Lemma \ref{lem: GradvsAvg_m} that
\begin{align*}
 \frac{1}{T} \sum_{t = 1}^T   \Ebb \bigg\| \! \sum_{i = t - K + 1}^t \! \alpha_i    {\gr}  \Fob(\overline{\x}_i)   \bigg\|  & \leq  \frac{1}{(1 - \beta^K) T}   \sum_{t = 1}^T \!  \Ebb \| \m_{t} \|   \\
 & \qquad \qquad  + \frac{\beta^K}{1 - \beta^K} \big( \delta_\y + \delta_\vb + \overline{L}_F \big)       + \delta_\y,
\end{align*}
where $\alpha_i = \frac{ \beta^{t-i}(1 - \beta) }{1 - \beta^K}$. From \eqref{Eq: Final_Bound_m}, we have $\frac{1}{T}\sum_{t = 1}^T   \Ebb \| \m_{t} \| \leq \frac{\epsilon}{4}$, therefore, we have
\begin{align*}
 \frac{1}{T} \sum_{t = 1}^T   \Ebb \bigg\|  \sum_{i = t - K + 1}^t  \alpha_i    {\gr}  \Fob(\overline{\x}_i)   \bigg\|  & \leq  \frac{\epsilon}{4(1 - \beta^K)}     + \frac{\beta^K}{1 - \beta^K} \big(   \delta_\vb + 2 \overline{L}_F \big)       + \delta_\y.
\end{align*}
Note that the choice of $K$ and $\beta$ satisfy $\beta^K \big(   \delta_\vb + 2 \overline{L}_F \big)   \leq \frac{\epsilon}{32}$ which implies that $K \geq \frac{1}{\ln \frac{1}{\beta}} \ln \frac{32 \big(  \delta_\vb + 2\overline{L}_F \big)  }{\epsilon}$. Moreover, assuming that $\epsilon \leq \big(  \delta_\vb + 2\overline{L}_F \big) $ which implies that $\beta^K \leq \frac{1}{32}$. Substituting these bounds in the above, we get
\begin{align*}
 \frac{1}{T} \sum_{t = 1}^T   \Ebb \bigg\|  \sum_{i = t - K + 1}^t  \alpha_i    {\gr}  \Fob(\overline{\x}_i)   \bigg\|  \leq  \frac{8 \epsilon}{ 31}     + \frac{\epsilon}{31 }        + \delta_\y   \leq \frac{\epsilon}{3} +   \delta_\y .
\end{align*}
Finally, the choice of $\delta_\y $ from Algorithm \ref{Algo: DS-BLO} implies that we have
\begin{align*}
 \frac{1}{T} \sum_{t = 1}^T   \Ebb \bigg\|  \sum_{i = t - K + 1}^t  \alpha_i    {\gr}  \Fob(\overline{\x}_i)   \bigg\|  & \leq    \epsilon .
\end{align*}
Hence, the proof is complete. 
\end{proof}

\begin{lem} 
    \label{lem: Term_T3}
    For the iterates generated by Algorithm \ref{Algo: DS-BLO}, the following holds
    \begin{align*}
   &  \sum_{t = 1}^T  \Ebb [\eta_t   \| \m_{t+1}\|^2  - \beta^2 \eta_t \| \m_t\|^2]\\
    & \qquad \qquad \qquad \qquad \geq \frac{\beta - \beta^2}{2}   \sum_{t = 1}^{T+1} \Ebb \bigg[ \frac{\| \m_t\|^2}{\gamma_1 \| \m_t\| + \gamma_2} \bigg]    -    \frac{3 \beta}{\gamma_2} (\delta_\y^2 + \delta_\vb^2 + \overline{L}_F^2),
\end{align*}
where $\m_t$ and the rest of the parameters are defined in Algorithm \ref{Algo: DS-BLO}.
\end{lem}
\begin{proof}
 First, let us define the term on the left of the inequality as $T_3$, i.e.,
\begin{align}
  \label{eq: T3} T_3 \coloneqq  \sum_{t = 1}^T  \Ebb [\eta_t   \| \m_{t+1}\|^2  - \beta^2 \eta_t \| \m_t\|^2]
\end{align}
Now we lower bound $T_3$ in \eqref{eq: T3} above. Using the definition of $\eta_t$, we have
\begin{align}
 T_3 & \coloneqq   \sum_{t = 1}^T  \Ebb [\eta_t   \| \m_{t+1}\|^2  - \beta^2 \eta_t \| \m_t\|^2] \nonumber\\
 & = \sum_{t = 1}^T \Ebb \bigg[ \frac{\|\m_{t+1} \|^2}{\gamma_1 \| \m_t\| + \gamma_2} - \beta^2 \frac{\| \m_t\|^2}{\gamma_1 \| \m_t\| + \gamma_2} \bigg] \nonumber \\
 & = \sum_{t = 1}^T \Ebb \bigg[ \frac{\|\m_{t+1} \|^2}{\gamma_1 \| \m_t\| + \gamma_2}  -  \frac{\|\m_{t+1} \|^2}{\gamma_1 \| \m_{t+1}\| + \gamma_2}  \bigg]  \nonumber  \\
 & \qquad \qquad \qquad \qquad   + \sum_{t = 1}^T  \Ebb \bigg[\frac{\|\m_{t+1} \|^2}{\gamma_1 \| \m_{t+1}\| + \gamma_2}     - \beta^2 \frac{\| \m_t\|^2}{\gamma_1 \| \m_t\| + \gamma_2} \bigg]      \nonumber \\
 & \overset{(a)}{\geq} \sum_{t=1}^T \Ebb \bigg[ \frac{\gamma_1 \| \m_{t+1}\|^2 (\|\m_{t+1}\| - \|\m_t \|)}{(\gamma_1 \| \m_t\| + \gamma_2)(\gamma_1 \| \m_{t+1}\| + \gamma_2)} \bigg]  \nonumber\\
 & \qquad \qquad -  \Ebb \bigg[ \beta^2 \frac{\| \m_1\|^2}{\gamma_1 \| \m_1\| + \gamma_2}  \bigg] +  (1- \beta^2) \sum_{t = 2}^{T+1} \Ebb \bigg[ \frac{\| \m_t\|^2}{\gamma_1 \| \m_t\| + \gamma_2} \bigg]  .
 \label{eq: Bound_T3}
\end{align}
  where $(a)$ follows from subtracting the term $\beta^2 \Ebb \bigg[ \frac{\| \m_{T+1}\|^2}{\gamma_1 \| \m_{T+1}\| + \gamma_2} \bigg]$ from the second summation on the right-hand side of the inequality and separating the term corresponding to $t = 1$.

Next, we consider the term
\begin{align}
 T_{31} & \coloneqq  \frac{\gamma_1 \| \m_{t+1}\|^2 (\|\m_{t+1}\| - \|\m_t \|)}{(\gamma_1 \| \m_t\| + \gamma_2)(\gamma_1 \| \m_{t+1}\| + \gamma_2)}.
 \label{eq: T31}
\end{align}
First, consider the case when $\|\m_{t+1} \| \leq \|\m_t\|$. Then, we have
\begin{align*}
  \nonumber  & \| \m_{t}\| - \| \m_{t + 1}\| \overset{(a)}{\leq} \| \m_t - \m_{t + 1}\| \\
    & \qquad \qquad  \overset{(b)}{=} (1 - \beta) \| \m_t - \g_{t+1}\| \overset{(c)}{\leq} (1 - \beta) \big( \| \m_t \| + \| \g_{t+1} \| \big),
 \end{align*}
 where $(a)$ follows from the reverse triangle inequality; $(b)$ results from the definition of $\m_{t+1}$ in Algorithm \ref{Algo: DS-BLO}; and $(c)$ uses the triangle inequality. Substituting in \eqref{eq: T31}, we get 
 \begin{align}
   \nonumber  T_{31} & \coloneqq  \frac{\gamma_1 \| \m_{t+1}\|^2 (\|\m_{t+1}\| - \|\m_t \|)}{(\gamma_1 \| \m_t\| + \gamma_2)(\gamma_1 \| \m_{t+1}\| + \gamma_2)}\\
   \nonumber  & \geq \frac{- (1- \beta) \gamma_1 \| \m_{t+1}\|^2 \big( \| \m_t \| + \| \g_{t+1} \| \big)}{(\gamma_1 \| \m_t\| + \gamma_2)(\gamma_1 \| \m_{t+1}\| + \gamma_2)} \\
   & \geq - (1- \beta) \frac{\|\m_{t+1}\|^2}{\gamma_1 \| \m_{t+1}\| + \gamma_2} - (1 - \beta) \frac{\gamma_1 \| \g_{t+1} \|}{\gamma_2} \frac{\|\m_{t}\|^2}{\gamma_1 \| \m_{t}\| + \gamma_2}  
   \label{eq: Bound_T31}
 \end{align}
 where the last inequality results from the fact that $x \to \frac{x^2}{\gamma_1 x + \gamma_2}$ is increasing on $\mathbb{R}^+$ and $\| \m_{t+1}\| \leq \| \m_t\|$, $\frac{\gamma_1 \|\m_t \|}{\gamma_1 \|\m_t\| + \gamma_2} \leq 1$, and $\gamma_1 \| \m_t\| + \gamma_2 \geq \gamma_2$.
 Second, for $\| \m_{t+1} \| \geq \|\m_t \|$, the same inequality trivially holds since the right-hand side of \eqref{eq: Bound_T31} is negative. Taking expectation on both sides of $T_{31}$ in \eqref{eq: T31} and using \eqref{eq: Bound_T31}, we get
 \begin{align}
  \nonumber \Ebb \bigg[  \frac{\gamma_1 \| \m_{t+1}\|^2 (\|\m_{t+1}\| - \|\m_t \|)}{(\gamma_1 \| \m_t\| + \gamma_2)(\gamma_1 \| \m_{t+1}\| + \gamma_2)} \bigg] & \geq - (1- \beta)  \Ebb \bigg[ \frac{\|\m_{t+1}\|^2}{\gamma_1 \| \m_{t+1}\| + \gamma_2} \bigg]  \\
   &  \qquad    - \frac{(1 - \beta) \gamma_1}{\gamma_2}  \Ebb \bigg[ \| \g_{t+1} \| \frac{\|\m_{t}\|^2}{\gamma_1 \| \m_{t}\| + \gamma_2} \bigg] \nonumber\\
   & \!\!\!\!\!\!\!\!\!\! \!\!\!\!\!\!\!\!\!\! \!\!\!\!\!\!\!\!\!\! \!\!\!\!\!\!\!\!\!\! \!\!\!\!\!\!\!\!\!\! \!\!\!\!\!\!\!\!\!\! \!\!\!\!\!\!\!\!\!\! =    - (1- \beta)  \Ebb \bigg[ \frac{\|\m_{t+1}\|^2}{\gamma_1 \| \m_{t+1}\| + \gamma_2} \bigg]       - \frac{(1 - \beta) \gamma_1}{\gamma_2}  \Ebb \bigg[ \Ebb \big[ \| \g_{t+1} \big| \mathcal{G}_t \| \big]\frac{\|\m_{t}\|^2}{\gamma_1 \| \m_{t}\| + \gamma_2} \bigg],
   \label{eq:  T31_Term2}
 \end{align}
where the last equality follows from the definition $\mathcal{G}_t \coloneqq \sigma(\g_1, \ldots, \g_t)$. Let us now consider the term 
 \begin{align}
   \Ebb \big[ \| \g_{t+1} \| \big|  \mathcal{G}_{t} \big] & =   \Ebb \big[ \|  \widehat{\gr}F_{q_{t+1}} (\overline{\x}_{t+1}) \| \big|\mathcal{G}_{t} \big]  \nonumber\\
  \nonumber & \overset{(a)}{\leq}  \Ebb \big[ \|  \widehat{\gr}F_{q_{t+1}} (\overline{\x}_{t+1}) -   {\gr}F_{q_{t+1}} (\overline{\x}_{t+1}) \| \big|\mathcal{G}_{t} \big] \\
  \nonumber  &\qquad \qquad + \Ebb  \big[ \| {\gr}F_{q_{t+1}} (\overline{\x}_{t+1})   -   {\gr}\Fob (\overline{\x}_{t+1}) \| \big|\mathcal{G}_{t} \big] + \Ebb  \big[ \|  {\gr}\Fob (\overline{\x}_{t+1})  \| \big|\mathcal{G}_{t} \big] \\
  & \overset{(b)}{\leq} \delta_\y + \delta_\vb + \overline{L}_F,
  \label{eq: Bound_g}
 \end{align}
 where $(a)$ follows from the triangle inequality, and $(b)$ from Lemmas \ref{lem:F} and \ref{lem:F_Variance}. 
 
 Next, utilizing \eqref{eq: Bound_g} to bound the second term in \eqref{eq:  T31_Term2}, we have
 \begin{align}
 \nonumber & \frac{(1 - \beta) \gamma_1}{\gamma_2}   \Ebb \bigg[ \| \g_{t+1} \| \frac{\|\m_{t}\|^2}{\gamma_1 \| \m_{t}\| + \gamma_2} \bigg] \\
  \nonumber & \qquad \qquad \qquad \leq \frac{(1 - \beta) \gamma_1}{\gamma_2} (\delta_\y + \delta_\vb + \overline{L}_F)  \Ebb \bigg[  \frac{\|\m_{t}\|^2}{\gamma_1 \| \m_{t}\| + \gamma_2} \bigg] \\
   & \qquad \qquad \qquad \leq \frac{\beta (1 - \beta)}{ 2}    \Ebb \bigg[  \frac{\|\m_{t}\|^2}{\gamma_1 \| \m_{t}\| + \gamma_2} \bigg]
   \label{eq: Bound_T3_Intermediate}
 \end{align}
 where the last inequality follows from $\delta_\y \leq \overline{L}_F$ and the choice of $\gamma_1 = \frac{\gamma_2}{4 ( \delta_\vb + 2 \overline{L}_F)}$ in Algorithm \ref{Algo: DS-BLO} which ensures that we have $\frac{\gamma_1 (\delta_\vb + 2 \overline{L}_F)}{\gamma_2} \leq \frac{\beta}{2}$ for the choice of $\beta \geq \frac{1}{2}$. Next, substituting \eqref{eq: Bound_T3_Intermediate} in \eqref{eq:  T31_Term2}, we get
 
 \begin{align}
\nonumber \Ebb \bigg[  \frac{\gamma_1 \| \m_{t+1}\|^2 (\|\m_{t+1}\| - \|\m_t \|)}{(\gamma_1 \| \m_t\| + \gamma_2)(\gamma_1 \| \m_{t+1}\| + \gamma_2)} \bigg] & \geq - (1- \beta)  \Ebb \bigg[ \frac{\|\m_{t+1}\|^2}{\gamma_1 \| \m_{t+1}\| + \gamma_2} \bigg] \\
 &  \qquad \quad  -  \frac{\beta (1 - \beta)}{ 2}    \Ebb \bigg[  \frac{\|\m_{t}\|^2}{\gamma_1 \| \m_{t}\| + \gamma_2} \bigg]
 \label{eq: Bound_T31_Final}
 \end{align}
 Next, substituting \eqref{eq: Bound_T31_Final} in \eqref{eq: Bound_T3}, we get
 \begin{align}
 \nonumber T_3 & \coloneqq   \sum_{t = 1}^T  \Ebb [\eta_t   \| \m_{t+1}\|^2  - \beta^2 \eta_t \| \m_t\|^2]  \\
 \nonumber & \geq  - (1- \beta) \sum_{t=1}^T  \Ebb \bigg[ \frac{\|\m_{t+1}\|^2}{\gamma_1 \| \m_{t+1}\| + \gamma_2} \bigg]  -  \frac{\beta (1 - \beta)}{ 2}   \sum_{t=1}^T  \Ebb \bigg[  \frac{\|\m_{t}\|^2}{\gamma_1 \| \m_{t}\| + \gamma_2} \bigg] \\
 \nonumber &  \qquad \qquad \qquad \qquad -  \beta^2  \Ebb \bigg[ \frac{\| \m_1\|^2}{\gamma_1 \| \m_1\| + \gamma_2}  \bigg] +  (1- \beta^2) \sum_{t = 2}^{T+1} \Ebb \bigg[ \frac{\| \m_t\|^2}{\gamma_1 \| \m_t\| + \gamma_2} \bigg] 
 \end{align}
 Rearranging and combining the terms we get 
 \begin{align}
 \nonumber T_3  & \geq (\beta - \beta^2)   \sum_{t = 2}^{T+1} \Ebb \bigg[ \frac{\| \m_t\|^2}{\gamma_1 \| \m_t\| + \gamma_2} \bigg] - \frac{\beta (1 - \beta)}{ 2}   \sum_{t=1}^T  \Ebb \bigg[  \frac{\|\m_{t}\|^2}{\gamma_1 \| \m_{t}\| + \gamma_2} \bigg] \\
\nonumber & \qquad \qquad \qquad \qquad \qquad \qquad \qquad \qquad \qquad \qquad -  \beta^2  \Ebb \bigg[ \frac{\| \m_1\|^2}{\gamma_1 \| \m_1\| + \gamma_2}  \bigg] \\
\nonumber & = (\beta - \beta^2)   \sum_{t = 2}^{T+1} \Ebb \bigg[ \frac{\| \m_t\|^2}{\gamma_1 \| \m_t\| + \gamma_2} \bigg] - \frac{ \beta - \beta^2}{ 2}   \sum_{t=2}^T  \Ebb \bigg[  \frac{\|\m_{t}\|^2}{\gamma_1 \| \m_{t}\| + \gamma_2} \bigg] \\
\nonumber & \qquad \qquad \qquad \qquad - \frac{ \beta - \beta^2}{ 2} \Ebb \bigg[ \frac{\| \m_1\|^2}{\gamma_1 \| \m_1\| + \gamma_2}  \bigg]  -  \beta^2  \Ebb \bigg[ \frac{\| \m_1\|^2}{\gamma_1 \| \m_1\| + \gamma_2}  \bigg] \\
& \nonumber \overset{(a)}{\geq}  \frac{\beta - \beta^2}{2}   \sum_{t = 2}^{T+1} \Ebb \bigg[ \frac{\| \m_t\|^2}{\gamma_1 \| \m_t\| + \gamma_2} \bigg]    - \frac{ \beta + \beta^2}{ 2} \Ebb \bigg[ \frac{\| \m_1\|^2}{\gamma_1 \| \m_1\| + \gamma_2}  \bigg]  \\
& \overset{(b)}{=}  \frac{\beta - \beta^2}{2}   \sum_{t = 1}^{T+1} \Ebb \bigg[ \frac{\| \m_t\|^2}{\gamma_1 \| \m_t\| + \gamma_2} \bigg]    -    \beta  \Ebb \bigg[ \frac{\| \m_1\|^2}{\gamma_1 \| \m_1\| + \gamma_2}  \bigg]  
 \end{align}
where $(a)$ follows from subtracting the term $\Ebb \bigg[ \frac{\| \m_{T+1}\|^2}{\gamma_1 \| \m_{T+1}\| + \gamma_2} \bigg]$ from the right-hand side of the inequality and combining the terms; and $(b)$ follows from adding and subtracting the term $\frac{\beta - \beta^2}{2}  \Ebb \bigg[ \frac{\| \m_1\|^2}{\gamma_1 \| \m_1\| + \gamma_2} \bigg] $ on the right-hand side and combining the terms. 
 
Next, using the fact that $\frac{1}{\gamma_1 \|\m_1\| + \gamma_2} \leq \frac{1}{\gamma_2}$, we get
\begin{align}
     T_3 \geq  \frac{\beta - \beta^2}{2}   \sum_{t = 1}^{T+1} \Ebb \bigg[ \frac{\| \m_t\|^2}{\gamma_1 \| \m_t\| + \gamma_2} \bigg]    -    \frac{\beta}{\gamma_2}  \Ebb \| \m_1\|^2
 \label{eq: Bound_T3_Penultimate}
 \end{align}
Finally, using the fact that $\m_1 = \g_1 = \widehat{\gr} F_{q_1} (\x_{1})$, we have
\begin{align}
 \nonumber    \Ebb [\| \m_1\|^2] & = \Ebb [\| \g_1\|^2] = \Ebb [\| \widehat{\gr} F_{q_1} (\x_{1})\|^2] \\
 \nonumber   & \overset{(a)}{\leq} 3 \Ebb [\| \widehat{\gr} F_{q_1} (\x_{1}) - {\gr} F_{q_1} (\x_{1})\|^2] \\
 \nonumber & \qquad \qquad + 3 \Ebb [\|   {\gr} F_{q_1} (\x_{1}) -  {\gr} \Fob (\x_{1})\|^2] + 3 \Ebb [\|     {\gr} \Fob (\x_{1})\|^2]  \\
 & \overset{(b)}{\leq} 3 (\delta_\y^2 + \delta_\vb^2 + \overline{L}_F^2),
 \label{eq: Bound_m1}
\end{align}
where $(a)$ utilizes $\| \sum_{k = 1}^K  \ab_k \|^2 \leq K \sum_{k =  1}^K \|\ab_k \|^2$ for a set of vectors $\{ \ab_k\}_{k =1}^K \in \du $; and $(b)$ results from the application of Lemmas \ref{lem:F} and \ref{lem:F_Variance}.

Substituting \eqref{eq: Bound_m1} in \eqref{eq: Bound_T3_Penultimate}, we get
\begin{align*}
    T_3 \geq \frac{\beta - \beta^2}{2}   \sum_{t = 1}^{T+1} \Ebb \bigg[ \frac{\| \m_t\|^2}{\gamma_1 \| \m_t\| + \gamma_2} \bigg]    -    \frac{3 {\beta}}{\gamma_2} (\delta_\y^2 + \delta_\vb^2 + \overline{L}_F^2).
\end{align*}
Hence, the proof is complete. 
\end{proof}

\begin{lem}
For the iterates generated by Algorithm \ref{Algo: DS-BLO}, the following holds
    \label{lem: GradvsAvg_m}
    \begin{align*}
 \frac{1}{T} \sum_{t = 1}^T   \Ebb \bigg\| \! \sum_{i = t - K + 1}^t \! \alpha_i    {\gr}  \Fob(\overline{\x}_i)   \bigg\|   \leq  \frac{1}{(1 - \beta^K) T} \!  \sum_{t = 1}^T   \Ebb \| \m_{t} \|   + \frac{\beta^K}{1 - \beta^K} \big( \delta_\y + \delta_\vb + \overline{L}_F \big)       + \delta_\y ,   
\end{align*}
where $\alpha_i = \frac{ \beta^{t-i}(1 - \beta) }{1 - \beta^K}$ and $K$ is defined in Algorithm \ref{Algo: DS-BLO}. 
\end{lem}
\begin{proof}
    First, let us observe from the definition of $\m_{t}$ that, we have
\begin{align*}
    \m_{t} = \beta^K \m_{t - K} + (1 - \beta) \sum_{i = t - K + 1}^t \beta^{t-i} \g_i.
\end{align*}
where $K$ is defined in Algorithm \ref{Algo: DS-BLO}. Rearranging the terms and multiplying both sides by $\frac{1}{1 - \beta^K}$, we have
\begin{align*}
 \frac{1}{1 - \beta^K} \sum_{i = t - K + 1}^t  (1 - \beta) \beta^{t-i} \g_i =     \frac{\beta^K}{1 - \beta^K}  \m_{t - K} -  \frac{1}{1 - \beta^K} \m_{t} 
\end{align*}
Using the definition of $\g_i \coloneqq \widehat{\gr} F_{q_i}(\overline{\x}_i)$, we get
\begin{align*}
  & \sum_{i = t - K + 1}^t  \alpha_i \Big[  \widehat{\gr}  F_{q_i}(\overline{\x}_i) -  {\gr}  F_{q_i}(\overline{\x}_i) +  {\gr}  F_{q_i}(\overline{\x}_i) - {\gr}  \Fob(\overline{\x}_i) + {\gr}  \Fob(\overline{\x}_i) \Big] \\
   & \qquad \qquad \qquad \qquad \qquad \qquad \qquad \qquad \qquad =     \frac{\beta^K}{1 - \beta^K}  \m_{t - K} -  \frac{1}{1 - \beta^K} \m_{t}, 
\end{align*}
where we have defined $\alpha_i = \frac{ \beta^{t-i}(1 - \beta) }{1 - \beta^K}$. Rearranging the terms, we get 
\begin{align*}
  & \sum_{i = t - K + 1}^t  \alpha_i    {\gr}  \Fob(\overline{\x}_i)   =     \frac{\beta^K}{1 - \beta^K}  \m_{t - K} -  \frac{1}{1 - \beta^K} \m_{t}   \\
    & \qquad - \sum_{i = t - K + 1}^t  \alpha_i \big[  \widehat{\gr}  F_{q_i}(\overline{\x}_i) -  {\gr}  F_{q_i}(\overline{\x}_i) \big] - \sum_{i = t - K + 1}^t  \alpha_i \big[   {\gr}  F_{q_i}(\overline{\x}_i) - {\gr}  \Fob(\overline{\x}_i) \big], 
\end{align*}
Taking expectation w.r.t. $\qb_i \sim \mathcal{Q}$ for $i \in [t - K +1, t]$ on both sides, we get
\begin{align*}
  & \sum_{i = t - K + 1}^t  \alpha_i    {\gr}  \Fob(\overline{\x}_i)   =     \frac{\beta^K}{1 - \beta^K} \Ebb_{\{ \qb_i\}_{t - K +1}^t} [ \m_{t - K}] -  \frac{1}{1 - \beta^K} \Ebb_{\{ \qb_i\}_{t - K +1}^t}[ \m_{t} ] \\
  & \qquad\qquad\qquad\qquad \qquad\qquad \qquad - \sum_{i = t - K + 1}^t  \alpha_i  \Ebb_{\qb_i}\big[  \widehat{\gr}  F_{q_i}(\overline{\x}_i) -  {\gr}  F_{q_i}(\overline{\x}_i) \big], 
\end{align*}
where we have used the unbiasedness of ${\gr}  F_{q_i}(\overline{\x}_i)$ established in Proposition \ref{pro:diff}. 

Taking norm on both sides and utilizing Jensen's and triangle inequality we get
\begin{align*}
   \bigg\|  \sum_{i = t - K + 1}^t  \alpha_i    {\gr}  \Fob(\overline{\x}_i)   \bigg\| & \leq     \frac{\beta^K}{1 - \beta^K} \Ebb_{\{ \qb_i\}_{t - K +1}^t}\| \m_{t - K} \|  +  \frac{1}{1 - \beta^K} \Ebb_{\{ \qb_i\}_{t - K +1}^t}\| \m_{t} \|  \\
    & \qquad \qquad \qquad     + \sum_{i = t - K + 1}^t  \alpha_i \Ebb_{\qb_i}\big\|  \widehat{\gr}  F_{q_i}(\overline{\x}_i) -  {\gr}  F_{q_i}(\overline{\x}_i) \big\| ,
\end{align*}
Next, taking full expectations on both sides, we get
\begin{align*}
  \Ebb \bigg\|  \sum_{i = t - K + 1}^t  \alpha_i    {\gr}  \Fob(\overline{\x}_i)   \bigg\|  ~&  \leq    ~ \frac{\beta^K}{1 - \beta^K} \Ebb \| \m_{t - K} \|  +  \frac{1}{1 - \beta^K} \Ebb \| \m_{t} \|  \\
    & \qquad  \qquad   + \sum_{i = t - K + 1}^t  \alpha_i \Ebb \big\|  \widehat{\gr}  F_{q_i}(\overline{\x}_i) -  {\gr}  F_{q_i}(\overline{\x}_i) \big\|  , 
\end{align*}
 Next, using Lemma \ref{lem: Induction}, we have $\Ebb \| \m_{t-K}\| \leq \delta_\y + \delta_\vb + \overline{L}_F$. Substituting in the above, we get
\begin{align*}
   \Ebb \bigg\|  \sum_{i = t - K + 1}^t  \alpha_i    {\gr}  \Fob(\overline{\x}_i)   \bigg\|  ~& \leq    ~ \frac{\beta^K}{1 - \beta^K} \big( \delta_\y + \delta_\vb + \overline{L}_F \big)  +  \frac{1}{1 - \beta^K} \Ebb \| \m_{t} \|  \\
    & \qquad  \qquad  \qquad    + \sum_{i = t - K + 1}^t  \alpha_i \Ebb \big\|  \widehat{\gr}  F_{q_i}(\overline{\x}_i) -  {\gr}  F_{q_i}(\overline{\x}_i) \big\|  , \\
    &  \overset{(a)}{\leq} \frac{\beta^K}{1 - \beta^K} \big( \delta_\y + \delta_\vb + \overline{L}_F \big)  +  \frac{1}{1 - \beta^K} \Ebb \| \m_{t} \|       + \delta_\y ,
\end{align*}
where $(a)$ follows from Lemmas \ref{lem:F} and the fact that $\sum_{i = t- K + 1}^t \alpha_i = 1$. Next, summing over $t \in [T]$ and multiplying both sides by $\frac{1}{T}$, we get
\begin{align*}
 \frac{1}{T} \sum_{t = 1}^T   \Ebb \bigg\|  \sum_{i = t - K + 1}^t  \alpha_i    {\gr}  \Fob(\overline{\x}_i)   \bigg\|  \!  \leq \!  \frac{1}{(1 - \beta^K) T} \!  \sum_{t = 1}^T   \Ebb \| \m_{t} \|   + \frac{\beta^K}{1 - \beta^K} \big( \delta_\y + \delta_\vb + \overline{L}_F \big)       + \delta_\y.
\end{align*}
Hence, the proof is complete. 
\end{proof}

\begin{lem}
    \label{lem: Induction}
 For $\m_t$ defined in Algorithm \ref{Algo: DS-BLO}, we have $\Ebb \| \m_t\| \leq \delta_\y + \delta_\vb + \overline{L}_F$ for all $t \in [T]$. 
\end{lem}
\begin{proof}
    We prove the above statement using induction. 
    \begin{itemize}
        \item[\RT] {\bf Base case.} For $t = 1$, we know from the definition of $\m_1$ that
    \begin{align*}
        \Ebb \| \m_1 \| = \Ebb \| \g_1 \| \overset{(a)}{\leq} \delta_\y + \delta_{\vb} + \overline{L}_F,
    \end{align*}
    where inequality $(a)$ follows from \eqref{eq: Bound_g}. 
    
    Therefore, the base case for $t = 1$ holds. 
    \item[\RT] {\bf Induction hypothesis.} We assume that $\Ebb \| \m_t\| \leq \delta_\y + \delta_{\vb} + \overline{L}_F$ for some $t \in [T]$.
    \item[\RT]  {\bf Inductive step.} Next, we show that $\Ebb \| \m_{t+1}\| \leq \delta_\y + \delta_{\vb} + \overline{L}_F$.

    Note that from the definition of $\m_{t+1}$ in Algorithm \ref{Algo: DS-BLO}, we have
    \begin{align*}
        \Ebb \| \m_{t+1}\| & = \Ebb \| \beta \m_t + (1 - \beta) \g_{t+1} \| \\
        & \overset{(a)}{\leq} \beta ~\Ebb \| \m_t \| + (1 - \beta)~ \Ebb \| \g_{t+1} \| \\
        & \overset{(b)}{\leq} \beta~(\delta_\y + \delta_{\vb} + \overline{L}_F) + (1 - \beta)~(\delta_\y + \delta_{\vb} + \overline{L}_F) \\
        & = \delta_\y + \delta_{\vb} + \overline{L}_F,
    \end{align*}
where $(a)$ follows from the application of Triangle inequality and $(b)$ utilizes the induction hypothesis along with \eqref{eq: Bound_g}.
    \end{itemize}
    Hence, the proof is complete.  
\end{proof}

\end{document}